\documentclass[12pt]{amsart}
\theoremstyle{plain}
\usepackage{fullpage, url, amssymb, bm}
\usepackage[all]{xy} 

\usepackage{xcolor}
\usepackage{tikz-cd}
\DeclareFontFamily{U}{wncy}{}
\DeclareFontShape{U}{wncy}{m}{n}{<->wncyr10}{}
\DeclareSymbolFont{mcy}{U}{wncy}{m}{n}
\DeclareMathSymbol{\Sh}{\mathord}{mcy}{"58} 
\usepackage{calligra}
\usepackage[T1]{fontenc}

\usepackage {mathpazo, bbold}
\usepackage{aurical}
\usepackage[T1]{fontenc}
\usepackage{libertinus}

\usepackage[utf8]{inputenc}
\usepackage{amsmath}
\usepackage{amssymb}
\usepackage{enumitem, graphicx}

\usepackage{amsthm}

\DeclareFontEncoding{OT2}{}{} 

\newcommand{\clD}{{\mathcal D}}

\newcommand{\clG}{{\mathcal G}}
\newcommand{\clH}{{\mathcal H}}
\newcommand{\clI}{{\mathcal I}}

\newcommand{\clO}{{\mathcal O}}
\newcommand{\clP}{{\mathcal P}}

\newcommand{\clR}{{\mathcal R}}
\newcommand{\clS}{{\hbox{\scalebox{.8}[.8]{\Fontauri\slshape{S}}}}}

\newcommand{\clV}{{\mathcal V}}
\newcommand{\clW}{{\mathcal W}}

\newcommand{\clZ}{{\mathcal Z}}

\newcommand{\lvP}{{\mathbb P}}
\newcommand{\lvQ}{{\mathbb Q}}

\newcommand{\lvZ}{{\mathbb Z}}

\newcommand{\eum}{{\mathfrak m}}

\newcommand{\euv}{{\mathfrak v}}

\newcommand{\euw}{{\mathfrak w}}

\DeclareMathOperator{\Aut}{Aut}

\DeclareMathOperator{\chr}{char}

\DeclareMathOperator{\Hom}{Hom}

\DeclareMathOperator{\im}{im}
\DeclareMathOperator{\Krdim}{Kr.\hb1dim}

\DeclareMathOperator{\rk}{rk}

\DeclareMathOperator{\Spec}{Spec}

\DeclareMathOperator{\Val}{Val}

\DeclareRobustCommand\longtwoheadrightarrow
     {\relbar\joinrel\twoheadrightarrow}

\newcommand{\hla}{\hookleftarrow}
\newcommand{\hra}{\hookrightarrow}

\newcommand{\rsdp}{\rtimes}

\newcommand{\srjr}{\longtwoheadrightarrow}

\newcommand{\dwnr}[2]{{\ha{#1}\downarrow\hb2\lower-5pt\hbox{$\scriptstyle #2$}}}
\newcommand{\dwnl}[2]{{\lower-5pt\hbox{${\scriptstyle #2}$}\hb2\downarrow\ha{#1}}}
\newcommand{\dwn}[1]{\ha1\downarrow\rlap{$\hb2\vcenter{\hbox{$\scriptstyle{#1}$}}$}}
\newcommand{\Ddwn}[1]{\ha2\ddwn\ha4\rlap{$\hb2\vcenter{\hbox{$\scriptstyle{#1}$}}$}}

\newcommand{\hhb}[1]{\hbox to#1pt{}}
\newcommand{\hor}[1]{\smash
       {\mathop{{\lgrghtar}}\limits^{\lower2pt\hbox{$\scriptstyle #1$}}}}
\newcommand{\horsrjl}[1]{\smash
       {\mathop{{\longleftarrow}}\limits^{\lower2pt\hbox{$\scriptstyle #1$}}}}
\newcommand{\horsrjr}[1]{\smash
       {\mathop{{\lgsrjr}}\hb{23}\lower-8pt\hbox{$\scriptstyle#1$}}\ha9}    
\newcommand{\horr}[1]{\ha2\smash
      {\mathop{{\lglgrghtar}}\limits^{\lower2pt\hbox{$\scriptstyle #1$}}}\ha2}
\newcommand{\hrr}[1]{\ha2\smash
      {\mathop{\to}\limits^{\lower2pt\hbox{$\scriptstyle #1$}}}\ha2}
\newcommand{\mpst}[1]{\,\smash
       {\mathop{{\mapsto}}\limits^{\lower2pt\hbox{$\scriptstyle #1$}}}\,}
\newcommand{\mpstt}[1]{\,\smash
       {\mathop{{\longmapsto}}\limits^{\lower2pt\hbox{$\scriptstyle #1$}}}\,}

\newcommand{\lgrghtar}{{\ha2{\relbar\joinrel\rightarrow}\ha2}}
\newcommand{\lglgrghtar}{{\ha1{\relbar\joinrel\relbar\joinrel\rightarrow}\ha1}}
\newcommand{\lgsrjr}{{\ha1{\relbar\joinrel\srjr}\ha1}}

\newcommand{\nmnm}[1]{{\scalebox{.85}[1.05]{\sc #1}}}

\newcommand{\plim}[1]{\hbox to14pt{\rm
lim\kern-17pt\lower4.5pt\hbox{$\scriptstyle\longleftarrow$}%
\kern-11pt\lower6.5pt\hbox{$\scriptscriptstyle{#1}$}}\ha{3}}

\newcommand{\ilim}[1]{\hbox to14pt{\rm
lim\kern-14pt\lower4.5pt\hbox{$\scriptstyle\longrightarrow$}%
\kern-9pt\lower6.5pt\hbox{$\scriptscriptstyle{#1}$}}\ha{3}}

\newcommand{\hb}[1]{\hbox to-#1pt{}}
\newcommand{\ha}[1]{\hbox to#1pt{}}

\newcommand{\abc}{abelian-by-central}
\newcommand{\abcZ}{$\lvZ/\ell$ \abc}

\newcommand{\bgg}[1]{\lower.5pt\hbox{{\scalebox{1}[1.2]{$\!#1\!$}}}}
\newcommand{\bltt}{\scalebox{1.2}[1.2]{$\bullet$}} 

\newcommand{\clp}[2]{{#1\!\leqslant\!#2}}
\newcommand{\coprodmic}{{\lower-2pt\hbox{\scalebox{.7}[.6]{$\ha1\coprod\ha1$}}}}

\newcommand{\ddwn}{{\lower-1.25pt\hbox{$\downarrow$}\hb{6.55}\lower1.25pt\hbox{$\downarrow$}}^{\phantom;}_\nix}

\newcommand{\EE}{E}

\newcommand{\FF}{F}

\newcommand{\FFw}{{F\hb{.5}w}}

\newcommand{\Gal}{{G}}

\newcommand{\GGa}[1]{\clG^{\ha{.5}a}_{\!#1}}
\newcommand{\GGc}[1]{\clG^{\ha{.5}c}_{\!#1}}

\newcommand{\GGac}{\Delta}

\newcommand{\gx}{g}
\newcommand{\gy}{h}

\newcommand{\itm}[2]{{\begin{itemize}[leftmargin=#1pt]#2\end{itemize}}}

\newcommand{\ii}{{\alpha}}
\newcommand{\ia}{{\alpha}}

\newcommand{\jb}{{\beta}}
\newcommand{\jj}{{\beta}}

\newcommand{\LL}{L}
\newcommand{\Lm}{\Lambda}

\newcommand{\NN}{N}

\newcommand{\nix}{{\phantom|}}

\newcommand{\Ox}{{\scalebox{.8}[1]{$\hb1\Omega\hb{.5}$}}}
\newcommand{\Oy}{{\scalebox{.8}[1]{$\hb1\Lambda\hb1$}}}
\newcommand{\Om}{\Omega}
\newcommand{\oli}{\overline}

\newcommand{\pml}{\scalebox{.9}[1.2]{{\rm(}}}
\newcommand{\pmr}{\scalebox{.9}[1.2]{{\rm)}}}

\newcommand{\pr}{{pr}}

\newcommand{\qx}{\rho}

\newcommand{\Spi}[2]{\clS(\pi_{\XKm{#1}{#2}})}

\newcommand{\St}{{\rm St}}
\newcommand{\scl}[2]{{\scalebox{.#1}[.#1]{#2}}} 
\newcommand{\sclx}[3]{{\scalebox{.#1}[.#2]{\sc #3}}} 
 
\newcommand{\sep}{{^{\rm sep}}}

\newcommand{\sbgg}[1]{{\langle #1\rangle}}

\newcommand{\tl}{\tilde }

\newcommand{\tlk}{{\tilde k}}
\newcommand{\tlK}{{\tilde K}}

\newcommand{\tlsgm}{\sigma^c}
\newcommand{\tltau}{\tau^c}
\newcommand{\tlclS}{{\clS\hb2\lower-3.25pt\hbox{$\tilde{}$\ha2}}}
\newcommand{\tlv}{\tilde v}
\newcommand{\tlw}{\tilde w}
\newcommand{\tlwv}{\tilde w_{\tl\euv}}

\newcommand{\vid}{{\not\kern-3pt\lower-1pt\hbox{$\scriptstyle\bigcirc$}}}

\newcommand{\vrg}{\hbox{,}\,}

\newcommand{\vv}{v}

\newcommand{\vy}{\euv}
\newcommand{\vz}{\euw}
\newcommand{\vxy}{v}

\newcommand{\wh}[1]{{\widehat{#1}}}
\newcommand{\ww}{w}

\newcommand{\XK}[2]{{#1\lower.5pt\hbox{/}\lower1pt\hbox{$#2$}}}
\newcommand{\XKm}[2]{{#1\lower.5pt\hbox{\scl{75}/}\lower1pt\hbox{\scl{7}{$#2$}}}}

\newcommand{\ZZ}{{\clZ\hb{8}\clZ}}

\newtheorem{theorem}{Theorem}[section]
\newtheorem*{theorem*}{Theorem}

\newtheorem{keylemma}[theorem]{Key Lemma}
\newtheorem{lemma}[theorem]{Lemma}
\newtheorem{proposition}[theorem]{Proposition}

\theoremstyle{definition}

\newtheorem{construction}[theorem]{Construction}
\newtheorem{convention/notations}[theorem]{Convention/Notations}

\newtheorem{definition}[theorem]{Definition}
\newtheorem{definition/notations}[theorem]{Definition/Notations}
\newtheorem{definition/remark}[theorem]{Definition/Remark}
\newtheorem{definition/remarks}[theorem]{Definition/Remarks}
\newtheorem*{definition/remarks*}{Definition/Remarks}
\newtheorem{example}[theorem]{Example}

\newtheorem{example/fact}[theorem]{Example/Fact}
\newtheorem{fact}[theorem]{Fact}

\newtheorem{fact/definition}[theorem]{Fact/Definition}

\newtheorem{notation}[theorem]{Notation}
\newtheorem{notation/remark}[theorem]{Notations/Remark}
\newtheorem{notations/remark}[theorem]{Notations/Remark}
\newtheorem{notations/remarks}[theorem]{Notations/Remarks}

\newtheorem{remark/definition}[theorem]{Remark/Definition}
\newtheorem{remark/notation}[theorem]{Remark/Notation}
\newtheorem{remarks/examples}[theorem]{Remarks/Examples}
\newtheorem{remark}[theorem]{Remark}
\newtheorem{remarks}[theorem]{Remarks}
\begin{document}
\title[\textbf{Generalizations and minimalistic Refinements 
of the {\lowercase{$\bm t$}}-birational Section Conjecture}] 
{\textbf{Generalizations and minimalistic Refinements of the \\ 
{\lowercase{$\bm t$}}\ha1-\ha1birational Section Conjecture}}

\author{Florian Pop}

\address{Department of Mathematics, University of Pennsylvania
        \vskip0pt
        DRL, 209 S 33rd Street, Phila\-delphia, PA 19104, USA}
\email{pop@math.upenn.edu}
\urladdr{https://www.math.upenn.edu/\~{}pop/}

\begin{abstract} 
In this note we give generalizations and prove  
``minimalistic'' refinements  of the $t$-birational Section 
Conjecture ($t\hb1$-\ha1BSC\ha1), cf.~\cite{Be}, by doing 
both: First, by extending the class of base fields over 
which the $t$-\ha1BSC holds, and second, by proving 
refinements of the $t$-\ha1BSC which involve much less, 
that is {\it minimalistic,\/} Galois theoretical information.
\end{abstract}

\vskip4pt
\thanks{Supported by the NSF FRG DMS-2152304 and
Simons SFM 00021717.}
\date{Variant of April, 2027.} 

\maketitle

\section{Introduction/Motivation}
For reader's sake and to make the presentation (to some extent) self 
contained, we begin by introducing notations and recalling basics. 
\vskip5pt
\noindent
\begin{notations/remarks} 
\label{CND}
Throughout the paper, if not otherwise explicitly stated, we will use the 
following notations/definitions, see e.g.\ \nmnm{Stix}~\cite{St0} for 
basic facts and much more.
\vskip2pt
\itm{25}{
\item[-] $k\,$ is a field, $\oli k|k\,$ is a separable closure of $k.$
\vskip2pt
\item[-] $\ell\neq\chr(k)$ is some odd prime number, fixed throughout.
\vskip2pt
\item[-] $X\,$ is a complete geometrically integral normal $k$-curve.
\vskip2pt
\item[-] $\,K=k(X)$ is the function field of $X$, hence $K|k$ a regular field extension.
\vskip2pt
\item[-] $\oli X=X\!\times_k\oli k$ is the base change, thus $\oli X$ is normal integral. 
\vskip2pt
\item[-] $\oli\pi_1(X)\!:=\pi_1(\oli X)$ and $\,\oli\pi_1(K)\!:=\pi_1(K\oli k)$ 
are the geometric \'etale fundamental groups.
}
\noindent
Hence we get the canonical commutative diagram with exact rows
(geometric point $\oli\eta=\Spec\oli K$): 
\[
{\baselineskip=16pt
\begin{matrix}
\pi_{\XKm Kk}\!:&\ha{10}&1\ha3\to&\oli\pi_1(K)&
             \hor{\oli p_K}&\pi_1(K)&\hor{p_K}&\pi_1(k)&\to\ha31\cr
&&&\Ddwn{\oli \qx_X}&&\Ddwn{\qx_X}&&\hb{1.5}|\hb{1.5}|&\cr
\pi_{\XKm X k}\!:&\ha{10}&1\ha3\to&\oli\pi_1(X)&
       \hor{\oli p_X}& \pi_1(X)&\hor{p_X}& \pi_1(k)&\to\ha31\cr
\end{matrix}
}
\]
Denote by $\Spi Xk$ the set of $\oli\pi_1(X)\hb1$-\ha1conjugacy 
classes of sections $s_X:\pi_1(k)\to\pi_1(X)$ of $p_X$, and 
by $\Spi Kk$ the set of $\oli\pi_1(K)$-\ha1conjugacy 
classes of sections $s_K:\pi_1(k)\to\pi_1(K)$ of~$p_K.$  
\vskip4pt
Since $X$ is a complete normal $k$-curve, the points 
$x\in X$ are in bijection with the $k$-valuation rings 
$\clO_v\in\Val_k(K)$ of the $k$-valuations of $K$ via 
$\clO_x=\clO_v$, thus in particular, equal residue fields 
$\kappa(x)\!:=\clO_x/\eum_x=\clO_v/\eum_v=:\!Kv.$ 
Further, for $x\leftrightarrow v_x$ one has:
$x\in X$ is closed iff $v_x$ is non-trivial iff the residue fields 
$\kappa(x)=Kv_x$ are finite over $k$. And $x\in X$ is 
$k$-rational iff $\kappa(x)=k$ iff $v_x$ is a $k$-rational 
valuation, i.e., $Kv_x=k$, thus getting a canonical bijection
\[
X(k)\to\Val_{k\scl8{-rat}}(K), \ x\mapsto \clO_x=\clO_{v_x}\mapsto v_x.
\]

By functoriality of the \'etale fundamental group, every 
$k$-rational point $x\in X(k)$ gives rise naturally to some $s_x\in\Spi Xk$. 
Further, for $v\in\Val_k(K)$, the prolongations $\oli v|v$ of $v$
to $K^\sep|K$ and~their inertia/decomposition groups 
$T_v\!:=T_{\oli v|v}\triangleleft Z_{\oli v|v}=:Z_v\!<\!\pi_1(K)$ are 
$\oli\pi_1(K)$-conjugated. And for $x\in X(k)$, thus
$v_x\in\Val_{k\scl8{-rat}}(K)$, one has the 
canonical split exact~sequence:
\[
\ha9(\pi_{v_x}):\ha{130}1\to T_{v_x}\hor{\oli p_K} Z_{v_x}
\hor{p_K}\pi_1(k)\to1.\ha{150}
\]

Conclude that the set of conjugacy classes of $s_{v_x}\in \Spi Kk$ 
defined by $x\in X(k)$ is in bijection with the conjugacy classes 
of splittings of $(\pi_{v_x})$ above, hence with 
$\clH_x\!:=H^1_{\rm cont}\pml \pi_1(k),T_{v_x}\pmr$,
the cohomology pointed set of $\pi_1(k)$ with values in $T_{v_x}$. In 
particular, if $\chr(k)=0$, one has that $T_{v_x}=\wh{\lvZ}(1)$, thus via Kummer 
Theory, one has $H^1_{\rm cont}\pml \pi_1(k),T_{v_x}\pmr=\widehat{k^\times}$.
\vskip5pt
Finally, let $\pi_1(k)\neq1$, i.e., $k$ is not separably closed, and
$T=X\vrg T.$ For every $x\in X(k)$ let $s_x\in\Spi Tk$ as above be
fixed, thus getting a map $\varphi_T:X(k)\to\Spi Tk$, $x\mapsto s_x.$ 
One has:
\vskip2pt
\itm{25}{
\item[a)] The map $\varphi_X:X(k)\to\Spi Xk$, $x\mapsto s_x$ is not necessarily 
injective. Indeed, if $X=\lvP^1_k$, then $p_X:\pi_1(X)\to\pi_1(k)$ is 
an isomorphism, hence $\Spi Xk=\{s_X\}$ consists of a single element,
the inverse of $p_X.$ On the other hand, $|X(k)|=|k|+1>1=|\Spi Xk|.$
\vskip2pt
\item[b)] The map  $\varphi_K:X(k)\to\Spi Kk$, $x\mapsto s_x$ is injective. 
Indeed, let $Z_{v_x}$ be a decomposition group above $v_x$ with
$\im(s_x)\subset Z_{v_x}.$ Then $x,x'\in X(k)$ and $\im(s_x)\cap \im(s_{x'})\neq1$ 
implies $Z_{v_x}\cap Z_{v_{x'}}\neq1.$ Since $v_x$, $v_{x'}$ are discrete, 
$Z_{v_x}\!\cap Z_{v_{x'}}\!\neq\!1$ implies\ha1\footnote{See 
e.g.~\cite{P0}, (1.9) Satz, for the general facts about this.}  
$v_{x'}\!=v_{x}$, thus $x'\!=x.$
}
\end{notations/remarks}
Next let $k_t\!:=k(t)$ be the rational function field. For the 
base changes $X_t\!:=X\times_k k_t$ and $K_t\!:=K(t)=k_t(X_t)$, 
consider the commutative diagram below in which $\oli p_t$, 
$\oli p$ and $\oli p_X$ are the canonical inclusions and all the 
other morphisms are the corresponding canonical projections. 
\[
{\baselineskip=16pt
\begin{matrix}
\pi_{\XKm {K_t}{k_t}}\!:&\ha{10}&1\ha3\to&\oli\pi_1(K_t)&
             \hor{\oli p_{t}}&\pi_1(K_t)&\hor{p_{t}}&\pi_1(k_t)&\to\ha31\cr
&&&\Ddwn{\oli q_{K}}&&\Ddwn{q_{K}}&&\Ddwn{q_{k}}&\cr
\pi_{\XKm{K}{k}}\!:&\ha{10}&1\ha3\to&\oli\pi_1(K)&
       \hor{\oli p_K}& \pi_1(K)&\hor{p_K}& \pi_1(k)&\to\ha31\cr
&&&\Ddwn{\oli \qx_X}&&\Ddwn{\qx_X}&&\hb{1.5}|\hb{1.5}|&\cr
\pi_{\XKm X k}\!:&\ha{10}&1\ha3\to&\oli\pi_1(X)&
       \hor{\oli p_X}& \pi_1(X)&\hor{p_X}& \pi_1(k)&\to\ha31\cr
       \end{matrix}
}
\]
%
%
%
\begin{definition/remark} 
\label{CND2}
In the above context, consider definitions as follows:
\vskip0pt
\itm{25}{
\item[1)] A section $s_X\in\Spi Xk$ is {\it birationally liftable\/}
(b.l.), if there is $s_K\in\Spi Kk$ such that $s_X=\qx_X\circ s_K.$ 
If so, $s_K$ is a birational lift of $s_X.$ Obviously, given a 
section $s_K\in\Spi Kk$, one has that $\,s_X\!:=\qx_X\circ s_K$ 
lies in $\,\Spi X k$ and $s_K$ is a birational lift of $s_X.$ 
\vskip2pt
\item[2)] Let $q_X:\pi_1(K_t)\to\pi_1(X)\!$, \ha1i.e., $q_X\!:=\rho_X\circ q_K\!$, 
be the canonical projection. For $T=X$,$\,K$ and  
$q_T:\pi_1(K_t)\to\pi_1(T)$, we say that a section $s_T\in\Spi Tk$ 
of $p_T:\pi_1(T)\to\pi_1(k)$ is $\bm t$-{\it birationally liftable\/}
($\bm t$-\ha1b.l.\ha1), if there is $s_t\in\Spi{K_t}{k_t}$ such that 
$q_T\circ s_t=s_T\circ q_t.$ 
\vskip2pt
\item[$\bullet$] Let $s_X\in\Spi Xk$ be $\bm t$-\ha1birationally liftable. 
It is relatively easy to show that $s_X$ lifts to some $s_K\in\Spi Kk.$
Moreover, in the classical case, i.e., $k|{\lvQ}$ f.g., Bresciani~\cite{Be}, 
Lemma~12, shows that $s_X\in\Spi Xk$ has a lift $s_K\in\Spi Kk$ 
which is itself $\bm t$-\ha1b.l. But in the generality we work in 
{\it is not clear whether\/} $s_X$ has a lift $s_K\in\Spi Kk$ which 
is $\bm t$-\ha1b.l.
\vskip2pt
\item[$\hb{10}\bullet$] To compensate, we say that $s^a_X$ is 
{\it strongly $\bm t$-\ha1b.l.,\/} if $s^a_X$ has a lift $s^a_K$ which 
is $\bm t$-\ha1b.l.
}
\vskip2pt
\itm{35}{
\item[{\bf Note}:]  Following \cite{St0}, Def.\ha{2}27, there is a canonical
map $\Spi Xk\to\Spi{X_t}{k_t}$, compatible with $X(k)\hra X_t(k_t)\hra\Spi {K_t}{k_t}$
defined by $x\mapsto x_t\!=x\times_k k_t$ followed by $x_t\mapsto s_{x_t}.$  
}  
\end{definition/remark}
\noindent
The section conjecture (SC) originates from \nmnm{Grothendieck}
\cite{G1}, \cite{G2}, see~\cite{GGA}, and asserts:
\vskip5pt
\noindent
{\bf Grothendieck SC.} {\it Let $k|\lvQ$ be finitely generated 
and $X$  be a proper hyperbolic $k$-curve. Then all 
{\rm $s_X\in\Spi Xk$} arise form $x\in X(k)$ as described above, 
and {\rm $\varphi_X\!:X(k)\to\Spi Xk$} is a bijection.\/}
\vskip5pt
There are several variants of section conjectures as follows.
The {\bf birational section conjecture (BSC)} asserts that 
in the context of SC above and $K=k(X)$ the function field 
of $X$, all sections $s_K\in\Spi Kk$ arise from $k$-rational 
valuations $v$ of $K|k$, thus from $k$-rational points $x\in X(k)$ 
as explained above. The {\bf $p$-adic SC} and {\bf $p$-adic BSC}
are obtained by replacing the f.g.\ field $k|\lvQ$ by a $p$-adic 
field $k$, i.e.,\ by a finite field extension $k|\lvQ_p$. Finally, 
in the context of Grothendieck~SC, the {\bf $\bm t$-BSC} asserts 
that any section $s_X\in\Spi Xk$ which is $t$-birationally~liftable 
originates from some $k$-rational point $x\in X(k)$ as 
explained above. 
\vskip5pt
Concerning results, conditional/weaker forms of the SC are 
part of the {\it local theory\/} in anabelian geometry by 
\nmnm{Nakamura}~\cite{Na}, \nmnm{Tamagawa}~\cite{Ta}, 
\nmnm{Mochizuki}~\cite{Mz1}, see e.g. the survey articles
\nmnm{Faltings}~\cite{Fa}, \nmnm{Szamuely}~\cite{Sz},
with a significant result by~\nmnm{Stix}~\cite{St2} concerning 
the BSC for geometrically integral hyperbolic curves over 
totally real number fields. This being said, one can say that 
SC and BSC are wide open, and there are only a few complete 
unconditional results concerning forms of the BSC, precisely: 
The $p$-adic BSC is known, see~\nmnm{Koenigsmann}~\cite{Ko1},
for curves and \nmnm{Stix}~\cite{St1} for higher 
dimensional varieties (see also \cite{P0}, Theorems E10, E12). 
The BSC is known for the generic curve $C_g$ over 
$\kappa(M_g)$ by \nmnm{Hain}~[Ha], and second, very 
recently, the $t$-BSC was proved {\it over all $k|\lvQ$ finitely 
generated\/} by \nmnm{Bresciani}~\cite{Be}.
To complete this short list of results, recall that the $p$-adic 
BSC for curves (for all $p$) and higher dimensional varieties 
(for $p>2$) holds under Galois ``minimalistic'' hypotheses. 
For instance, if the $p$-adic field $k$ contains the $p^{\rm th}$ 
roots of unity, then the $\lvZ\!/\!p\,$-{\it metabelian Galois 
theory\/} encodes already the rational points of proper smooth 
$k$-varieties. See~\nmnm{Pop}~\cite{P1}, \cite{P2} and 
\nmnm{L\"udtke} \cite{Lu} for details and further more general facts. 
\vskip5pt
The aim of this note is to both {\it extend\/} the $t$-BSC in its 
initial form and define/introduce and prove ``{\it Galois-minimalistic\/}'' 
type results for the $t$-BSC over quite general~base fields $k$,
thus giving wide generalizations of the $t$-BSC beyond $k|\lvQ$
finitely generated.
\vskip5pt
\noindent
\bltt\ First, recall that given a prime number $\ell\!$, we say 
that a field $k$ {\it is not $\ell\!$-\ha1closed\/} if the degree 
$[\oli k\!:\!k]$ is divisible by $\ell$, or equivalently, the 
$\ell\!$-\ha1Syllow groups of $\pi_1(k)$ are non-trivial.\footnote{\ha2Note 
that being not $\ell\!$-closed is a much weaker hypothesis than 
$G_k$ having non-trivial finite $\ell\!$-quotients.} 
This being said, an application/consequence of the methods 
developed in this note is the following.
%
%
%
\begin{theorem}[{\bf Generalized $t$-BSC}]
\label{gentbsc}
Let $\,k$ be a not $\ell\hb1$-closed perfect field, where 
$\ell\neq\chr(k)$ is odd, and $X$ be a complete integral 
normal $k$-curve with function field $K=k(X)$. 
Then one has:
\vskip2pt
\itm{25}{
\item[{\rm1)}] Every $t$-b.l.\ha3section  
$\,s_K\in\Spi Kk$ arises from a \textbf{unique} 
$\,x_{s_K}\in X(k)$ as explained above.
\vskip2pt
\item[{\rm2)}] Every strongly $t$-b.l.\ha3section $\,s_X\in\Spi Xk$ 
arises from some (possibly not unique) $k$-rational point 
$x_{s_X}\in X(k)$ 
as explained above.
}
\end{theorem}
The above Theorem~\ref{gentbsc} is a relatively easy consequence
of Theorem~\ref{MThm1} below, which is a 
{\it geometrically\ha1-$\lvZ/\ell\hb1$-\ha1minimalistic\/} version of 
the above Theorem~\ref{gentbsc}. First, let us introduce notations 
as follows. For any profinite group $G$, let $G^c=G/G^{(c)}$ and 
$G^a=G/G^{(a)}$ be the $\lvZ/\ell$-abelian-by-central, respectively 
$\lvZ/\ell$-abelian quotients of $G$ as profinite groups. 
Since $G^{(c)}$ and $G^{(a)}$ are characteristic subgroups of $G$, 
every $\Phi\in\Aut(G)$ defines canonically $\Phi^\star\in\Aut(G^\star)$ 
for $\star=a$,$c.$ An important special case of the situation above is 
as follows:
\vskip2pt
Let $T$ be a geometrically irreducible $k$-scheme and 
$1\to\oli\pi_1(T)\to \pi_1(T)\to\pi_1(k)\to1$ be the resulting 
homotopy exact sequence, where $\oli T:=T\times_k\oli k$
and $\oli\pi_1(T)\!:=\pi_1(\oli T)$ is the geomet-
\vskip-2pt\noindent
ric part of $\pi_1(T).$ For $\star=a$,$\,c$
one has: $\oli\pi_1^{(\star)}\!\triangleleft\,\oli\pi_1(T)$ is characteristic 
in $\oli\pi_1(T)\triangleleft\pi_1(T)\!$,~thus
\vskip-2pt\noindent
$\oli\pi_1^{(\star)}\triangleleft\pi_1(T).$ Hence 
setting $\oli\pi_1^\star(T)=\oli\pi_1(T)/\oli\pi_1^{(\star)}(T)\!$ and 
$\Pi^\star(T)\!:=\pi_1(T)/\oli\pi_1^{(\star)}(T)\!$, one gets a canonical
commutative diagram:
\[
\begin{matrix}
1\to&\oli \pi_1(T)&\hor{}&\pi_1(T)&\hor{p_T}&\pi_1(k)&\to1\cr
&\hb3\ddwn{}&&\ddwn{}&&|\hb2|&\cr
1\to&\oli\pi^c_1(T)&\hor{}&\Pi^c(T)&\hor{p^c_T}&\pi_1(k)&\to1\cr
&\Ddwn{\oli{pr}_{T}}&&\Ddwn{pr_T}&&|\hb2|&\cr
1\to&\oli\pi^a_1(T)&\hor{}&\Pi^a(T)&\hor{p^a_T}&\pi_1(k)&\to1\cr
\end{matrix}
\]
\vskip3pt
The process $T\leadsto\Pi^\star(T)\to\pi_1(k)$ with $\star=a$,$\,c$ 
is functorial in $T$ for geometrically irreducible separated $k$-schemes. 
Hence in the notation introduced right before Definition~\ref{CND2}
above, one has the commutative diagrams below, 
the middle diagram being a quotient of the first one, and the right hand
one being a quotient of the middle one:
\vskip-15pt
\usetikzlibrary{decorations.pathmorphing}
\[
\mathsurround=0pt
\begin{tikzcd}[column sep=scriptsize] 
\pi_1(K_t) \arrow[r, two heads, "p_t" '] \arrow[d, two heads, "q_K"]
  &\pi_1(k_t) \arrow[d, two heads, "q_k"] \arrow[l, dashed, bend right=10, "s_t" ']                       
    &\Pi^c(K_t) \arrow[r, two heads, "p^c_t" ']\arrow[d, two heads, "q_K^c"] 
      &\Pi^c(k_t) \arrow[d, two heads, "q_k^c"] 
                \arrow[l, dashed, bend right=10, "s^c_t" ']
&\Pi^c(K_t) \arrow[r, two heads, "p^c_t" ']\arrow[d, two heads, "q_K^a"] 
     &\Pi^c(k_t) \arrow[d, two heads, "q_k^a"] 
          \arrow[l, dashed, bend right=10, "s^c_t"']
                          \\
\pi_1(K) \arrow[r, two heads, "p_K" '] \arrow[d, two heads, "\qx_X"]
    &\pi_1(k) \arrow[l, dashed, bend right=10, "s_K" '] \arrow[d, "{\rm id}"] 
                      \arrow[r, squiggly] 
       &\Pi^c(K) \arrow[r, two heads, "p_K^c" '] \arrow[d, two heads, "\qx^c_X"]
         &\pi_1(k) \arrow[l, dashed, bend right=10, "s^c_K"'] \arrow[d, "{\rm id}"]             
                     \arrow[r, squiggly ] 
&\Pi^a(K) \arrow[r, two heads, "p_K^a" '] 
        \arrow[d, two heads, "\qx^a_X"]
       &\pi_1(k) \arrow[l, dashed, bend right=10, 
           "s^a_K"'] \arrow[d, "{\rm id}"]             
                        \\
\pi_1(X) \arrow[r, two heads, "p_X" ']
   &\pi_1(k) \arrow[l, dashed, bend right=10, "s_X" ']              
        &\Pi^c(X) \arrow[r, two heads, "p_X^c" ']
            &\pi_1(k) \arrow[l, dashed, bend right=10, "s_X^a"']                 
&\Pi^a(X) \arrow[r, two heads, "p_X^a" ']&\pi_1(k)
         \arrow[l, dashed, bend right=10, "s_X^a" '] 
          \\
\end{tikzcd}
\]
\vskip-20pt
\noindent
\bltt\ In the sequel, if not otherwise explicitly stated, we denote 
$\star=a\hb1$,$c$ and $T=X$,$\,K$ and consider the 
canonical projection $\,q^a_T:\Pi^c(K_t)\to\Pi^a(T)$ from the 
diagrams above.
%
%
%
\begin{remarks}
\label{CND3}
In the above notation and context, we notice/have the following:
\vskip2pt
\itm{25}{
\item[1)] Every section $s_T\in\Spi Tk$ gives rise canonically to 
sections $s^\star_T:\pi_1(k)\to\Pi^\star(T)$ of $p^\star_T$ such 
that $s_T$ is a lift of $s^c_T$ and $s^c_T$ is a lift of $s^a_T$, 
thus $s_T$ is a lift of $s^a_T$ as well.
\vskip2pt
\item[2)] Every section $s_t\in\Spi {K_t}{k_t}$ gives rise canonically 
to sections $s^\star_t:\Pi^\star(k_t)\to\Pi^\star(K_t)$ of $p^\star_t$
such that $s_t$ is a lifting of $s^c_t$ and $s^c_t$ is a lifting of 
$\,s^a_t$, thus $s_t$ is a lift of $s^a_t$ as well.
\vskip2pt
\item[3)] For $s_T\in\Spi Tk\!$ and $s_T^\star$ defined by $s_T$
as at item~1) above, one has: If $s_t\in\Spi{K_t}{k_t}$ lifts $s_T\in\Spi Tk\!$, 
then $\,s_t$ gives rise canonically to a section $\,s_t^\star$ of $\,p_t^\star$ 
which lifts~$s_T^\star.$
}
\end{remarks}
%
%
%
\begin{definition} 
We say that a section $s^a_T:\Pi^a(k)\to\Pi^a(T)$ of 
$p^a_T:\Pi^a(T)\to\pi_1(k)$ is {\it $t\!$-a.b.c.\ birationally liftable,\/} 
if there is a section $\, s^c_t$ of $\,p^c_t$ lifting $s^a_T$, i.e., 
$s^a_T\circ q^a_k=q^a_T\circ s_t^c.$ If so, we also say that 
$s^c_t$ is a {\it $t\!$-a.b.c.\ birational lift\/} of $s^a_T.$
\end{definition}
%
%
%
\begin{remarks}
\label{CND4}  
In the above notation and context, we notice/have the following:
\vskip2pt
\itm{25}{
\item[1)] Let $s^c_t$ be a $t\!$-a.b.c.\ birational lift of a given 
$s^a_K\in\Spi Kk.$ Then $s^c_t$ is a $t\!$-a.b.c.\ birational lift of 
$s^a_X\!:=\qx^a_K\circ s^a_K$, thus $s^a_X$ is $t\!$-a.b.c.\ 
birationally\ha2liftable.
\vskip2pt
\item[2)] If $s^a_X$ is $t\!$-a.b.c.\ birationally liftable, then $s^a_X$
lifts to some sections $s^c_K$ of $p^c_K$, thus $s^a_K$ of $p^a_K.$  
But it is {\it unclear whether\/} $s^a_X$ lifts to some $s^a_K$ which 
is $t\!$-a.b.c.\ birationally liftable.
\vskip2pt
\item[$\hb{10}\bullet$] To compensate, we say that a section $s^a_X$ of
$p^a:\Pi^a(X)\to\pi_1(k)$ is {\it strongly $t\!$-a.b.c.\ birationally liftable,\/} 
if $s^a_X$ lifts to some $s^a_K$ which is itself $t\!$-a.b.c.\ birationally 
liftable.
}
\end{remarks}
\noindent
This been said, the above Theorem~\ref{gentbsc} 
is a consequence of the following deeper fact.
%
%
%
\begin{theorem}[{\bf\hb{3.25}(\ha1Minimalistic $t$-BSC\ha1}]
\label{MThm1}
Let $\,k$ be a non-$\ell\hb1$-closed perfect field with $\ell\neq\chr(k)$ 
odd, and $X$ be a complete integral normal $k$-curve with function 
field $K=k(X)$. Then one has:
\vskip2pt
\itm{25}{
\item[{\rm1)}] Every $t$-a.b.c.\ birationally liftable section 
$s^a_K\!:\pi_1(k)\to \Pi^a(K)$ of $\,p^a_K\!:\Pi^a(K)\to\pi_1(k)$ is defined 
by a \textit{\textbf{unique}} $k$-rational point $\,x_{s_K}\!\in X(k)$ as 
explained above.
\vskip2pt
\item[{\rm2)}] Every strongly $t$-\ha1a.b.c.\ birationally liftable 
section $\,s^a_X:\pi_1(k)\to\Pi^a(X)$ of $p^a_X$ is defined by 
some $k$-rational point $s_{s_X}\!\in X(k)$\ha2---\ha1maybe not 
unique\ha1---\ha3in the way explained above.
}
\end{theorem}
Let us show that
Theorem~\ref{MThm1} implies Theorem~\ref{gentbsc}. 
By mere definitions, assertion~2) follows from assertion 1),
therefore it is enough to prove~1). For $s_K\in\Spi Kk$ and 
$s_t\in\Spi{K_t}{k_t}$ a $t$-birational lift of $s_K\!$, 
let $K_{s_t}|K_{s_K}\hra K_t^\sep|K^\sep$ 
be the fixed fields of $s_K\pml\pi_1(k)\pmr\subset\pi_1(K)$
and $s_t\pml\pi_1(k_t)\pmr\subset\pi_1(K_t).$ Set
$K_{s_K}=\cup_\alpha K_\alpha$ with $K_\alpha|K\hra K_{s_K}|K$ 
the finite subextensions. 
Then the normalization $X_\alpha\to X$ of $X$ in $K_\alpha|K$ is 
a geometrically integral model of $K_\alpha|k$ and the inclusion 
$K_\jb\supset K_\ia$ gives rise canonically to a (surjective) 
$k$-morphism $f_{\jb\ia}\!:X_\jb\!\to X_\ia.$ Further, setting $X_{\ia,t}\!:=
X_\ia\times_k k_t$ and $K_{\alpha,t}=k_t(X_{\alpha,t})=K_\ia(t)\!$, 
by mere definitions one has \hbox{$K_{\ia,t}\subset K_{s_t}$.} 
Hence $K_s\supset K_\ia\supset K$ 
and $K_{s_t}\supset K_{\ia,t}\supset K_t$ implies 
$\im(s_K)\subset\pi_1(K_\ia)\subset \pi_1(K)$ 
and $\im(s_t)\subset\pi_1(K_{\ia,t})\subset\pi_1(K_t).$ Thus
$s_K$ and $s_t$ give rise to a sections $s_{K_{\ia}}\in\Spi{\ia,t}k$ of 
$p_{K_\ia}=p_K|_{\pi_1(K_\ia)}$ and $s_{K_{\ia,t}}\in\Spi{K_{\ia,t}}{k_t}$ of 
$p_{\ia,t}=p_t|_{\pi_1(K_{\ia,t})}$, respectively. In particular, one has
that $q_{K_\ia}=q_K|_{\pi_1(K_{\ia,t})}$ and further, $s_{K_{\ia,t}}$ 
is a $\bm t$-birational lift of $s_{K_\ia}$, see the diagrams below:
\vskip-10pt
\[
\mathsurround=0pt
\begin{tikzcd}[column sep=scriptsize] 
\pi_1(K_t) \arrow[r, two heads, "p_t" '] \arrow[d, two heads, "q_K"]
  &\pi_1(k_t) \arrow[d, two heads, "q_k"] \arrow[l, bend right=10, "s_t" ']                       
    &\pi_1(K_t)\supset\hb{32}&\pi_1(K_{\ia,t}) \arrow[r, two heads, "p_{\ia,t}" ']
                      \arrow[d, two heads, "q_{K_\ia}"] 
      &\pi_1(k_t) \arrow[d, two heads, "q_k"] 
                \arrow[l, dashed, bend right=10, "s_{\ia,t}" ']
                          \\
\pi_1(K) \arrow[r, two heads, "p_K" '] 
    &\pi_1(k) \arrow[l, bend right=10, "s_K" '] 
       &\pi_1(K)\supset\hb{37}&\pi_1(K_\ia) \arrow[r, two heads, "p_{K_\ia}" '] 
         &\pi_1(k) \arrow[l, dashed, bend right=10, "s_{\ia}"'] \end{tikzcd}
         \\
\]

To conclude, recall that by Remark~\ref{CND3}, for $K_\alpha$ 
and $s_{K_\ia}\in\Spi {K_\alpha}k$ and its $t$-birational 
lifting $s_{\alpha,t}\in\Spi{K_{\alpha,t}}k$, one has the
resulting section $s^a_\alpha$ of 
$p^a_{K_\alpha}:\Pi^a(K_\alpha)\to\pi_1(k)$ which has the 
$t$-a.b.c.\ birational lifting $s_{\alpha,t}^c:\Pi^c(k_t)\to\Pi^c(K_t).$  
Hence by Theorem~\ref{MThm1} applied to the $k$-curve
$X_\alpha$ and $s_\alpha$, etc., it follows that $s_\alpha^a$ 
is defined by a unique closed point unique $\,x_\alpha\in X_\alpha(k)$.
On the other hand, for $K_\alpha\subset K_\beta$ and the 
canonical projection $f_{\beta\alpha}:X_\beta\to X_\alpha\!$,  
sorting through the definitions, 
one has: $x'_\alpha=f_{\beta\alpha}(x_\beta)\in X_\alpha(k)$
is a $k$-rational point of $X_\alpha$ which defines $s_{K_\ia}$
as well. Hence by the uniqueness $x_\alpha\in X_\alpha(k)$ 
defining $s_\alpha^a$, one must have $x'_\alpha= x_\alpha$. 
\vskip2pt
{\bf Conclude:} The compatible 
system of $k$-rational points $(x_\alpha)_\alpha$ defines 
the unique $k$-rational point $x_s\in X(k)$ which defines the 
$t$-birationally liftable section $s_K$ we started with.
\vskip5pt
\noindent
{\bf Plan of the paper}. 
\vskip2pt
In Section 2 we review special facts 
about Hilbert decomposition theory: Especially the situation
of the a.b.c.\ field extensions, and explaining how some 
results of \nmnm{Topaz},~\cite{To1} applied to our situation 
recover arithmetically significant valuations from 
the group theoretical information. 
\vskip2pt
In section 3 we show how to 
apply the results from Section~2 to get arithmetically significant 
valuations from Galois sections in the concrete situation at the core 
of the paper. 
\vskip2pt
Finally, in Section 4 we consider an {\it arithmetical\/} refinement 
of the context of Theorem~\ref{MThm1}, by introducing 
{\it $\tl k|k$-$t$-a.b.c.\ birationally\/} liftable sections\ha3---\ha1thus
refining the $t$-a.b.c.\ birational liftability of sections, which coincides
with the latter in case $\tl k=\oli k.$ This being said,
Theorem~\ref{MThm1} above follows by mere definitions
from Theorem~\ref{MThm} of Section~4.
\section{Commuting Liftability and Recovering Valuations}
\subsection{\bf Basics of valuations theory}
$\ha0$
\vskip3pt
For arbitrary fields $\Oy$, let $\Val{}(\Oy)$ be the set of (equivalence 
classes of) valuations $v$ of $\Oy.$ 
For $v\in\Val{}(\Oy)$, let $\eum_v\subset\clO_v$ be its valuation 
ideal\ha1/\ha1ring, $\Oy v=\kappa(v)=\clO_v/\eum_v$ its residue  
field, and $v\Oy=\Oy^\times\hb3/\clO_v^\times$ the (canonical) 
value group of $v.$ Recall that $\Spec(\clO_v)$ is a chain w.r.t.
inclusion, and for each $\eum_1\in\Spec(\clO_v)$, the localization 
$\clO_1\!:=(\clO_v)_{\eum_1}$ is a valuation ring with valuation 
ideal $\eum_1\!$ and valuation $v_1\in\Val(\Oy)$, i.e., 
$\clO_1=\clO_{v_1}$ and $\eum_1=\eum_{v_1}.$
Moreover, all the rings $\clO_1\subset\Oy$ with $\clO_v\subset\clO_1$ 
are the valuation rings of the form above, i.e., $\clO_1=(\clO_v)_{\eum_1}$ 
for some $\eum_1\in\Spec(\clO_v).$ 
Thus setting 
$\clV_v(\Oy)\!:=\{v_1\in\Val(\Oy)\mid \clO_{v_1}\supset \clO_v\}$ and
$\clR_{\clO_v}\!:=\{\clO_1\subset\FF\mid\clO_v\subset\clO_1\}\!$,
one has canonical bijections:
\[
\clV_v(\Oy)\to\clR_{\clO_v}\to\Spec(\clO_v),\quad
v_1\mapsto\clO_{v_1}\mapsto\eum_{v_1}.
\] 
Finally, $\Val{}(\Oy)$ carries a natural partial ordering $\leqslant$ 
defined by the equivalent conditions: 
\vskip6pt
\centerline{$v_1\leqslant v$ \ iff  \ $\clO_{v_1}\supset\clO_{v}$ \ iff \
$\eum_{v_1}\subset\eum_{v}$ \ iff \ $\eum_{v_1}\in\Spec(\clO_{v}).$}
\vskip5pt
\noindent
If $v_1\!\leqslant\! v$, we say that $v_1$ a {\it coarsening\/} of $v$, 
and $v$ is a {\it refinement\/} of $v_1.$ Further, if \hbox{$v_1\!\leqslant\! v$,} 
then $\clO_{0}\!:=\clO_v/\eum_{v_1}$ is a valuation ring of the residue
field $\Oy_0\!:=\Oy v_1$ 
having $\eum_{0}=\eum_{v}/\eum_{v_1}$ as valuation ideal, 
and obviously, $\Oy_0v_0=\clO_{v_0}/\eum_{v_0}=\clO_v/\eum_v=\Oy v.$ 
We denote the resulting valuation of $\clO_0$ by $v_0\!:=v/v_1$ 
and call $v_0$ the (valuation theoretical) 
{\it quotient\/} of $v$ by $v_1\!$, and set $v=v_0\circ v_1$ and call 
$v$ the (valuation theoretical) {\it composition\/} of $v_0$ and $v_1.$
Further, the projection $v\Oy=\Oy^\times\hb3/\clO_v^\times 
\srjr\Oy^\times\hb3/\clO_{v_1}^\times=v_1\Oy$ is order preserving, 
and its kernel is a convex subgroup $\Delta_1\leqslant v\Lambda$
which is canonically isomorphic to $v_0\Lambda_0,$ thus getting
an exact sequence of ordered groups \hbox{$0\to v_0\Oy_0\to v\Oy\to v_1\Oy\to0.$}
Conversely,  if $\Delta_1\leqslant v\Oy$ is a convex subgroup, 
then $v\Oy\to v\Oy/\Delta_1$ is order preserving, giving rise to a valuation 
$v_1\in\Val(\Oy)$ with $v_1\leqslant v.$ Conclude that $\clV_v(\Oy)$ is in 
canonical bijection with the set 
$\{\Delta_1\leqslant v\Oy\,|\,\hbox{convex subgroup}\}.$
\vskip5pt
Last but not least, for $v_1\vrg v_2\in\Val{}(\Oy)$ there is a well defined 
valuation $v=\min(v_1\vrg v_2)$ in $\Val{}(\Oy)$ whose valuation ring $\clO_v$ 
is characterized as follows: $(\clO_{v_1})_{\eum}=\clO_v=(\clO_{v_2})_{\eum}$ 
and $\eum_v=\eum$, where $\eum\in\Spec(\clO_{v_1})\cap\Spec(\clO_{v_2})$
is the unique maximal element w.r.t.\ inclusion. Equivalently, $\eum$ is  
maximal in $\Spec(\clO_{v_1})\cap\Spec(\clO_{v_2})$ satisfying
$\eum\cap\clO_{v_1}^\times=\emptyset=\clO_{v_2}\cap\eum.$ 
\vskip5pt
Finally, every $v\in\Val{}(\Oy)$ defines a field 
topology $\tau_v$ on $\Oy$, and $\tau_v$ is discrete
iff $v$ is the trivial valuation. 
Let $v_1\vrg v_2\in\Val{}(\Oy)$ be nontrivial. Then
$\tau_{v_1}=\tau_{v_2}$ \ iff \ $v_1$ and $v_2$ have
a common \underbar{non-trivial} coarsening $v\leqslant v_1\vrg v_2$
and if so, $\tau_{v_1}=\tau_v=\tau_{v_2}$.
If this is the case, we say that $v_1\vrg v_2$ are {\it dependent.\/} 
Complementary, we say that $v_1\vrg v_2$ are 
{\it independent,\/} if $\tau_{v_1}\neq\tau_{v_2}$, or 
equivalently, the diagonal embedding 
$\Oy\to(\Oy,\tau_{v_1})\times (\Oy,\tau_{v_2})$ has 
a dense image. Notice that for $v_1\vrg v_2\in\Val{}(\Oy)$, and 
$U_i\subset \Oy$ non-empty $v_i$-open, $i=1,2$, 
the following are equivalent:
\vskip5pt
\centerline{{\rm(i)} \ $v_1$, $v_2$ are independent;
   \ \ \ {\rm(ii)} \ $U_1\!-U_2=\Oy$; \ \ \
    {\rm(iii)} \ $\Oy^\times\subset U_1\!\cdot U_2$.}
\vskip5pt
\noindent
\begin{fact}
\label{fkt0}
{\it In general, given $v_1,v_2\in\Val(\Oy)$ non-trivial and
$v\!:=\min(v_1,v_2)$, set $U_{v_i}=1+\eum_{v_i}$, 
$i=1,2$ and $U_v=1+\eum_v.$ The following hold:
\vskip2pt
\itm{25}{ 
\item[{\rm1)}] If $v_1\leqslant v_2$, one has 
$U_{v_1}\hb3\cdot\hb1 U_{v_2}=U_{v_2}\!$, 
$\clO_{v_1}\hb3\cdot\clO_{v_2}=\clO_{v_1}\!$,
\ $U_{v_2}\hb3-\hb1 U_{v_1}=\eum_{v_2}.$
\vskip2pt
\item[{\rm2)}] If $v<v_1,v_2$ strictly, then 
$U_{v_1}\hb3\cdot\hb1 U_{v_2}=\clO_v^\times=
\clO_{v_1}^\times\hb3\cdot\clO_{v_2}^\times$,
\ $U_{v_1}\hb3-\hb1 U_{v_2}=\clO_v=\clO_{v_1}-\clO_{v_2}.$
\vskip2pt
}}
\end{fact}
\begin{proof} The assertions from 1) follow by mere definition.
\vskip2pt
\noindent
To~2): By mere definitions, the quotient valuations
$\oli v_i=v_i/v$, on the residue field $\Oy v$ are independent.
Hence setting $U_{\oli v_i}\!:=1+\eum_{\oli v}$, $i=1,2$ one
has that $\Oy v^\times=U_{\oli v_1}\hb3\cdot U_{\oli v_2}$
by the discussion above. Further, the canonical exact sequence 
\[
1\to U_v\to\clO_v^\times\to \Oy v^\times\to 1\,
\leqno{\indent(*)}
\]
defines exact sequences $1\to U_v\to U_{v_i}\to U_{\oli v_i}\to 1$, 
thus the subsequence below of $(*)$ above:
\[
1\to U_v\hra U_{v_1}\!\cdot U_{v_2}\srjr U_{\oli v_1}\!\cdot U_{\oli v_2}\to 1\,,
\leqno{\indent(**)}
\]
in which the first map is injective, and the second one is surjective. 
On the other hand, sine $\oli v_1,\oli v_2$ are independent on $\Oy v$, 
one has $\Oy v^\times=U_{\oli v_1}\cdot U_{\oli v_2}$. Hence since 
$U_{v_1}\cdot U_{v_2}\subset\clO_v^\times$ and $\ker(\pi)=U_v$, we 
conclude that $(**)$ is exact, implying finally $U_{v_1}\!\cdot U_{v_2}=\clO_v^\times$.
\vskip4pt
The proof of the assertion $U_{v_1}\hb3-\hb1 U_{v_2}=\clO_v$ is similar,
being the additive variant.
\end{proof}
{\bf Canonical $v$-valuation.} Let $\Ox|\Oy$ be an arbitrary field 
extension and 
$w\in\Val(\Ox)$ and $\vxy\in\Val(\Oy)$ satisfy 
$w\!_\Lm:=w|_\Ox\geqslant\vxy.$ Equivalently, by general 
valuation theory, one has: 
\[
\clO_w\cap\Oy=\clO_{w_\Lambda}\subset\clO_\vxy\,\hbox{,} \ \  
(1+\eum_w)\cap\Oy=1+\eum_{w_\Lm}\supset 1+\eum_\vxy\hbox{, \ \ etc.}
\]
In particular, by the above discussion 
about coarsening, $\clO_\vxy=(\clO_{w_\Lm})_{\eum_v}$ 
is the localization of $\clO_{w_\Lm}$ with respect to its prime
ideal $\eum_\vxy\in\Spec(\clO_{w_\Lm}).$ Equivalently, setting
$\Sigma_{w_\Lm}:=\clO_{w_\Lm}\backslash\,\eum_\vxy$, one has
that $\Sigma_{w_\Lm}$ is a multiplicative system in
$\clO_{w_\Lm}$ defining $\clO_v$ as follows:
\[
\clO_v=(\clO_{w_\Lm})_{\eum_v}=\Sigma^{-1}_{w_\Lm}\clO_{w_\Lm}.
\]
\begin{lemma}
\label{can-val-abstr}
$\clO_*\!:=\Sigma^{-1}_{w_\Lm}\clO_w\subset\Ox$ 
is a valuation ring with valuation $w_*\in\Val(\Ox)$ s.t.\ $w_*|_\Oy=v.$
\end{lemma}
\begin{proof}
Indeed, $\clO_*\cap\Oy=\{{a\over r}\in\Oy\mid a\in\clO_w,\,r\in\Sigma_{w_\Lm}\}$
and we have to prove that $\clO_*\cap\Oy=\clO_\vxy.$ For the direct 
inclusion, let $x={a\over r}\in\clO_*\cap\Oy$ with $a\in\clO_w$, 
$r\in\Sigma_{w_\Lm}.$ Then $a=rx\in\Oy$, thus concluding that
$a\in\clO_w\cap\Oy=\clO_{w_\Lm}\subset\clO_v$, and finally
$x={a\over r}\in\Sigma^{-1}_{w_\Lm}=\clO_v.$ The converse implication
is clear, because $\clO_v=\Sigma^{-1}_{w_\Lm}\clO_{w_\Lm}
\subset\Sigma^{-1}_{w_\Lm}\clO_w=\clO_*.$ 
\end{proof}
Let $\eum_1\in\Spec(\clO_{w_*})\subset\Spec(\clO_w)$ be the
(unique) prime ideal which is {\it minimal\/} satisfying 
$\eum_1\cap\Oy\supset\eum_v.$ Then one
has $\eum_v\subset\eum_1\cap\Oy\subset\eum_{w_*}\cap\Oy=\eum_v$,
thus $\eum_1\cap\Oy=\eum_v.$ We conclude that the valuation 
of the valuation ring $\clO_1=(\clO_{w_*})_{\eum_1}\!$, 
denoted $w_v$, satisfies $w_v|_\Oy=v.$
\begin{definition}
In the above notation and context, $w_\vxy$ is the 
{\it canonical $\vxy$-valuation of~$\Ox$.\/} Thus $w_v$ is unique 
minimal with $w_v\leqslant w$, $w_v|_\Lm=v$, that is,
$\clO_{w_\vxy}\cap\Oy=\clO_\vxy$, $\eum_{w_\vxy}\cap\Oy=\eum_\vxy.$
\end{definition}
Finally, let $(\Ox'\!,w'\,|\,(\Ox,w)$ and $(\Oy'\!,\vxy')\,|\,(\Oy,\vxy)$ be
algebraic extensions of valued fields such that $\Ox'\supset\Oy'$ and 
$w'|_{\Lambda'}\geqslant v'\!$, thus $w=w'|_\Omega$ satisfies 
$w|_\Lambda=w'|_\Lambda\geqslant v'|_\Lambda=v.$ For short, 
we denote this situation by $(\Ox'|\Oy'\!,\,w'|\vxy')\big|(\Ox|\Oy,w|\vxy).$
Then one has the following (obvious) facts. 
\begin{fact}
\label{fkt-c-val}
{\it In the above notation, the following hold: 
\vskip2pt
\itm{25}{
\item[{\rm1)}] $\clO_{w_v}^\times\cap\Oy=\clO^\times_\vxy$ 
and $(1+\eum_{w_v})\cap\Oy=1+\eum_\vxy.$ 
\vskip2pt
\item[{\rm2)}] Let $(\Ox'|\Oy'\!,\,w'|\vxy')\big|(\Ox|\Oy,v|\vxy)$ be as 
above, thus $w'|_\Oy\geqslant\vxy$ for the field extension $\Ox'|\Oy.$ 
\item[] Then $w'_\vxy=w'_{\vxy'}$ and $w'_\vxy|_{\Om}=w_\vxy=w'_{\vxy'}|_\Om.$
}}
\end{fact}
\begin{proof} 
Assertion 1) follows by mere definitions, etc. For assertion~2), recall 
that for any valuations $\tilde w'$ of $\Ox'$ and $\tilde\vxy'|\tilde\vxy$
of the algebraic extension $\Oy'|\Oy$ one has: $\tilde w'|_{\Oy'}=\vxy'$
iff $\tilde w'|_{\Oy}=\vxy$, etc.
\end{proof}
\vskip5pt
\subsection{\bf Basics of Hilbert decomposition theory, especially in $\GGa\FF$}
$\ha0$
\vskip5pt
Let $\FF'|\EE$ be an algebraic field extension, $v\in\Val(\EE)$ be
a fixed valaution, and $\clV_v(\FF')$ be the set of prolongations 
$w'|v$ of $v$ to $\FF'|\EE.$ Recall that $\Val_v(\FF')$ is a profinite 
topological space in the patch topology.\footnote{Actually, 
$\Val_v(\FF')$ endowed with the patch topology is a profinite
space even if $\FF'|\EE$ is not algebraic.} Moreover, 
if $\FF'|\EE$ is normal algebraic, the profinite automorphisms 
group $G(\FF'|\EE):=\Aut_\EE(\FF')$ acts transitively 
and continuously on the profinite space $\clV_v(\FF')$ via 
$(w'\hb2,\,\gx)\mapsto w'^\gx\hb2:=w'\circ\gx^{-1}\hb3=:\hb2w''\!$. 
And if  $T_{w'|v}\lhd Z_{w'|v}$ are the inertia/decomposition 
groups of $w'|v$, then $T_{w''|v}\!=\gx\,T_{w'|v}\,\gx^{-1}$ 
and $Z_{w''|v}\!=\gx\,Z_{w'|v}\,\gx^{-1}\!$, and for any 
$w'\in\clV_v(\FF')$ fixed have: 
\vskip4pt
\centerline{$\clV_v(\FF')=G(\FF'|\EE)\!\cdot w' \cong Z_{w'|v}\!
\backslash G(\FF'|\EE)$ as $G(\FF'|\EE)$-spaces, canonically.} 
\vskip4pt
\noindent
Further, the residue field extension $\FF'\hb1w'|\EE v$ is normal 
algebraic, and setting $G_{w'|v}:={\rm Aut}_{\EE v}(\FF'\hb1w')$, 
one has the canonical exact sequence: 
\vskip4pt
\centerline{$1\to T_{w'|v}\hor{} Z_{w'|v}\hor{pr} G_{w'|v}\to1$.}  
\vskip2pt
Next let $v_1<v$ in $\Val{}(\EE)$. There are prolongations 
$w'_1|v_1$ of $v_1$ to $\FF'|\EE$ such that $w'_1<w'.$ 
Further, for any such $w'_1|v_1$ the following hold: First, 
$Z_{w'|v}\subset Z_{w'_1|v_1}$ and both $T_{w'_1|v_1}\lhd T_{w'|v}$ 
and $T_{w'_1|v_1}\lhd Z_{w'|v}$. Second, $w'_0\!:=w'/w'_1$ prolongs 
the valuation $v_0\!:=v/v_1$ to $\FF' w'\!$, and via the canonical exact sequence 
$1\to T_{w'_1|v_1}\hor{} Z_{w'_1|v_1}\hor{\scriptscriptstyle pr} G_{w'_1|v_1}\to 1$
the following hold: 
\vskip5pt
\centerline{$Z_{w'_0|v_0}=pr(Z_{w'|v})=Z_{w'|v}/T_{w'_1|v_1}$ and
$T_{w'_0|v_0}=pr(T_{w'|v})=T_{w'|v}/T_{w'_1|v_1}$,} 
\vskip5pt
\noindent
giving rise to a commutative diagram exact sequences of the form:
\[
\begin{matrix}
1\,\,\to& T_{w'_1|v_1}&\hor{} &Z_{w'_1|v_1}&\hor\pr& 
                                      G_{w'_1|v_1}&\to\,\,1_{\nix_{\nix_\nix}}  \cr
          &|\hb2| &&\uparrow&         &\uparrow&                              \cr
1\,\,\to& T_{w'_1|v_1}&\hor{} &Z_{w'|v}&\hor{\pr}&Z_{w'_0|v_0}&\to\,\,1\,.
\end{matrix}
\] 

And recall that if $w'|v$ is tame, i.e., $T_{w'|v}$ has order prime 
to $\chr(\EE v)$, $T_{w'|v}$ is abelian and
$T_{w'|v}=\Hom(w'\FF'\!/\hb1v\EE,\,\mu_{\FF'\hb1w'})$
with $\mu_{\FF'\hb1w'}\subset \FF'\hb1w'$ the roots of unity in 
$\FF'\hb1w'$. And the  action of $Z_{w'|v}$ on $T_{w'|v}$ factors 
through $Z_{w'|v}\srjr G_{w'|v}$, and $G_{w'|v}$ acts on 
$T_{w'|v}=\Hom(w'\!\FF'\!/\hb1v\EE,\ha1\mu_{\FF'\hb1w'})$ 
via the cyclotomic character of $G_{w'|v}$.
\vskip5pt
We next consider the following special situation of the
general context above: 
Let $\ell>2$ be an odd prime number (fixed throughout), and 
$\FF|\EE$ be a Galois extension, $\chr(\EE)\neq\ell$, 
and $\mu_\ell\subset\FF$. Let $\FF^c|\FF^a|\FF$ be the 
maximal \abcZ, respectively $\lvZ/\ell$ 
elementary abelian, extensions of $\FF$, and for the 
corresponding exact sequence of Galois groups
\[
1\to\Delta_\FF\!:=\Gal(\FF^c|\FF^a)\to \GGc\FF:=
\Gal(\FF^c|\FF)\to\GGa\FF:=\Gal(\FF^a|\FF)\to1,
\]
denote $\GGc\FF\ni\tlsgm
\mapsto\tlsgm|_{\FF^a}=:\sigma\in\GGa\FF$ the 
corresponding projection. 
%
\vskip5pt
We notice/recall that by Kummer Theory, 
one has $\GGa\FF=\Hom(\FF^\times\!\!,\mu_\ell)$,  
and $\GGac_\FF$ 
is the maximal $\lvZ/\ell$ elementary abelian 
quotient of the absolute Galois group $G_{\FF^a}$ 
on which $\GGa\FF$ acts trivially. 
One obtains $\FF^c|\FF$ as follows: $\GGa\FF$
acts canonically on ${\FF^a}^\times\hb4/\ell$, and let
$A:=({\FF^a}^\times\hb4/\ell)^{\GGa\FF}$ be the 
subgroup of invariants; that is, $u\in \FF^a$
lies in $A$ iff $\forall\,\sigma\in\GGa\FF$ 
$\exists\,r_\sigma\in \FF^a$ such that $\sigma(u)=ur_\sigma^\ell$. 
Then one has $\FF^c=\FF^a[\root\ell\of A]$. 
From this discussion immediately follows the following.
\vskip7pt
\noindent
{\bf Basic Fact.} \ {\it $\FF^c|\FF^a|\EE$
are Galois extensions of $\EE$.\/} 
\vskip5pt
\noindent
One has the following basic facts (well known to experts, but 
I cannot give a precise reference).
\begin{fact}
\label{basic-fact}
%
%
{\it Suppose that $\mu_\ell\subset F$ provided
$\chr(\FF)\neq\ell$. For a valuation $w\in\Val{}(\FF)\!$,
let $w^a|w$ be a prolongation of $w$ to $\FF^a|\FF$, and 
$\FF^h$ be the $w$-Henselization of $F.$ The following hold:
\vskip2pt
\itm{25}{
\item[{\rm1)}] The compositum $\FF^h\FF^a$ equals the maximal 
$\ell$-elementary abelian extension $(\FF^h)^a$ of $\FF^h$.
\vskip2pt
\item[{\rm2)}] The separable part of $\FF^a w^a|\FF w$ is the 
maximal $\ell$-abelian extension $(\FF w)^a|\FF w$ of $\FF w.$
}}
\end{fact}
\begin{proof} We prove the assertion along the following two
reductions steps:
\vskip2pt
\noindent
\underbar{Step 1}. The valuation $w$ has rank one. In particular,
$\FF$ is dense in $\FF^h$.
\vskip2pt
Case a): $\ell=\chr(\FF).$ \ Then the $\ell$-elementary abelian extension
of both $\FF$ and $\FF^h$ are composita of $\ell$-cyclic extensions, 
all of which being Artin-Schreier extensions. Let $\FF^h(x')|\FF^h$
with $x'^\ell-x'=a'$ and $a'\in\FF^h$ by such an extension. Since $v$ has 
rank~$1$, hence $\FF$ is dense in $\FF^h$, one can choose $a\in \FF$ 
such that $v^h(a'-a)>0$. Then setting $a''=a'-a\in\FF^h$, or equivalently,
$a'=a''+a$, one has: First, the Artin-Schreier equation $T^\ell-T=a''$ 
has a solution in $x''\in\FF^h$ (\ha1because $v(a'')=v(a'-a)>0\ha1$). 
Second, $x'$ is a solution of $T^\ell-T=a''+a$ (by the additivity of $T^\ell-T$). 
Hence we conclude that 
$\FF^h(x')\subset\FF^h(x''+x)\subset\FF(x''\!,x)\subset\FF^h\FF^a$. 
\vskip0pt
Similarly, if $T^\ell-T=\oli a$ is an Artin-Schreier equations over $\FF w$, 
and $a\in\clO_w$ is a representative of $\oli a\in\FF w$ and $x$ is a 
solution of the equation $T^\ell-T=a$, it follows that the reduction $\oli x$ 
of $x$ is a root of $T^\ell-T=\oli a$.
\vskip2pt
Case b): $\ell\neq\chr(\FF w).$ \ Proceed as above, but using Kummer
type equations $T^\ell=a$, etc.
\vskip2pt
Case c): $\chr(\FF)=0$, $\chr(\FF w)=\ell$. \ Assertion~1) follows in the 
same way as in Case~b). For assertion~2), recall that $(Fw)^a|Fw$ 
is generated by the solutions of Artin--Schreier equations of the form 
$\wp(\oli u_0)=\oli a$ with $\oli a\in Fw.$ For a preimage $a\!\in\!\FF$ 
of $\oli a\!\in\!\FF w$, set $a=(b-1)/\pi^\ell\!$, where $\pi\!:=\zeta_\ell-1$ 
with $\zeta_\ell\in\mu_\ell$ a primitive root of unity. Recall that 
$\ell=\pi^{\ell-1}\epsilon$ in $\lvZ[\mu_\ell]$ 
with $\epsilon\in1+\pi\lvZ[\mu_\ell]\subset\clO_v^\times$ a principal
$v$-unit. Then the Kummer equation $(T+1)^\ell=b$ has roots 
$u\in\FF^a$ for each $b\in\FF$, that is
$u^\ell +\sum_{\ell>i>0}{\ell\choose i} u^i+1=b.$ Thus diving this 
equation by $\pi^\ell=\ell\pi\epsilon$ and setting $u=u_0\pi$, the 
equation satisfied by $u_0$ is:
\[\textstyle
u_0^\ell+\epsilon^{-1}\sum_{\ell>i>0}
    {1\over\ell}{\ell\choose i}\pi^{i-1}u_0^i=(b-1)/\pi^\ell\,.
\leqno{\indent(u_0)}
\]
On the other hand, the equation $(u_0)$ specializes 
to $\wp(\oli u_0)=\oli a$, that is $\oli u_0\in\FF^aw^a$, as claimed. 
\vskip4pt
\noindent
\underbar{Step 2}. The valuation $w$ has finite rank $d=\rk(v)=\Krdim(\clO_v)<\infty$.
We make induction on $d$. Namely, let $w_1\leqslant w$ be the minimal 
non-trivial coarsening of $w$, and $w_0=w/w_1$ the resulting valuation 
of the residue field $\FF_0=\FF w_1$. Then $w_1$ has rank one and $w_0$ 
has rank $d-1<d$. Hence by the induction hypothesis and Step~1, the
assertions~1),~2) hold for both $w_1$ and $w_0$. From this instantly
follows the same for $w$ (by the fuctoriality of Hilbert Decomposition 
for valuations).
\vskip4pt
\noindent
\underbar{Step 3}. \ Let $\FF=\cup_\alpha\FF_\alpha$ be the inductive
union of its finitely generated subfields with $\mu_\ell\subset\FF_\alpha$
provided $\chr(\FF)\neq\ell$. Then considering $\FF^a_\alpha|\FF_\alpha$, 
it follows that $\FF^a=\cup_\alpha\FF^a_\alpha$ and the extension of 
valued fields $\FF^a|\FF,w^a|w$ is the inductive limit of the system of
valued fields $\FF^a_\alpha|\FF_\alpha,w^a_\alpha|w_\alpha$. And since 
$\FF^h=\cup_\alpha F^h_\alpha$ and $\FF^a w^a=\cup_\alpha
\FF^a_\alpha w^a_\alpha$, by mere definitions one has that assertions~1),~2)
from the Fact hold iff they hold for each $\FF^a_\alpha|\FF_\alpha$
endowed with $w^a_\alpha|w_\alpha$. 
\vskip0pt
On the other hand, since $\FF_\alpha$ is finitely generated, the valuation 
$w_\alpha$ has finite rank (bounded by the Kronecker dimension of $\FF_\alpha$). 
Hence assertions 1),~2) hold for each $\FF^a_\alpha|\FF_\alpha,\,w^a_\alpha|w_\alpha$ 
by the discussion at Step~1. Hence conclude that assertions~1),~2) hold 
for $\FF^a|\FF, w^a|w$.
\end{proof}
Recall that via the canonical exact sequence 
$1\to \GGa\FF\hor{\bm\imath} G(\FF^a|\EE)\hor{p_\EE} G(\FF|\EE)\to 1$ 
the group $\Gal(\FF|\EE)$ acts {\it canonically} by conjugation 
on subsets $\Sigma$ of the three groups above, say 
\vskip5pt
\centerline{$\gx(\Sigma)\!:=\gx\,\Sigma\,\gx^{-1}$ \ \ 
for \ $\gx\in G(\FF|\EE)$ and
$\Sigma\subset G(\FF^a|\FF),\ G(\FF^a|\EE),\ G(\FF|\EE)$,}
\vskip5pt
\noindent
compatibly with the morphisms $\bm\imath,p_\EE$. We fix the 
above notation for this action throughout.
\vskip5pt
Let $\clV\subset\Val{}{(\EE)}$ be a non-empty set. For $\vv\in\clV$, 
let $\ww^a|\ww|\vv$ be the prolongations of $\vv\in\clV$ to 
$\FF^a|\FF|\EE$, and $\clV_{\vv}(\FF)\subset\clV(\FF)$ denote 
the subset of prolongations of $\vv\in\clV$ and of $\clV$ to 
$\FF$. And to fix notation, recall that $G(\FF|\EE)$ acts on 
$\clV_\vv(\FF)$ by $\gx(w)\!:=w\circ\gx^{-1}=:w^\gx\!.$ 
\vskip5pt
By Hilbert decomposition theory for valuations, one has:
Since $\Gal(\FF^a|\FF)$ is abelian, it follows that 
$T_{\ww}^a\!:=T_{\ww^a|\ww}\leqslant Z_{\ww^a|\ww}=:\!Z_{\ww}^a$ 
and $T_{\ww^a|\vv}\leqslant Z_{\ww^a|\vv}$ depend on $\ww$ 
only and not on the concrete prolongation~$\ww^a|\ww$. 
And for $\ww\in\clV_{\vv}(\FF)$, $\gx\in\Gal(\FF|\EE)$ one has: 
\vskip5pt
\centerline{$Z^a_{\gx(\ww)}=\gx\, Z_{\ww}^a\,\gx^{-1}=\gx(Z^a_w)$ 
\ and \ $Z_{\gx(w)|\vv}=\gx\, Z_{\ww|\vv}\,\gx^{-1}=\gx(Z_{w|v})$.}
\begin{definition/remark} 
%
%
In the above context and notation consider/define:
\label{indep}
$\ha0$
\vskip0pt
\itm{25}{
\item[1)] Let $Z^a_{\ww}\neq1$ for some $\ww\in\clV_{\vv}(\FF).$ 
Then $Z^a_{\ww'}\neq1$ for all $\ww'\in\clV_{\vv}(\FF)$, and 
we indicate this by writing $\clZ^a_{\vv}(\FF)\neq1$. And if 
$\clZ^a_{\vv}(\FF)\neq1$ for all $\vv\in\clV$, we write 
$\clZ^a_{\clV}(\FF)\neq1$. 
\vskip2pt
\item[2)] (See/compare with\ha3\cite{P1}, Def 5.) We say 
that $\vv\in\clV$ equals its {\it $\FF^a|\FF$-\ha1core\/} if for any 
proper coarsening $v_1<v$ and $w_1\in\Val_{v_1}(\FF)\!$ one
has $Z^a_{w/w_1}\neq1.$ We notice the following:
\itm{20}{
\vskip2pt 
\item[a)] Since $Z^a_{w/w_1}=Z^a_w/T^a_{w_1}$, one has:
$Z^a_{w/w_1}\neq1$ iff $T^a_{w_1}< Z^a_w$ strictly. 
\vskip2pt
\item[b)] For every $v\in\Val{}(\FF)$ there is a 
valuation $v^0\in\Val{}(\FF)$ which is maximal with the properties: 
$v^0\leqslant v$ and $v^0$ equals its $\FF^a|\FF$-core.
}
\vskip2pt
\item[3)] We say that $\clV$ equals its $F^a|F$-core if each 
$v\in\clV$ does so. For instance, this is the case if all $v\in\clV$ 
have rank one and $Z^a_w\neq1$ for $w\in\clV_v$.
}
\end{definition/remark}
In the above context and notation, let 
$v_i\in\Val{}(\EE)$, $i=1,2$ be given, and $w^a_i|w_i|v_i$ be 
prolongations of $v_i$ to $\FF^a|\FF|\EE.$ 
Setting $v=\min(v_1,v_2)$ and $w=\min(w_1,w_2)$, it follows 
that $w|v$ prolongs $v$ to $\FF|\EE$, and 
setting $\oli v_i\!:=v_i/v$, $\oli w_i\!:=w_i/w$, $\oli w^a_i\!:=w^a_i/w$, 
one has: $\oli w_i^a|\oli w_i|\oli v_i$ prolong $\oli v_i$
to $\FF^a w^a_i|\FF w_i|\EE v_i$ and further, $v_i=\oli v_i\circ v$, 
$w_i=\oli w_i\circ w\!$, $w^a_i=\oli w^a_i\circ w^a\!$ for $i=1,2$. 
And for $i=1\vrg2$, one has a commutative diagram of exact sequences:
\[
\begin{matrix}
1\,\,\to& T_{w^a|v}&\hor{} &Z_{w^a|v}&\hor\pr& G_{w^a|v}&\to\,\,1&\cr
          &\big|\hb2\big| &&\uparrow&         &\uparrow&          &\cr
1\,\,\to& T_{w^a|v}&\hor{} &Z_{w^a_i|v_i}&\hor{\pr}&Z_{\oli w^a_i|\oli v_i}&\to\,\,1&
\end{matrix}
\leqno{\indent(\dag)}
\] 
\begin{proposition}
\label{fkt01}
%
%
In the above notation, suppose that any two distinct valuations 
$v_1,v_2\in\clV$ are not comparable, $\clV$ equals its $\FF^a|\FF$-core, 
and $\clZ^a_{\clV}(\FF)\neq1.$ Then for any valuations 
$v,\, v_1,\,v_2\in\clV$ and $w\in\clV_v(\FF)$, $w_i\in\clV_{v_i}(\FF)$,
$w^a_i\in\clV_{v_i}(\FF^a)$, $i=1,2$, the following hold:
\vskip2pt
\itm{25}{
\item[{\rm1)}] Suppose that $w_1\neq w_2.$ Then $w_1,w_2$ are 
not comparable, and setting $w\!:=\min(w_1,w_2)$, one has:
$Z^a_{w_1}\cap Z^a_{w_2}=T^a_{w}$, and in
particular, $Z^a_{w_1}\neq Z^a_{w_2}$. Therefore,
\vskip5pt
\centerline{$\clV_{v}(\FF)\to\clZ^a_{v}(\FF)$, $w\mapsto Z^a_{w}$ is an 
isomorphism of topological $\Gal(\FF|\EE)$-spaces.}
\vskip5pt
\item[{\rm2)}] For $\gx\in\Gal(\FF|\EE)$ one has: \ $\gx\in Z_{w|v}$ \ iff 
\ $\gx(Z^a_{w})=Z^a_{w}$. Therefore,
\vskip5pt
\centerline{$\clZ^a_{\vv}(\FF):=\{Z^a_{\ww}|\,
\ww\in\clV_{\vv}(\FF)\} \to \{Z_{\ww|\vv}|\,\ww\in\clV_{\vv}(\FF)\}
=:\clZ_{\vv}(\FF)$, \ \ $Z^a_{\ww}\mapsto Z_{\ww|\vv}$.}
\vskip5pt
\item[] is a well defined surjective projection of 
topological $\Gal(\FF|\EE)$-spaces.
}
\noindent
Hence items \ {\rm 1), 2)} above give a group theoretical recipe 
to recover the $\Gal(\FF|\EE)$-space isomorphism
\[
\clV(\FF)\to \ZZ_{\clV}(\FF|\EE)\!:=\big\{(Z^a_{w},Z_{w|v})\,|\,v\in\clV(\EE),
w\in\clV_v(\FF)\big\}, \ \ w\mapsto(Z^a_{w},Z_{w|v})
\] 
from $\Gal(\FF^a|\EE)\to\Gal(\FF|\EE)$ endowed~with~$\clZ^a_{\clV}(\FF)$. 
\end{proposition}
\begin{proof} To 1): First, by contradiction, let $w_1,w_2\in\clV_{v_i}(\FF)$ be 
comparable, say $w_1<w_2.$ Then  $\FF|\EE$ being algebraic
implies $v_1<v_2$, contradiction! Hence $w<w_1,w_2$ strictly, 
and $\oli w_i=w_i/w$ are two independent valuations on the 
residue field $\FF w$. Further, by Fact~\ref{basic-fact}, 
$\FF^a w^a|\FF w$ is the maximal $\ell$-elementary abelian extension of 
$\FF w$, hence $G_{w^a|w}\!:=G(\FF^a w^a|\FF w)=G\big((\FF w)^a|\FF w\big)$. 
Further, by the commutative diagram~$(\dag)$ above,
$Z_{\oli w^a_i|\oli w_i}=Z_{w^a_i|w_i}/T_{w^a|w}\subset G_{w^a|w}$ 
is the decomposition group of $\oli w^a_i|\oli w_i$ in $G_{w^a|w}
=G\big((\FF w)^a|\FF w\big)$ for $i=1,2$. In particular, since 
$\oli w_1,\oli w_2$ are independent valuations of $\FF w$, it
follows that by Lemma~\ref{indepval} below that 
$Z_{\oli w^a_1|\oli w_1}\cap Z_{\oli w^a_2|\oli w_2}=1$.
Hence since $T_{w^a|w}=\ker(Z_{w^a_i|w_i}\to Z_{\oli w_i^a|\oli w_i}),$
finally getting $Z_{w^a_1|w_1}\cap Z_{w^a_2|w_2}=T_{w^a|w}.$
\vskip2pt
To 2): By mere definitions one has that $\sigma\in Z_{w|v}$ iff 
$w^\sigma=w$. First, for the direct implication, if $\sigma\in Z_{w|v}$, then 
$w=w^\sigma$, thus $Z^a_{w}=Z^a_{w\,^\sigma}=\sigma(Z^a_{w}).$
For the converse implication, suppose that $w_1\!:=w\neq w^\sigma=:\!w_2$. 
Then by assertion~1) above one has $Z^a_{w_1}\neq Z^a_{w_1},$
that is $Z^a_{w}\neq Z^a_{w^\sigma}.$ 
\vskip2pt
Finally, the last assertion is an immediate consequence 
of the discussion above. 
\end{proof}
%
%
\begin{lemma}
\label{indepval}
Let $\NN$ be a field with $\mu_\ell\subset\NN$ provided $\ell\neq\chr(\NN)\!$, and
$\GGa\NN\!:=G(\NN^a|\NN)$ be the Galois group of the maximal $\ell$-elementary
abelian extension $\NN^a|\NN.$ If $\euv_i\in\Val{}(\NN),$ $i=1,2$ are independent 
valuations, their decomposition groups $Z^a_{\euv_i}\subset\GGa\NN$ 
satisfy $Z^a_{\euv_1}\cap Z^a_{\euv_2}=1.$
\end{lemma}
\begin{proof} \ Set $U_i=1+\eum_{\euv_i},$ $i=1,2$. We analyze 
separately the cases:
\vskip5pt
\underline{Case 1}. \ $\chr(\NN)\neq\ell$. \ By Hensel Lemma, for all 
$u_i\in U_i$, $i=1,2$ one has: $T^\ell-u_i\in \NN[T]$ splits in linear 
factors over the Henselization of $\NN$ with respect to $\euv_i$. 
Therefore, $\euv_i$ is totally split in $\NN_i\!:=\NN[\root\ell\of U_i]$, 
and equivalently, $\NN_i$ is contained in the fixed field of 
$Z^a_{\euv_i}\subset\GGa\NN$ in $\NN^a\!.$ On the other 
hand, since $\euv_1,\euv_2$ are independent, one has 
$U_1\cdot U_2=\NN^\times\!,$ hence $\NN^a=\NN_1\NN_2$. 
Conclude by Kummer theory that $Z^a_{\euv_1}\cap Z^a_{\euv_2}=1$. 
\vskip2pt
\underline{Case 2}. \ $\chr(\NN)=\ell.$ Let $\wp_u^{-1}(0)$ denote 
the set of roots of the Artin--Schreier polynomial $\wp_u(T)=T^\ell-T-u.$
Then by Hensel Lemma, $\wp^{-1}_{u_i}(0)$ is contained in the 
decomposition field of $\euv_i$ in $\NN^a$ for any $u_i\in\eum_i$, 
$i=1,2.$ Hence so 
does $\NN_i\!:=\NN[\wp_{u_i}^{-1}(0)]_{u_i\in\eum_{\euv_i}}.$ On the 
other hand, since $\euv_1,\euv_2$ are independent, one has 
$\NN=U_1-U_2=\eum_{\euv_1}-\eum_{\euv_2}$, hence 
$\NN^a=\NN_1\NN_2$. Conclude by Artin--Schreier theory that 
$Z^a_{\euv_1}\cap Z^a_{\euv_2}=1$. 
\end{proof}
%
%
%
\subsection{\bf Commuting liftability} 
{ See \nmnm{Topaz}~\cite{To1} (and 
\nmnm{Pop}~\cite{P1}, section 3) for more details.\/}
\vskip5pt
Commuting liftability was developed in stages over several years
by several people, {\sc Ware, Jacob, Arason--Elman--Jacob, 
Bogomolov, Koenigsmann, Bogomo\-lov--Tschinkel}, 
culminating with contributions by\ha5{\sc Topaz},~\cite{To1}
(where more literature can be found). The essential property
and consequence of commuting liftability is that it relates
in an intimate way to (arithmetically significant) valuations of 
the fields in discussion, see Theorem~\ref{Topaz}~(after~\cite{To1}) 
below.   
\vskip5pt
Let $\FF$ be a field with $\chr(\FF)\neq\ell$,
$\mu_\ell\subset \FF$, and $\FF^a|\FF$ be the maximal
$\lvZ/\ell$ elementary abelian extension. For
a valuation $w$ of $\FF$, set $\FF^D\!:=\FF[\root\ell
\of{1+\eum_w}\,]$, $\FF^I:=\FF[\root\ell\of
{\clO^\times_w}\,]$. The groups $I_w\leqslant D_w$ 
below are called the {\it minimized inertia\/}$/\!${\it decomposition 
groups\/} of $w$:
\[
I_w:=\Gal(\FF^a|\FF^I)=\Hom(\FF^\times\!/\clO_w^\times,\mu_\ell)
\leqslant\Hom\big(\FF^\times\!/(1+\eum_w),\mu_\ell\big)
=\Gal(\FF^a|\FF^D)=:D_w.
\]

We notice that under valued field extensions the minimized 
inertia/decomposition groups behave as follows. Let 
$(\FF\!,w)\,|\,(\NN,\euw)$ is an extension of valued fields, 
thus $\clO_\euw^\times=\clO_w^\times\cap\NN$ 
and $1+\eum_\euw=(1+\eum_w)\cap\NN.$ Then supposing 
that $\mu_\ell\subset\NN$, by mere definitions one has:
\begin{fact}[{\bf Functoriality}]
\label{fkt3-4}
%
%
%
{\it The canonical projection $p^a:\GGa\FF\to\GGa\NN$ gives
rise canonically to embeddings $p^a(I_w)\subset I_\euw$ and
$p^a(D_w)\subset D_\euw$. 
Moreover, if $\clO_w^\times/\clO_\euw^\times$  and
$(1+\eum_w)^\times\!/(1+\eum_\euw)^\times$ have no $\ell$-torsion, 
then $p^a(I_w)=I_\euw$ and $p^a(D_w)=D_\euw.$ 
}
\end{fact}
%
%
\begin{fact} [{\bf Basics}]
\label{fkt4}
{\it In the above notation, the following hold:
\vskip2pt
\itm{25}{
\item[{\rm1)}] $I_w\cong\Hom(w\FF/\ell,\mu_\ell)$ 
and $D_w/I_w\cong\Hom(\FF w\ha1^\times\hb3/\hb1\ell,\mu_\ell)$.
Hence one has:
\vskip5pt
\centerline{$\,I_w=1\,$ iff $\,w\FF$ is $\ell$-divisible, and 
$\,I_w=D_w\,$ iff $\,\FF w\ha1^\times$ is $\ell$-divisible.}
\vskip2pt
\item[{\rm 2)}] If $\,\chr(\FF w)\!\neq\!\ell$, then 
$T^a_w\hb2=\!I_v\subset D_w\hb2=Z^a_w$. Further,
$(\FF w)^a\!=\FF^a w^a\!$, thus $\,\GGa{\FFw}=Z^a_w/T^a_w$.
\vskip2pt
\item[{\rm 3)}] If $\,\chr(\FF w)\!=\!\ell$, then 
$\,I_w\!\subset\! T^a_w\,$ and $\,D_w\hb3\subset\! Z^a_w$.
}}
\end{fact}
\begin{proof} Everything follows by mere definitions, 
Pontryagin duality, and Kummer~theory from the exact 
sequences $1\to\clO_w^\times\to \FF^\times\!\to w\FF\to0$ 
and $1\to(1+\eum_w)^\times\!\to\clO_w^\times\to \FF w^\times\to1$.
\end{proof} 
%
%
\begin{fact}
\label{red-inr-dec-coars}
{\it In the above context, let $w_1\leqslant w_2$ in $\Val(FF).$  
Then $I_{w_1}\subset I_{w_2}$ and $D_{w_1}\supset D_{w_2}$.} 
\end{fact}
\begin{proof} This is obvious, because $w_1\leqslant w_2$ iff the 
following equivalent conditions are satisfied:
\vskip2pt
(i) $\,\clO_{w_1}\supset\clO_{w_2}$; \ \ \ \ (i)$'$ 
$\,\eum_{w_1}\subset\eum_{w_2}$; \ \ \ \ (ii) $\,\clO^\times_{w_1}\supset 
\clO^\times_{w_2}$; \ \ \ \ (ii)$'$ $\,1+\eum_{w_1}\subset1+\eum_{w_2}.$
\end{proof}
\vskip5pt
In the above notation, recall the canonical exact sequence
$1\to\Delta_\FF\to\GGc\FF\to\GGa\FF\to 1.$ For $\sigma\in\GGa\FF$ 
and $\Sigma\subset\GGa\FF$, denote by $\tlsgm\in\GGc\FF$ 
preimages of $\sigma$ and by $\Sigma^c\subset\GGc\FF$ 
the preimage of $\Sigma$. Recall the following canonical maps 
in this context, see e.g.~\cite{To1} for these basic facts:
\vskip2pt
\begin{itemize}[leftmargin=25pt]
\item[-$\ha2$] The bilinear map 
$\psi:\GGa\FF\times\GGa\FF\to\GGac_\FF$, defined 
by $(\sigma,\tau)\mapsto[\tlsgm\!,\tltau]$.
\vskip2pt
\item[-$\ha2$] The linear map $\beta:\GGa\FF\to\GGac_\FF$,
$\sigma\mapsto\sigma^\beta:=(\tlsgm)^\ell\!$.
\end{itemize}
This being said, we recall basics about {\it commuting liftability,\/} 
see \nmnm{Topaz}~\cite{To1} for details. 
%
%
\begin{definition/remarks}
\label{defrem1}
A pair of elements $\sigma,\tau\in\GGa\FF$ is {\it independent,\/} 
if $\sbgg{\sigma,\tau}\cong(\lvZ/\ell\ha1)^2$. Given independent 
$\sigma,\tau\in\GGa\FF$ and liftings $\sigma^c\!,\tau^c\in\GGc\FF\!$, 
we say that:
\vskip2pt
\begin{itemize}[leftmargin=25pt]
\item[1)] $\sigma,\tau$ are {\it commuting liftable\/} (c.l.) if 
$\sigma,\tau$ satisfy the equivalent conditions: 
\vskip4pt
\centerline{(i) $\exists$ $\tlsgm\!,\tltau$ such that 
$[\tlsgm,\tltau]\in\langle\sigma^\beta\!,\tau^\beta\rangle$; \
(ii) $\,\forall$ $\tlsgm\!,\tltau$ one has 
$[\tlsgm,\tltau]\in\langle\sigma^\beta\!,\tau^\beta\rangle$.}
\vskip4pt
\item[2)] $\sigma,\tau\in\GGa\FF$ is called {\it c.l.\ pair,\/} 
if $\sigma,\tau$ satisfy the equivalent conditions:  
\vskip4pt
\centerline{(i) $\exists$ $\tlsgm\!,\tltau$ such 
that $[\tlsgm,\tltau]\in\sbgg{\sigma^\beta}$; \
(ii) $\,\forall$ $\tlsgm,\tltau$ 
one has $[\tlsgm,\tltau]\in\sbgg{\sigma^\beta}$.}
\vskip4pt
\item[] {\bf Note} the following: Let $\sigma,\tau\in\GGa\FF$ 
be independent and c.l. Then the following hold:
\vskip2pt
\begin{itemize}[leftmargin=25pt]
\item[a)] If $\sigma_1,\tau_1\in\sbgg{\sigma,\tau}$ are
independent, then $\sbgg{\sigma,\tau}=\sbgg{\sigma_1,\tau_1}$,
and $\sigma_1,\tau_1$ is c.l.
\vskip2pt
\item[b)] There exists $1\neq\sigma_1\in\sbgg{\sigma,\tau}$
such that $[\sigma_1,\tau_1]\in\sbgg{\sigma_1^\beta}$
for all $\tau_1\in\sbgg{\sigma,\tau}$. 
\vskip2pt
\item[c)] For $k\in\lvZ$ with $(k,\ell)=1$ one has:
$\sigma^k\!,\tau^k$ is a c.l.\ pair, provided $\sigma,\tau$ 
is a c.l.\ pair.
\vskip2pt
\item[d)] One has: $\sigma,\tau$ and $\tau,\sigma$ are both
c.l.\ pairs if and only if $[\tlsgm,\tltau]=1$.
\end{itemize}
\vskip2pt
\item[3)] Let$I\subset D\subset\GGa\FF$ be subgroups. We say
that $\clp ID$ is a {\it c.l.\  pair\/} (of groups) in $\GGa\FF$ if 
$I\neq1$, $D$ is non-cyclic, and all independent pairs $\sigma,\tau$ 
with $\sigma\in I$, $\tau\in D$ are c.l.\ pairs. In particular, if so, then 
$\sigma,\ha1\sigma'$ is a c.l.\ pair for all independent 
$\sigma,\,\sigma'\in I.$ 
\vskip2pt
\item[] We say that $D$ is c.l., if $\clp DD$ is a c.l.\ pair, i.e., 
$\tau,\sigma$ is a c.l.\ pair for all independent $\sigma,\tau\in D.$ 
\vskip2pt
\item[] {\bf Note} that if $\sigma,\ha1\tau\in\GGa\FF$ define 
a c.l.\ha3pair, then in the notation from 2),~b) above,  one has: 
\[
I:=\sbgg{\sigma_1}\leqslant\sbgg{\sigma,\tau}:=D \ \ 
        \hbox{is c.l.\ha3pair (of groups)\,.}
\]
\item[4)] For $\clp ID$ c.l.\ha2pair, the following hold: 
\vskip2pt
\begin{itemize}[leftmargin=25pt]
\item[a)] There exists a unique maximal 
$I_D\subset\GGa\FF$ such that $\clp{I_D}{DI_D}$ 
is c.l.\ha2pair, hence $I\subset I_D$.
\vskip2pt
\item[b)] There exists a unique maximal $D_I\subset\GGa\FF$ 
such that $\clp I{D_I}$ is c.l.\ha3pair, hence $D\subset D_I$.
\end{itemize}
\vskip2pt
\item[$\bullet$] Finally, a c.l.\ha3pair $\clp ID$ is called 
{\it maximal,\/} if $I=I_D$, $D=D_I$. We notice the following:
\vskip2pt
Starting with a c.l.\ pair $I\leqslant D$, one has: 
$\clp{I_D}{D_{I_D}}$ and $\clp{I_{D_I}}{D_I}$ 
are maximal.
\vskip2pt
\item[5)] Let $\phi^a\in\Aut(\GGa\FF)$ be the automorphism 
which lifts to an automorphism $\phi^c\in\Aut(\GGc\FF).$  
Then for every pair of subgroups $I\subset D\subset\GGa\FF$ one has:
\vskip2pt
\itm{20}{
\item[a)] $\clp ID$ is a maximal c.l.\ pair in $\GGa\FF$ iff 
$\phi(\clp ID)\!:=\big(\clp{\phi(I)}{\phi(D)}\big)$ is so.
\vskip2pt
\item[b)] If $\clp ID$ is a maximal c.l.\ pair, then $\phi(\clp ID)=(\clp ID)$ 
iff $\phi(I)=I$ \ iff \ $\,\phi(D)=D.$
}
\vskip2pt
\item[] (Indeed, $\clp ID$ is a maximal c.l.\ pair \ iff \ 
$\clp{\phi(I)}{\phi(D)}$ is a maximal c.l.\ pair, etc.)
\end{itemize}
\end{definition/remarks}
\noindent
For the next basic fact, see e.g.\ {\sc Pop}~\cite{P1}, Section~3, 
and {\sc Topaz}~\cite{To1}.
%
%
%
\begin{fact}
\label{fkt5}
{\it In the above notation, suppose that $w\FF$ 
is not $\ell$-divisible, and $\FF^\times\!/(1+\eum_w)$
is non-cyclic, or equivalently, $I_w\neq1$ and 
$D_w$ is non-cyclic. The following hold:
\vskip2pt
\itm{25}{
\item[{\rm1)}] $\clp{I_w}{D_w}$ is a c.l.\ pair. In particular, 
$\clp{I_w}{I_{D_w}}$ and $\clp{D_w}{D_{I_w}}$.
\vskip2pt
\item[{\rm2)}] Moreover, if $\,w$ has rank one, then 
$\clp{I_{D_w}}{D_w}$ is a maximal c.l.\ pair. In particular, 
in this case, every group automorphism of $D_w$ 
defined by some $\sigma\in\GGa\FF$ maps 
$I_{D_w}$ into itself.
}
}
\end{fact}
%
%
%
\begin{notations/remark}[{cf.\ha2{\sc Topaz}~\cite{To1}, \S1.2 for some details}] 
\label{notarem1}
In the above context, consider:
\vskip2pt
\itm{25}{
\item[1)] Let $\clW_\FF$ be the set of valuations $w\in\Val{}(\FF)$ which  
satisfy:
\vskip2pt
\itm{20}{
\item[(i)$\hb1$] If $w_1<w$ strictly, then the value group of
$w/w_1$ is not $\ell$-divisible, i.e., $I_{w/w_1}\neq1$. 
\vskip2pt
\item[(ii)] If $w< w_2$ strictly, then $D_{w_2}=D_{w}$ implies 
$I_{w_2}= I_{w},$ i.e., $I_{w_2/w}=1$.
}
\vskip2pt
\item[$\bullet$] Notice that {\it every $w\in\clW_\FF$ equals its 
$\FF^a|\FF$-core.\/} Indeed, if $w\in\clW_\FF$ and $w_1<w$ 
strictly, then $I_{w/w_1}\neq1$, implying that $\FF^a w_1|\FF w_1$ is 
not purely inseparable.
\vskip2pt
\item[2)] Let $\clP_\FF$ be the set of \textit{\textbf{maximal}} c.l.\ pairs 
$\,I\leqslant D$ in $\GGa\FF$ with $I\neq1,\, D$ not cyclic, and~denote:
\vskip2pt
\centerline{$\clI_\FF\!:=\{I\subset\GGa\FF\,|\,\exists\,
\clp ID \hbox{ in } \clP_\FF\},$  \ \ \
$\clD_\FF\!:=\{D\subset\GGa\FF\,|\,\exists\,
\clp ID \hbox{ in } \clP_\FF\}.$}
\vskip3pt
\item[$\bullet$] Notice that given $\,I\leqslant D$ in $\clP_\FF,$ {\it each 
$I$ and $D$ individually determine the c.l.\ pair $I\!\leqslant\!D.$\/} 
\item[] Indeed, by by Definition/Remarks~\ref{defrem1},~4), one has  
both $D=D_I$ and $I=I_D.$ 
\vskip2pt
\item[$\bullet$]
In particular, the maps $\clP_\FF\to\clI_\FF$,
$\clp ID\mapsto I$ and $\clP_\FF\to\clD_\FF$,
$\clp ID\mapsto D$ are bijective.
}
\end{notations/remark}
%
%
%
\begin{theorem}[\ha1{cf.\ha2{\sc Topaz}~\cite{To1}, {\rm Thm 1,\ha2(1) \&\
Thm 6, for $N\!=\!n=\!\scalebox{.9}[1]{1}=$\scalebox{.9}[.9]{\bf R}(1)\ha1}}]
\label{Topaz}
The following hold:
\vskip2pt
\itm{25}{
\item[{\rm1)}] For $\,w\,$ in $\,\clW_\FF,$ there is $\clp ID$ in 
$\clP_\FF$ such that $D=D_w$ and the following hold:
\vskip2pt
\itm{20}{
\item[{\rm a)}] $D$ is a c.l.\ subgroup in $\GGa\FF$ iff $D/I_w$ is cyclic.
\vskip2pt
\item[{\rm b)}] If $D$ is not a c.l.\ subgroup in $\GGa\FF$, then
$D_w=D$ and $I_w=I.$ Moreover, this is so if $D/I$ is not cyclic,
i.e., if $\,\FF w^\times\hb3/\ell$ is not cyclic.
}
\vskip2pt
\item[{\rm2)}] For $\clp ID$ in $\clP_\FF,$ there is $\,w\in\clW_\FF$ 
satisfying the condition from $1)$ above. 
}
\end{theorem}
%
%
%
\begin{remark}
\label{unique}
Note that in Theorem~\ref{Topaz} above, both $w$ in $\clW_\FF$ 
and $\,I\leqslant D\,$ in $\clP_\FF$ are {\it unique corresponding 
to each other.\/} \ {\bf Notation}: $w\leadsto (\clp ID)^w\leadsto I^w\!,D^w\!,\,$ 
resp.\ $\clp ID\leadsto I, D\leadsto w^I\!,w^D\!.$
\end{remark}
\begin{proof}
Uniqueness of $I\!\leqslant\!D$:
Let $\clp{I_i}{D_i}\leadsto w$, $i=1,2$. Then $D_i=D_w$ and $\clp{I_i}{D_w}$,
$i=1,2$ are both c.l.\ pairs, hence so is $\clp{I_1I_2}{D_w}$. 
And $\clp{I_i}{D_w}$ being maximal implies $I_1=I_1I_2=I_2$. 
\vskip2pt
Uniqueness of $w$: \ By contradiction, let $w_1,w_2\leadsto (\clp ID),$
$w_1\neq w_2.$ Then $D_{w_i}=D$ and $I_{w_i}\subset I$, and $w_1,w_2$ 
are not comparable by Notations/Remark~\ref{notarem1},~1). Set
$w=\min(w_1,w_2)$, hence and $\oli w_i=w_i/w$ are non-trivial.  
Then letting $\pr:Z^a_{w}\to G_{w^a|w}$ be the canonical projection, 
one has $1\neq I_{\oli w_i}=I_{w_i}/I_{w}=pr(I_{w_i}),$ $i=1,2.$ On the 
other hand, one has: 
\vskip5pt
\centerline{$1\neq I_{\oli w_i}\subset\pr(I)\subset \pr(D)=
    D_{\oli w_i}\subset Z^a_{\oli w_i}\subset G_{w^a|w}$, \ hence
    $1\neq \pr(I)\subset Z^a_{\oli w_1}\cap Z^a_{\oli w_2}$.}
\vskip5pt
\noindent
Since $\oli w_1,\oli w_2$ are independent, this is a contradiction by
Proposition~\ref{fkt01}. 
\end{proof}
%
%
%
\subsection{\bf Commuting liftability and Galois action}$\ha0$
In the above context, let $\FF|\EE$ be a Galois extension with 
$\mu_\ell\subset\FF$ and Galois group $\Gal(\FF|\EE)$. Recall that  
$G(\FF|\EE)$ acts on the spaces $\clP_\FF,\clD_\FF,\clI_\FF$ 
(by conjugation) and on $\clW_F$ by $\gx(w)=w\circ\gx^{-1}\hb3$,
$\gx\in G(\FF|\EE).$ 
Further, the $G(\FF|\EE)\!$-actions are compatible with the previous 
constructions/introduced objects in the following sense: 
\vskip2pt
\itm{20}{
\item[-] If $\gx(w)=w,$ then $\gx\big(\clp{I_\euw}{D_\euw}\big)=
\big(\clp{\gx(I_\euw)}{\gx(D_\euw)}\big)$ and 
$\gx\big(\clp{I_{D_\euw}}{D_\euw}\big)=
\big(\clp{\gx(I_{D_\euw})}{\gx(D_\euw)}\big).$
\vskip2pt
\item[-] If $\clp ID\leadsto w,$ then 
$\gx(\clp ID\big)=\big(\clp{\gx(I)}{\gx(D)}\big)\leadsto \gx(w).$
}
\noindent
Further, by mere definitions one has that $D_w\triangleleft Z_{w^a|v}$ 
and $I_w,\, I_{D_w}\triangleleft \, Z_{w^a|v}.$ 
And if $\chr(\FF w)\neq\ell,$ then $D_w=Z^a_w,$
$I_w=T^a_w.$
%
%
\begin{notation}
\label{nota2}
In the above notation, recalling Remark/Notation~\ref{notarem1}, 
we denote:
\vskip2pt
\itm{30}{
\item[a)] $\clW_{\FF|\EE}\!:=\big\{w|v\,\big|\,w\in\clW_\FF,\ v=w|_\EE\big\}.$
\vskip2pt
\item[b)] $\clD_{\FF|\EE}\!:=\big\{D_w\leqslant Z_{w^a|v}\,\big|\,
w|v\in\clW_{\FF|\EE}\}.$ 
} 
\end{notation}
In the above context, 
for $\clp ID$ from $\clP_\FF$ and 
$\clp ID\leadsto w\in\clW_\FF$ relating to each other as
in~Theorem~\ref{Topaz}, set $v:=w|_\EE,$ thus $w|v\in\clW_{\FF|\EE},$ 
and let $w^a|w|v$ be the prolongations $\FF^a|\FF|\EE.$ 
%
%
\begin{proposition}
\label{fktTopaz}
In the above notation, the following hold:
\vskip4pt
\itm{25}{
\item[{\rm1)}] {\bf (Proposition~\ref{fkt01} revisited)}. \ For $\gx\in G(\FF|\EE)$ 
and $\clp ID\leadsto w|v$ 
the following hold:
\[
\gx(I)\!=\!I \ \ {\it iff} \ \ \gx(D)\!=\!D \ \ {\it iff} \ \ \gx(Z^a_w)\!=\!Z^a_w \ \
\ \ {\it iff} \ \ \gx(Z_{w^a|v})\!=\!Z_{w^a|v} \ \ {\it iff} \ \ 
\gx(w)\!=\!w \ \ {\it iff} \ \ \gx\!\in\!Z_{w|v}.
\]
\item[{\rm2)}] {\bf(Galois action)}. 
$\clW_\FF,\, \clW_{\FF|\EE},\, \clD_{\FF|\EE},\, \clP_\FF,\, \clD_\FF,\,\clI_\FF$ 
are $G(\FF|\EE)$-spaces, and the maps
\[
\clW_{\FF}\to\clW_{\FF|\EE}\to\clD_{\FF|\EE}\to\clP_\FF\to\clD_\FF\vrg\clI_\FF
\quad w\mapsto w|v\mapsto D_w\leqslant Z_{w|v} \mapsto
  \clp{I_{D_w}}{D_w}\mapsto I_{D_w}\vrg D_w
\]
are $G(\FF|\EE)$-isomorphisms, where the last two maps are
as defined in Notations$/\!$Remark~\ref{notarem1},~2).
}
\end{proposition}
\begin{proof} To 1): First, $\clp ID\in\clP_\FF$ and $w\in\clW_\FF$ 
relate to each other iff $D=D_w$ and $I=I_{D_w}.$ Next,
by Remark~\ref{unique}, since $w\in\clW_\FF$ equals its 
$\FF^a|\FF$-core, the last three equivalences follow from 
Proposition~\ref{fkt01}. Further, $\gx(w)=w$ iff $\gx(\clO_w^\times)=\clO_w^\times$
iff $\gx(1+\eum_w)=1+\eum_w$. Hence by the definitions of 
$I_w\subset D_w$ and Kummer theory one has: $\gx(w)=w$ 
$\Rightarrow$ $\gx(\FF^I)=F^I\!,$ $\gx(\FF^D)=\FF^D\!,$
and therefore, $\gx(w)=w$ $\Rightarrow$ 
$\gx(I_w)=I_w,\,\gx(D_w)=D_w.$ And further, by mere 
definitions, this implies $\gx(I_{D_w})=I_{D_w}.$ Hence 
it is left to show that $\gx(I)=I$ and/or 
$\gx(D)=D$ implies $\gx(w)=w$. First, since both $I$ and 
$D$ individually define $\clp ID$ uniquely, it is sufficient to prove 
one of the assertions, e.g., that $\gx(I)=I$ implies $\gx(w)=w$. 
This is more-or-less a reformulation of the last part of the proof of 
the Remark~\ref{unique} above, along the following lines: First, we 
notice that $\clW_\FF$ is invariant under automorphisms of $\FF$ 
(by mere definitions). Hence $w_1\!:=w\in\clW_\FF$ iff 
$w_2\!:=\gx(w)\in\clW_\FF$. And if so, by mere definitions 
on has $D_{w_2}=\gx(D)=D=D_{w_1},$ hence 
$I_{D_{w_2}}=\gx(I)=I=I_{D_{w_1}}.$ 
Conclude that $w_2=w_1$ by arguing as at the end of Remark~\ref{unique}.
\vskip2pt
To 2): Recall that by Remark~\ref{unique}, 
for $I\in\clI_\FF$ given, there is a unique $D\subset\GGa\FF$ 
with $\clp ID$ in $\clP_\FF$. Hence the stabilizer $\St_{G(\FF|\EE}(I)$ 
of $I$ in $G(\FF|\EE)$ stabilizes $D,$ i.e., stabilizes 
$I\!\leqslant\!D$. Since $w\in\clW_\FF$ with $D=D_w$ is unique, 
we conclude: $\St_{G(\FF|\EE}(I)=\St_{G(\FF|\EE)}(w)=Z_{w|v}$ for the 
unique $w|v\in\clW_{\FF|\EE}$ with $D_w=D$, $I_w\subset I.$ 
Similarly, starting with $w|v\in\clW_{\FF|\EE}$ 
and setting $D=D_{w}\in\clD_\FF,$ it follows
that $\St_{G(\FF|\EE)}(D)=Z_{w|v},$ etc. 
\end{proof}
\vskip2pt
This being said, we notice though that Theorem~\ref{Topaz} and 
Proposition~\ref{fktTopaz} above do not give conditions to ensure that the 
valuation~$\,w\,$ has $\chr(\FF w)\neq\ell.$ In the next section 
we discuss ---\ha1among other things\ha1--- this issue, which is 
essential for the proof of the main results of the paper. 
%
%
%
\section{Commuting liftability, field extensions, and sections}$\ha0$
Let $\EE|\LL$ be a regular field extension, $\tl\LL|\LL$ be a 
Galois extension, and $\tl\EE:=\EE\tl\LL$ be the compositum 
of $\EE$ and $\tl\LL$ over $\LL$ (which is well defined up
to $\LL$-isomorphism, because $\EE|\LL$ was a regular 
field extension). In particular, $\tl\EE|\EE$ is Galois such that
the canonical projection map $\tl\imath:G(\tl\EE|\EE)\to G(\tl\LL|\LL)$
is an isomorphism. For valuations $v\in\Val(\EE)$, let $\tl v|v$ 
denote prolongations of $v$ to $\tl\EE|\EE$, and $v_\LL\!:=v|_\LL$ 
and $\tl v_{\tl\LL}\!:=\tl v|_{\tl\LL}$ be the corresponding
restrictions, thus in particular,  $v_{\LL}=(\tl v_{\tl\LL})|_\LL$. Next 
let $\chr(\LL)\neq\ell$ and $\mu_\ell\subset\tl\LL$.
Recall that $\tl\LL^c|\tl\LL^a|\tl\LL|\LL$ and $\tl\EE^c|\tl\EE^a|\tl\EE|\EE$ 
are Galois extensions, giving rise to a commutative diagram 
below with exact rows and surjective columns,
$\tl\imath:G(\tl\EE|\EE)\to G(\tl\LL|\LL)$ an isomorphism, and
$\star=a$,$\,c.$ To complete the picture, let 
$s^a:G(\tl\LL^a|\LL)\to G(\tl\EE^a|\EE)$ be a section of 
$p^a:G(\tl\EE^a|\EE)\to G(\tl\LL^a|\LL)$, which is a.b.c.-liftable, 
i.e., $s^a$ lifts to a section $s^c$ of $p^c:G(\tl\EE^c|\EE)\to G(\tl\LL^c|\LL).$ 
Hence one has:
\[
\mathsurround=0pt
\begin{tikzcd}[column sep=scriptsize, row sep=scriptsize]
1\arrow[r]&\clG^\star_{\tl\EE} \arrow[r] \arrow[d, two heads, "\,\tl p^\star"]&
  G(\tl\EE^\star|\EE) \arrow[d, two heads, "\,p^\star"] \arrow[r, "pr_{\tl\EE}"] 
            &G(\tl\EE|\EE) \arrow[r] \arrow[d, "\tl\imath" ] &1 \\
1\arrow[r]&\clG^\star_{\tl\LL} \arrow[r] \arrow[u, bend left, dashed]
  &G(\tl\LL^\star|\LL) \arrow[r, "pr_{\tl\LL}"] \arrow[u, bend left, "s^\star" ]
      &G(\tl\LL|\LL) \arrow[r]&1
\end{tikzcd}
\]
\vskip-5pt
\noindent
{\bf Claim.} 
{\it The restriction $\tl s^\star\!:=s^\star|_{\clG^\star_{\tl\LL}}$ is a section 
of $\tl p^\star$ and $pr_{\tl\EE}\circ s^\star\circ pr_{\tl\LL}^{-1}$
 is the inverse map of $\,\tl\imath.$\/}
\vskip5pt
Indeed, 
one has $\tl\imath\circ pr_{\tl\EE}= pr_{\tl\LL}\circ \tl p^\star$ and 
$\tl\imath$ being an isomorphism implies $pr_{\tl\EE}\big(s^\star(g)\big)=1$ 
iff $pr_{\tl\LL}(g)=1.$ Hence $s^\star(g)\in\clG^\star_{\tl\EE}$ iff 
$pr_{\tl\EE}\big(s^\star(g)\big)=1$ iff $pr_{\tl\EE}(g)=1$ iff
$g\in\clG^\star_{\tl\LL},$ concluding that $s^\star$ is a section 
of $\tl p^\star\!.$ For the last assertions, for $h\in G(\tl L|L)$ and
$g\in\ G(\tl\LL^\star|\LL)$ one has that $g\in pr_{\tl\LL}^{-1}(h)$ iff 
$h=pr_{\tl\LL}(g)=\tl\imath\circ pr_{\tl\EE}\circ s^\star(g)$,
etc. (Note that $pr_{\tl\LL}^{-1}$ is just a correspondence, but 
$pr_{\tl\EE}\circ s^\star\circ pr_{\tl\LL}^{-1}$ is indeed single 
valued, hence a map.) The Claim is proved.
\vskip-10pt
\subsection{\bf $s^\star\!$-valuations arising from $\Val^1(\LL)$}$\ha0$
In the above context, let $\LL_1|\LL\hra\tl\LL|\LL$ be a fixed {\it finite\/} 
Galois subextension of $\tl\LL|\LL$ satisfying, among other things, 
$\mu_\ell\subset\LL_1.$
%
%
%
\begin{notation/remark}
\label{nota/rem3.1}
Let $\LL_1|\LL\hra\NN|\LL\hra\tl\LL|\LL$ denote finite (fixed) Galois 
subextensions of $\tl\LL|\LL$, and for $\euv\in\Val(\LL)$ let 
$\tl\euv|\euw|\euv_1|\euv$ be the prolongations of $\euv$ to 
$\tl\LL|\NN|\LL_1|\LL.$ Further, let $\Val^1(\LL)\subset\Val(\LL)$ be the 
set of valuations $\euv\in\Val(\LL)$ satisfying: 
\vskip4pt
\centerline{(i)$\,\euv\LL=\lvZ$; \ \ (ii)$\,\chr(\LL\euv)\neq\ell$; \ \
(iii)$\,{\LL_1\euv_1}^\times\!/\ell\neq1$; \ \ (iv) $\tl\euv|\euv$ 
are unramified.} 
\vskip4pt
Notice that for $\euv\in\Val^1(\LL)$, the prolongations $\euw|\euv_1|\euv$ 
satisfy (i), (ii), (iii), (iv), correspondingly, as well. [{\it Proof.\/} Everything 
is clear, but maybe assertion (iii), which follows from the fact 
that $[\NN\euw\!:\!\LL_1\euv_1]\leqslant [\NN\!:\!\LL_1]<\infty,$ 
and ${\LL_1\euv_1}^\times\hb3/\ell\neq1$ 
implies $\NN\euw^\times\hb3/\ell\neq1$ (because $\ell>2$), etc.] 
\vskip3pt
Further, since $\chr(\tl\LL\tl\euv)=\chr(\NN\euw)=\chr(\LL\euv)\neq\ell$ and 
$\lvZ=\euv\LL=\euw\NN=\tl\euv\tl\LL,$ by functoriality and 
basics of Hilbert decomposition theory, the following hold:  
\vskip2pt
1) $T_{\tl\euv^a|\euv}\!=\!T_{\tl\euv^a|\tl\euv}\!\cong\!\lvZ/\ell\!\cong
T_{\euw^a|\euw}\!=\!T_{\euw^a|\euv}$
and $T_{\tl\euv^c|\euv}\!=\!T_{\tl\euv^c|\tl\euv}\!\cong\!\lvZ/\ell^2\!\cong
T_{\euw^c|\euw}\!=\!T_{\euw^c|\euv}.$
Further, denoting by $G_{\tl\euv^\star|\euv}$ and $G_{\euw^\star|\euv}$
the Galois group of the residue field extensions
$\tl\LL\!^\star\tl\euv^\star|\LL\euv$, respectively
$\NN^\star\euw^\star|\LL\euv$, one has
$Z_{\tl\euv^\star|\euv}=T_{\tl\euv^\star|\euv}\rsdp G_{\tl\euv^\star|\euv}$
and $Z_{\euw^\star|\euv}=T_{\euw^\star|\euv}\rsdp G_{\euw^\star|\euv}$,
the action being in both cases by the $\ell$-adic cyclotomic character
of $G_{\tl\euv^\star|\euv}$, respectively $G_{\euw^\star|\euv}.$
\vskip2pt
2) $q_\LL^\star:G(\tl\LL^\star|\LL)\to G(\NN^\star|\LL)$ 
maps $T_{\tl\euv^\star|\euv}$ isomorphically onto $T_{\euw^\star|\euv}$
and defines a surjective morphism of the residue Galois groups 
$G_{\tl\euv^\star|\euv}\srjr G_{\euw^\star|\euv}$ which is 
obviously compatible with the $\ell$-adic characters. Finally, 
$q^\star_\LL$ defines the surjective canonical map
$q^\star_\LL:Z_{\tl\euv^\star|\euv}\srjr Z_{\euw^\star|\euv}$, 
compatible with the maps on inertia and residue Galois groups.
\vskip2pt
\underbar{Conclude}: \ First, $\mu_\ell\subset\NN$ $\Rightarrow$ 
$G_{\tl\euv^a|\euw}$ and $G_{\euw^a|\euw}$ act trivially on 
$T_{\tl\euv^a|\euw}\!\cong\lvZ/\ell\cong\!T_{\euw^a|\euw}.$
Hence $Z_{\tl\euv^a|\euw}$ and $Z_{\euw^a|\euw}$ 
are abelian and $\clp{I_\euw=T_{\euw^a|\euw}}{Z_{\euw^a|\euw}=D_\euw}$ 
is a c.l.\ pair in $G(\NN^a|\NN).$ 
Second, for $\sigma^c\in T_{\tl\euv^c|\euw}$, $\tau^c\in Z_{\tl\euv^c|\euw}$
one has $[\sigma^c\!,\tau^c]\in\langle\sigma^\beta\rangle.$ By 
abuse of language, we will say:
\vskip5pt
\noindent
{\bf Terminology}. \ \textit{\it $\clp{T_{\tl\euv^a|\euw}}{Z_{\tl\euv^a|\euw}}$ is a \textbf{generalized-commuting 
c.l.\ pair} of subgroups of $G(\tl\LL^a|\NN).$}
\end{notation/remark}
Recall $\tl\EE|\EE$ with $\tl\EE=\tl\LL\EE$ 
and the isomorphism $\tl\imath:G(\tl\EE|\EE)\to G(\tl\LL|\LL).$
Setting $\FF=\NN\EE$, thus $\FF|\EE\hra\tl\EE|\EE$, for
$\star=a$,$c$ one has the resulting surjective projections of Galois 
groups $q^\star_\EE:G(\tl\EE^\star|\EE)\srjr G(\FF^\star|\EE),$ 
$q^\star_\FF:G(\tl\EE^\star|\FF)\srjr G(\FF^\star|\FF)$ 
and $q^\star_\LL:G(\tl\LL^\star|\LL)\srjr G(\NN^\star|\LL),$ 
$q^\star_\NN:G(\tl\LL^\star|\NN)\srjr G(\NN^\star|\NN),$ and  
$p^\star:G(\FF^\star|\EE)\srjr G(\NN^\star|\LL),$  
$p^\star:\clG^\star_\FF\srjr \clG^\star_\NN.$ 
\vskip2pt
Thus 
given an 
a.b.c.-liftable section $s^a:G(\tl\LL^a|\LL)\to G(\tl\EE^a|\LL)$, 
one has commutative diagrams:
\vskip-10pt
\[
\mathsurround=0pt
\begin{tikzcd}[row sep=scriptsize, column sep=scriptsize]
\clG_{\tl\EE}^\star \arrow[dr, "q^\star_\FF"]  
       \arrow[dd, two heads, "\hbox{\lower30pt\hbox{$\tl p^\star$}}"] \arrow[rr, hook] 
      & &G(\tl\EE^\star|\EE) \arrow[dr, two heads, "q^\star_\EE"] 
             \arrow[dd, near end, two heads, "\tl p^\star"] 
                     \arrow[rr, two heads, near end, "\hb{10}p_{\tl\EE}"]
      & &G(\tl\EE|\EE) \arrow[dr, two heads, "q_\EE"]  \arrow[dd, crossing over, 
                 "\hbox{\lower30pt\hbox{$\tl\imath$}}"] \\
&\clG_{\FF}^\star \arrow[rr, crossing over, hook] 
         \arrow[dd, two heads, "\hbox{\lower30pt\hbox{$\,p^\star$}}"] 
      & & G(\FF^\star|\EE) \arrow[dd, two heads, near end, "p^\star"]           
      \arrow [rr, crossing over, two heads, "\hb{40}p_\FF"] 
      & & G(\FF|\EE) \arrow[dd, "\hbox{\lower20pt\hbox{$\imath$}}" ] \\
\clG_{\tl\LL}^\star \arrow[dr, "q^\star_\NN" '] \arrow[rr, hook] 
         \arrow[uu, bend left, "\hbox{\lower-25pt\hbox{$s^\star$}}"] 
      & & G(\tl\LL^\star|\LL) \arrow[dr, two heads, near start, "q^\star_\LL" '] 
        \arrow[rr, two heads, near end, "p_{\tl\LL}" ] 
               \arrow[uu, bend left, "\hbox{\lower-25pt\hbox{$s^\star$}}"]  
      & & G(\tl\LL|\LL) \arrow[dr, two heads,  "q_\LL"] \\
&\clG_{\NN}^\star \arrow[rr, hook] \arrow[from=uu, crossing over] 
     & &  G(\NN^\star|\LL) \arrow[rr, two heads, "\hb{35}p_{\NN}"] 
            \arrow[from=uu, crossing over] 
     & &G(\NN|\LL) 
\end{tikzcd}
\]
\vskip-3pt
\noindent
in which all maps are the canonical projections and $\tl\imath$ 
and $\imath$ are isomorphisms. Moreover, the diagram for $\star=a$
is a canonically quotient of the diagram for $\star=c.$
\begin{keylemma}
\label{keylemma1}
In the above notation$\!/\!\!$context, set
$I^a\!:=q^a_\FF\big(s^a(T_{\tl\euv^a|\euw})\big),\
D^a\!:=q^a_\FF\big(s^a(Z_{\tl\euv^a|\euw})\big).$ Then 
$\clp{I^a}{D^a}$ 
is a c.l.\ pair in $\GGa\FF\subset G(\FF^a|\EE)$ which is mapped 
by $p^a:G(\FF^a|\EE)\to G(\NN^a|\LL)$ onto the c.l.\ pair 
$\clp{I_\euw=T_{\euw^a|\euw}}{Z_{\euw^a|\euw}=D_\euw}$ 
in $\GGa\NN\subset G(\NN^a|\LL).$ And the same holds, 
correspondingly, for the maximal c.l.\ pair $\clp{I_{D_\euw}}{D_\euw}$
in $\GGa\NN.$ 
\end{keylemma}
\begin{proof}
In the above notation, recalling the remarks at 1) above, one has:
The canonical projection $G(\tl\EE^c|\FF)\srjr G(\tl\EE^a|\FF)$ 
maps the subgroups 
$\clp{s^c(T_{\tl\euv^c|\euw})}{s^c(Z_{\tl\euv^c|\euw})}$ of $G(\tl\EE^c|\FF)$
onto the subgroups $\clp{s^a(T_{\tl\euv^a|\euw})}{s^a(Z_{\tl\euv^a|\euw})}$
of $G(\tl\EE^a|\FF).$ Recalling that  $[\tlsgm,\tltau]\in\langle\sigma^\beta\rangle$ 
for all $\tlsgm\in T_{\tl\euv^c|\euw}$, $\tltau\in Z_{\tl\euv^c|\euw},$
one has $[s^c(\tlsgm),s^c(\tltau)]\in\langle s^c(\sigma)^\beta\rangle$ for all
$\tlsgm\hb3\in T_{\tl\euv^c|\euw}$ and $\tltau\hb3\in Z_{\tl\euv^c|\euw}.$
Hence since the canonical projections  
$G(\tl\EE^\star|\FF)\srjr G(\FF^\star|\FF)$ are surjective and 
$q_\NN^\star=p^\star\circ q_\FF^\star\circ s^\star\!\!$, it follows that the 
subgroups $I^\star\!:=q_\FF^\star\big(s^\star(T_{\tl\euv^\star|\euw})\big),$ 
$D^\star=q_\FF^\star\big(s^\star(Z_{\tl\euv^\star|\euw})\big)$ 
of $\clG_\FF^\star=G(\FF^\star|\FF)$ satisfy:
\vskip2pt
\itm{25}{
\item[a)] $I^a\!\subset D^a$ are subgroups of $\GGa\FF.$ Further, 
$q_\NN^\star=p^\star\circ q_\FF^\star\circ s^\star$ implies 
$p^a(I^a)=T_{\euw^a|\euw}$ and $p^a(D^a)=Z_{\euw^a|\euw}.$ 
Hence $I^a\neq1$ and $D^a$ is not cyclic.
\vskip2pt
\item[b)] $I^c\!\subset D^c$ are subgroups of $\GGc\FF$ which project
onto $I^a\subset D^a$ under $G(\FF^c|\FF)\srjr G(\FF^a|\FF)$
and $\forall\,\sigma\in I^a\!,\,\forall\, \tau\in D^a$ and any 
preimages $\tlsgm\in I^c\!,\, \tltau\in D^c$ one has 
$[\tlsgm\!,\tltau]\in\langle\sigma^\beta\rangle.$
}
\noindent
Finally, in the case when $\clp{I_\euw=T_{\euw^a|\euw}}{Z_{\euw^a|\euw}=D_\euw}$ is
replaced by $\clp{I_{D_\euw}=I_{Z_{\euw^a|\euw}}}{Z_{\euw^a|\euw}=D_\euw},$ the 
assertion of Lemma follows along the same lines, so we omit the details.
\end{proof}
%
%
\begin{construction}[{\bf\hb{3.5}($\bm{\euw|\euv\leadsto
\clp{I_{D_I}}{D_I}\leadsto w|v\in\clW_{\FF|\EE}}$}]$\ha0$
\label{construction}
In the above context, $\clp{I^a}{D^a}$ is a c.l.\ pair in $\GGa\FF$, 
and in notation from Lemma~\ref{keylemma1}, let $\clp{I_{D_{I^a}}}{D_{I^a}}$ 
be the maximal c.l.\ pair in $\GGa\FF$ attached to $\clp{I^a}{D^a}$ 
as in Definition/Remark~\ref{defrem1},~4. Further, if 
$\clp{I_{D_{I^a}}}{D_{I^a}}\leadsto w\in\clW_\FF$ as 
in~Remark~\ref{unique}, by~Theorem~\ref{Topaz} one has:
\vskip5pt
\centerline{$I_{D_{I^a}}=I_{D_w}$ and $D_{I^a}=D_w.$}
\vskip5pt 
\noindent
Set $\euv=\euw|_\LL$, $v=w|_\EE.$ Then in Notation~\ref{nota2}
one has $\euw|\euv\in\clW_{\NN|\LL}$, $w|v\in\clW_{\FF|\EE}$,~and~further: 
\vskip2pt
\itm{25}{
\item[a)] $D_\euw\leqslant Z_{\euw^a|\euv}$ lies in $\clD_{\NN|\LL}$
and $D_{w}\leqslant Z_{w^a|v}$ lies in $\clD_{\FF|\EE}.$
\vskip2pt
\item[] Further, the canonical projections  
$Z_{\euw^a|\euv}\srjr Z_{\euw|\euv}$ and $Z_{w^a|v}\srjr Z_{w|v}$
are surjective.
\vskip2pt
\item[b)] Recall the isomorphism $\imath:G(\FF|\EE)\to G(\NN|\LL)$, 
say $\gx\mapsto h.$ 
Then by Proposition~\ref{fktTopaz} applied to $\euw|\euv$ and 
$w|v$, the following hold:
\vskip3pt
\item[] \scalebox{.999}[1]{$(*)_\euw\ha5\gy(I_{D_\euw})\!=\!I_{D_\euw} 
\ha4 {\it iff} \ha4\gy(D_\euw)\!=\!D_\euw \ha4 {\it iff} \ha4 
\gy(Z^a_\euw)\!=\!Z^a_\euw  \ha4 {\it iff} \ha4 
\gy(\euw)\!=\!\euw  \ha4 {\it iff} \ha4 \gy\in Z_{\euw|\euv}.$}
\vskip2pt 
\item[] \scalebox{.999}[1]{$(*)_w\ha4 \gx(I_{D_{I^a}})\!=\!I_{D_{I^a}} 
  \ha4 {\it iff} \ha4 \gx(D_{I^a})\!=\!D_{I^a} \ha7 {\it iff} \ha4
  \gx(Z^a_w)\!=\!Z^a_w
   \ha4 {\it iff} \ha4 \gx(w)\!=\!w \ha4 {\it iff} \ha4 \gx\in Z_{w|v}.$}
\vskip2pt
\item[c)] In particular, $\imath:Z_{w|v}\to Z_{\euw|\euv}$ is a well 
defined isomorphism.
}
\end{construction}   
%
%
\begin{proposition}[{\bf Proposition~\ref{fkt01} re-revisited}]
\label{fktTopaz1} 
In the above notation, the following hold:
\vskip5pt
\itm{25}{
\item[{\rm1)}] The valuation $w_\NN\!:=w|_\NN\hb1$ satisfies
$w_\NN\geqslant\euw.$
\vskip2pt
\item[{\rm2})] $Z^a_\euw=D_{w_\NN}\!=p^a(D_w)$ \
and \ $T^a_\euw=p^a(I^a)\subset I_{w_\NN}\subset I_{D_\euw}.$  
\vskip2pt
\item[{\rm3)}] $\imath(Z_{w|v})=Z_{\euw|\euv}$ and 
$p^a(Z_{w^a|v})=Z_{\euw^a|\euv}.$ Further, for $g\in G(\FF|\EE)$ one has:
\vskip5pt
\centerline{$\gx(I^a)\!=\!I^a$ iff $\,\gx(D_{I^a})\!=\!D_{I^a}\,$ 
   iff $\,\gx(D_{w})=D_{w}$ 
   iff $\,\gx(Z^a_w)\!=\!Z^a_w$
      iff $\,\gx(w)\!=\!w\,$ iff $\,\gx\in Z_{w|v}.$}
} 
\end{proposition}
\begin{proof} To 1): We first notice that the surjectivity of $p^a$ 
and $p^c$ imply that for every c.l.\ pair $\clp{I'}{D'}$ in $\GGa\FF$ 
with $p^a(I')\neq1$ and $p^a(D')$ non-cyclic, the image 
$\clp{p^a(I')}{p^a(D')}$ of $\clp{I'}{D'}$ under $p^a$ is a c.l.\ pair 
in $\GGa\NN.$ (Actually, by mere definitions one has: If $\sigma,\tau$ 
is a c.l.\ pair in $\GGa\FF,$ then $p^a(\sigma),p^a(\tau)$ is a 
c.l.\ pair in $\GGa\NN,$ provided $p^a(\sigma),p^a(\tau)$ are 
independent in $\GGa\NN.$) Therefore, $\clp{p^a(I_{D_w})}
{p^a(D_w)}$ is a c.l.\ pair in $\GGa\NN$, and 
$I^a\subset I_{D_{I^a}}=I_{D_w}$, $D^a\subset {D_{I^a}}=D_w$ implies:
\vskip5pt
\centerline{$T^a_\euw=p^a(I^a)\subset p^a(I_{D_w})$ \ \ and \ \
$Z^a_\euw=p^a(D^a)\subset p^a(D_w).$}
\vskip5pt
\noindent
We proceed case-by-case as follows:
\vskip2pt
(I) First, suppose that $D_{I^a}=D_w$ is c.l. Then by Theorem~\ref{Topaz} 
one has that $D_w/I_w$ is cyclic. Thus $p^a(D_w)/p^a(I_w)$ 
is cyclic, hence so is $Z^a_\euw/I'\subset p^a(D_w)/p^a(I_w)$,
where $I'=p^a(I_w)\cap Z^a_\euw.$ Since $Z^a_\euw$ is not 
cyclic, we conclude that $I'\neq1$, and therefore 
$p^a(I_w)\supset I'\neq1.$ 
On the other hand, by basics of Hilbert decomposition theory, one 
has $p^a(I_w)\subset I_{w_\NN}$, thus finally implying 
$I_{w_\NN}\supset p^a(I_w)\supset I'\neq1.$ Therefore, $w_\NN$ 
is non-trivial.
\vskip2pt
(II) Second, suppose that $D_{I^a}=D_w$ is not c.l. Then by Theorem~\ref{Topaz}
we have that $I_w=I_{D_{I^a}}.$ Hence since $I^a\subset I_{D_{I^a}}=I_w$
and $p^a(I^a)=T^a_\euw$, it follows that $p^a(I_w)\supset T^a_\euw\neq1\vrg$
and finally, as above, $I_{w_\NN}\supset p^a(I_w)\neq1.$ Therefore, $w_\NN$ 
is non-trivial.
\vskip5pt
Further, by basics of Hilbert decomposition theory, one 
has that the decomposition group $Z^a_{w_\NN}$ satisfies
$Z^a_{w_\NN}\supset p^a(Z_w)$, and therefore, 
$I'\subset Z^a_\euw\cap Z^a_{w_\NN}$ in case~(I), respectively
$T^a_\euw\subset Z^a_\euw\cap Z^a_{w_\NN}$ in case~(II).
Therefore, by Proposition~\ref{fkt01}, it follows that $\euw$ and $w_\NN$
are not independent. Since $w_\NN$ is non-trivial and $\euw$
has rank one (being discrete), if follows that $\euw\leqslant w_\NN.$
\vskip5pt
To 2): Taking into account that $\euw\leqslant w_N$ by assertion~1) 
above, using e.g., {\sc Topaz}~\cite{To1}, Lemma~4.1, one gets: 
Since $\euw\leqslant w_N$, i.e., $w_N$ is a refinement of $\euw$, 
one has that 
\vskip5pt
\centerline{$Z^a_{w_N}\subset Z_\euw$, $T^a_{w_N}\supset T_\euw$,
and further, $D_{w_N}\subset D_{\euw}$, $I_{w_N}\supset I_{\euw}$, 
etc.}
\vskip5pt
\noindent
On the other hand, by functoriality of (reduced) Hilbert 
decomposition theory, $p^a(D_w)\subset D_{w_\NN}.$ 
Hence since $Z^a_\euw=D_\euw= p^a(D^a)$ and $D^a\subset D_w$, 
putting everything together, we get:
\vskip6pt
\centerline{$Z^a_\euw=D_\euw=p^a(D^a)\subset p^a(D_w)\subset
       D_{w_\NN}\subset D_\euw=Z^a_\euw$,}
\vskip5pt
\noindent
concluding that $Z^a_\euw=D_\euw=p^a(D_w)=D_{w_\NN}.$ 
\vskip2pt
Concerning the assertion about (reduced) inertia, 
let $I_{D_\euw}\subset D_\euw=p^a(D)$ be the unique maximal 
subgroup of $\GGa\NN$ such that $\clp{I_{D_\euw}}{D_{\euw}}$ 
is a c.l.\ pair in $\GGa\NN.$ Since $D_{w_\NN}=D_\euw$, and 
$\clp{I_{w_\NN}}{D_{w_\NN}}$ is a c.l.\ pair, it follows that 
$I_{w_N}\subset I_{D_\euw}.$ Conclude that 
$T^a_\euw=p^a(I)\subset I_{w_\NN}\subset I_{D_\euw}.$ 
\vskip4pt
To 3): Since $p^a(D_{w})=D_\euw$, taking into account 
assertions $(*)_\euw$ and $(*)_w$ above, by mere definitions one has  
$\imath(Z_{w|v})\subset Z_{\euw|\euv}$. For the converse 
inclusion, proceed as follows: By assertion~$(*)_\euw$, one 
has that $h\in Z_{\euw|\euv}$ iff $h(I_{D_\euw})=I_{D_\euw}.$
Let $h'\mapsto h$ under $G(\NN^a|\LL)\srjr G(\NN|\LL)$, 
and denote $g'=s^a(h')$. Then by mere definitions one has: 
$\tau\in D_{I^a}$ iff $\forall\,\sigma\in I\!:=s^a(I_{D_\euw})$ have: 
$\sigma,\tau$ is a c.l.\ pair in $\GGa\FF$. Obviously, since the inner 
conjugation by $g'$ is an automorphism of $\GGa\FF$ which 
lifts to the inner conjugation in $\GGc\FF,$ the latter assertion is 
equivalent to $\sigma'\!,\tau'$ being a c.l.\ pair in $\GGa\FF$, 
where $\sigma'\!=g'\sigma\,g'^{-1}$ and $\tau'\!=g'\tau g^{-1}\!$. 
On the other hand, $\sigma\mapsto\sigma'\!:=g'\sigma g'^{-1}$ 
is an automorphism of $I_{D_\euw}$ (because $g'I_{D_\euw}\,g'^{-1}=
g'(I_{D_\euw})=I_{D_\euw}$). Hence $\tau\in D_{I^a}$ iff 
$\tau'\!:=g'\hb1\tau\ha1g'^{-1}\in D_{I^a}$. Thus $g'D_{I^a}g'^{-1}=D_{I^a}$,
that is, $g'(D_{I^a})=D_{I^a}$. Since $g \in Z_{\euw|\euv}$ was arbitrary and 
$g'(D_{I^a})=D_{I^a}$, conclude by Fact~\ref{fktTopaz} that 
$\imath^{-1}(\tau) \in Z_{w|v}$ under the isomorphism 
$\imath:G(\FF|\EE)\to G(\NN|\LL)$. Thus
$\imath(Z_{w|v})=Z_{\euw|\euv}$ as claimed.
\end{proof} 
%
%
%
\subsection{\bf Canonical $s^\star$-valuations and their
functorial behavior}$\ha0$
\vskip1pt
In the context of Proposition~\ref{fktTopaz1} above, recall that 
given valuations $\euw|\euv$ of $\NN|\LL$, $\euv\in\Val^1(\LL)$, 
via the section $s^\star$ of $p^\star$ one gets valuations 
$w|v$ of $\FF|\EE$ satisfying $w_\NN\!:=w|_{\NN}\geqslant\euw$ 
and $v_\LL\!:=w|_\LL\geqslant\euv.$ In particular, 
Fact~\ref{fkt-c-val} above applies in this context, and one has:
\vskip2pt
\itm{25}{
\item[1)] the canonical $\euw\!$-valuations and/or $\euv\!$-valuations
of $\FF\!$, which turn out to be equal $w_\euw=w_\euv.$ Indeed,
this follows by Fact~\ref{fkt-c-val},~3), because $\NN|\LL$ is an
algebraic extension.  
\vskip2pt
\item[2)] The canonical $\euw^a\!$- and $\euw\!$- and $\euv$-valuations
of $\FF^a$ are equal $w^a_{\euw^a}=w^a_\euw=w^a_\euv$ and
prolong $w_\euw=w_\euv$ to $\FF^a.$ Indeed, this follows from 
Fact~\ref{fkt-c-val},~3), because $\FF^a|\FF$ is algebraic. 
\vskip2pt
\item[] And since $\FF^a|\FF|\EE$ are Galois extensions, all the 
above valuations have the same restriction to~$\EE\!$, denoted 
$v_\euv\!:=(w^a_{\euw^a})|_\EE=(w^a_\euw)|_\EE=(w_\euw)|_\EE$, etc.
}
%
%
\begin{definition}
\label{def-can-val}
In the above context, the valuations $w^a_{\euw^a}=w^a_\euw=w^a_\euv$ 
of $\FF^a$ and $w_\euw=w_\euv$ of~$\FF$ and $v_\euv$ of $\EE$ are 
called {\it canonical $s^\star\!$-valuations} of $\FF^a|\FF|\EE$ (defined by 
$\euw^a|\euw|\euv$
via $s^\star$). 
\vskip2pt\noindent
\bltt\ \ We notice the following: Since $\chr(\NN\euw)=\chr(\LL\euv)\neq\ell$, 
$\euv\in\Val^1(\LL)$ and $(w_\euw)|_\NN=\euw$, one has that 
$\chr(\FF w_\euw)\neq\ell.$ In particular, one also has 
$D_{w_\euw}=Z^a_{w_\euw}$ and $I_{w_\euw}=T^a_{w_\euw}$,
etc.
\end{definition}
\newpage
%
\begin{fact}
\label{fkt-can-val}
{\it The $s^\star\!$-canonical valuations $w_{\euw}|v_\euv$ arising from 
$\euw|\euv$, $\euv\in\Val^1(\LL)$ satisfy: 
\vskip5pt
\centerline{\rm (i) $\clO_{w_\euw}^\times\!\supset\clO_{w}^\times$, 
$1+\eum_{w_\euw}\!\subset1+\eum_{w}$;  \ \ \
(ii) $\clO_{w_\euw}^\times\cap\NN=\clO^\times_\euw$, 
$(1+\eum_{w_\euw})\cap\NN=1+\eum_\euw.$}
\vskip5pt
\noindent 
Further, (by mere definitions) one has:
\vskip2pt
\itm{25}{
\item[{\rm1)}] $s^a(T^a_\euw)\subset T^a_{w_\euw}\subset I_w$, 
$
Z^a_{w_\euw}=D_w\supset s^a(Z^a_\euw)$, and
$Z_{w^a_\euw|v_\euv}= Z_{w^a|v}\supset s^a(Z_{\euw^a|\euv}).$
\vskip2pt
\item[{\rm2})] $p^a(T^a_{w_\euw})=T^a_\euw$, $p^a(Z^a_{w_\euw})=Z^a_\euw$, 
and $p^a(Z_{w^a_\euw|v_\euv})=Z_{\euw^a|\euv}.$ Thus 
$\imath(Z_{w_\euw|v_\euv})=Z_{\euw|\euv}.$ 
}
}
\end{fact}
\begin{proof} To 1): The inclusions $s^a(T^a_\euw)\subset T^a_{w_\euw}\subset I_w$
and $Z^a_{w_\euw}\supset D_w\supset s^a(T^a_\euw)$ are clear by mere definitions. 
For the converse inclusion $Z^a_{w_\euw}\subset D_w$ one has: Since
$I^a\subset T^a_{w_\euw}=I_{w_\euw}$ and $Z^a_{w_\euw}=D_{w_\euw}$ 
acts on $I^a_{w_\euw}$ via the cyclotomic character, the following hold:
$\forall\,\sigma\in I^a$ and $\forall\,\tau\in Z^a_{w_\euw}$ one has that 
$\sigma,\tau$ is a c.l.\ pair in $\GGa\FF.$ On the other hand, by definition
one has that $D_w=D_{I^a}.$ Thus finally it follows that 
$D_{w_\euw}\subset D_{I^a}\subset D_w$, as claimed. 
\vskip2pt
To 2): \ The equality $p^a(D_{w_\euw})=D_\euw$ follows from
$p^a(D_w)=D_\euw\!$, cf.\  Proposition~\ref{fktTopaz1}, 1). The other
equalities follow along the same lines using loc.cit.\ 2) and 3). 
\end{proof}
Next, in the notation from the previous subsections, let 
$(\tl\LL|\NN_\ia|\LL_1|\LL)_\ia$ be the family of finite Galois 
subextensions of $\tl\LL|\LL$ with $\LL_1\subset\NN_\ia.$ We 
notice that this family is (filtered) partially ordered by $\ia\leqslant\jb$ 
iff $\NN_\ia\subset\NN_\jb$, thus setting $\FF_\ia\!:=\EE\NN_\ia$,
one has the resulting (filtered) partially ordered family 
of finite Galois subextensions $\tl\EE|\FF_\ia|\EE_1|\EE.$ 
Finally, let $p^\star_\ia:G(\FF_\ia^\star|\EE)\to 
G(\NN_\ia^\star|\LL)$ be the resulting projections, thus
$p^\star:G(\tl\EE^\star|\EE)\to G(\tl\LL^\star|\LL)$ is the (porjective)
limit of $(p^\star_\ia)_\ia.$
\vskip2pt
For $\euv\in\Val^1(\LL)$ and its prolongations
$\tl\euv|\euw_\ia|\euv_1|\euv$ to $\tl\LL|\NN_\ia|\LL_1|\LL$, let 
$w_\ia|v_\ia\in\clW_{\FF_\ia|\EE}$~be~the valuation defined 
via $s^\star$ and $\euv\in\Val^1(\LL)$ as indicated above. 
In particular, $w_\ia$ and the resulting 
$w_\ia|_{\NN_\ia}\geqslant\euw_\ia$ give rise to the canonical
$\euw_\ia$-valuation $w_{\euw_\ia}$ of $\FF_\ia$, which is the 
minimal valuation $w'_\ia$ of $\FF_\ia$ such that $w'_\ia\leqslant w_\ia$ 
and $w'_\ia|_{\NN_\ia}=\euw_\ia.$ Further, for 
$\NN_\ia\subset\NN_\jb$ and the resulting $\FF_\ia\subset\FF_\jb$, 
let $p^\star_{\jb\ia}:G(\FF^\star_\jb|\EE)\to G(\FF^\star_\ia|\EE)$ be
the resulting canonical projections, thus 
$p_\jb^\star=p^\star_\ia\circ p^\star_{\jb\ia}.$ And recalling the 
notation introduced in Construction~\ref{construction}, set 
$I^a_\ia\!:=s_\ia(I_\euw)$, $D^a_\ia=s_\ia(D_\euw)$ and notice that 
$p_\jb^\star=p^\star_\ia\circ p^\star_{\jb\ia}$ implies:
\[
p^a_{\jb\ia}(I^a_\jb)=I^a_\ia, \ \ \ \ \ p^a_{\jb\ia}(D^a_\jb)=D^a_\ia.
\leqno{\indent(*)_{\jb\ia}}
\]
%
%
\begin{keylemma}[{\bf Functoriality of $s^\star\!$-canonical 
valuations}] 
\label{key-lemma}
For $\FF_\ia\subset\FF_\jb$, let $w_{\euw_\ia}\in\Val(\FF_\ia)$ and 
$w_{\euw_\jb}\in\Val(\FF_\jb)$ be the corresponding 
$s^\star\!$-canonical valuations. Then $w_{\euw_\jb}|_{\FF_\ia}=w_{\euw_\ia}.$
\end{keylemma}
\begin{proof} Let $w'_\ia\!:=w_{\euw_\jb}|_{\FF_\ia}$ and set
$w^0=\min(w'_\ia\vrg w_{\euw_\ia})\in\Val(\FF_\ia)$. 
\vskip2pt
\underbar{Step 1}. We claim that $w^0$ is non-trivial, or equivalently, 
$w'_\ia$ and $w_{\euw_\ia}$ are not independent. Indeed, 
by Fact~\ref{fkt-can-val},~1), one has 
$I^a_\ia=s_\ia^a(I_{\euw_\ia})\subset I_{w_{\euw_\ia}}\!$, and by
assertion~$(*)_{\jb\ia}$ above 
one has $p^a_{\jb\ia}(I^a_\jb)=I^a_\ia\subset p^a_{\jb\ia}(I_{w_{\euw_\jb}}).$
Thus since $w'_\ia=w_{\euw_\jb}|_{\FF_\ia}$, by Fact~\ref{fkt3-4},
one has  $p^a_{\jb\ia}(I_{w_{\euw_\jb}})\subset I_{w'_\ia}.$
Hence $1\neq I^a_\ia\subset I_{w'_\ia}\cap I_{w_{\euw_\ia}}\!$, 
and therefore, by Proposition~\ref{fkt01},~1), it follows that $w'_\ia$ 
and $w_{\euw_\ia}$ are not independent, as claimed. 
\vskip2pt
\underbar{Step 2}. We claim that $\euw^0\!:=w^0|_{\NN_\ia}$ 
is non-trivial. Indeed, by Fact~\ref{fkt3-4}, one has 
$p_\ia^a(I_{w^0})\subset I_{\euw^0}.$ Thus $I_{w^0}\supset I^a_\ia$ 
implies $I_{\euw^0}\supset p_\ia^a(I_{w^0})\supset p_\ia^a(I^a_\ia)=I_{\euw_\ia}\neq1$,
and therefore,  $w^0$ is non-trivial. Second, since $w^0\leqslant w_{\euw_\ia}$, 
one has $\euw^0=w^0|_{\NN_\ia}\leqslant w_{\euw_\ia}|_{\NN_\ia}=\euw_\ia.$ 
Thus since $\euw_\ia$ is discrete and $\euw^0\leqslant \euw_\ia$ 
is nontrivial, one must have 
$\euw^0=\euw_\ia$. Finally, since $w^0\leqslant w_{\euw_\ia}\leqslant w_\ia$
and $w^0|_{\NN_\ia}=\euw_\ia\!$, by the definition of the canonical 
$\euw_\ia$ valuation $w_{\euw_\ia}$ one must have $w^0=w_{\euw_\ia}.$
Hence since $w^0=\min(w'_\ia\vrg w_{\euw_\ia})$ and $w^0=w_{\euw_\ia}$,
conclud that $(w_{\euw_\jb})|_{\FF_\ia}=w'_\ia\geqslant w_{\euw_\ia}.$
\vskip2pt
\underbar{Step 3}. \ Using Lemma~\ref{can-val-abstr}, let 
$w'_\jb\leqslant w_{\euw_\jb}$ be the minimal with
$w'_\jb|_{\FF_\ia}=w_{\euw_\ia}.$ Then $w'_\jb\leqslant w_{\euw_\jb}$
implies $\euw'_\jb=w'_\jb|_{\NN_\jb}\leqslant w_{\euw_\jb}|_{\NN_\jb}=\euw_\jb.$
Since $\euw'_\jb\leqslant\euw_\jb$ both prolong $\euw_\ia$ to the 
algebraic extension $\NN_\jb|\NN_\ia$, we must have $\euw'_\jb=\euw_\jb.$
Thus by the definition of the 
canonical $\euw_\jb$-valuation $w_{\euw_\jb}$ one must have
$w_{\euw_\jb}\leqslant w'_\jb$, thus finally 
$w'_\jb=w_{\euw_\jb}.$ Conclude that $w_{\euw_\ia}=w_{\euw_\jb}|_{\FF_\ia}.$ 
\end{proof}
%
%
\begin{remark}
Recall that $\tl\LL=\cup_\ia\NN_\ia\hra\cup_\ia\FF_\ia=\tl\EE$, and that
by Key Lemma~\ref{key-lemma}, one has $w_{\euw_\ia}=w_{\euw_\jb}|_{\FF_\ia}$ 
for~$\ia\leqslant\jb.$ Therefore, there is a unique valuation $\tlwv$ of $\tl\EE$
with $\tlwv|_{\FF_\ia}=w_{\euw_\ia}$ for all $\ia$, and moreover, 
since $w_{\euw_\ia}|_{\NN_\ia}=\euw_\ia=\tl\euv|_{\NN_\ia}$, etc., 
the following hold:
\vskip2pt
\itm{25}{
\item[a)] $\tl\euv=\tlwv|_{\tl\LL}$, and $\tlwv|_\EE=w_{\euw_\ia}|_\EE=v_\euv$, 
$\tlwv|_\LL=w_{\euw_\ia}|_\LL=\euw_\ia|_\LL=\euv.$ 
\vskip2pt
\item[b)] By Fact~\ref{fkt-can-val},~2) one has: \ 
$s^a_\ia(I_{\euw_\ia})=I^a_\ia\subset I_{D_{w_{\euw_\ia}}}\,$
and $\,s_\ia(D_{\euw_\ia})\subset D_{w_{\euw_\ia}}\hb1$. Hence since
\vskip2pt
$G(\LL^a|\LL)=\plim\ia G(\NN_\ia^a|\LL)$ and 
$G(\EE^a|\EE)=\plim\ia G(\FF_\ia^a|\EE)$ canonically, one has
\vskip3pt
\centerline{
$s^a(I_{\tl\euv})=\plim\ia I^a_\ia=\plim\ia T^a_{\euw_\ia}$ \ and
\  $s^a(D_{\tl\euv})=\plim\ia D^a_\ia=\plim\ia Z^a_{\euw_\ia}.$}
}
\end{remark}
\noindent
From the above discussion, we conclude with the following:
\begin{fact}[{\bf Fact~\ref{fkt-can-val} revisited}]
\label{fkt-finally}
%
%
{\it The $s^\star\!$-canonical valuations $\tlwv$ arising 
from $\euv\in\Val^1(\LL)$ in the way explained above have
$\chr(\tl\EE\tlwv)=\chr(\tl\FF\tlwv)\neq\ell$ and further satisfy the following: 
\vskip2pt
\itm{25}{
\item[{\rm1)}] $\tlwv|_{\FF_\ia}=w_{\euw_\ia}\hb1$, $\tlwv|_\EE=v_\euv\hb1$,
$\tlwv|_{\NN_\ia}=\euw_\ia\hb1$, $\tlwv|_\LL=\euv.$ \ Hence
$I_{\tlwv}=T^a_{\tlwv}$ and $D_{\tlwv}=Z^a_{\tlwv}$. 
\vskip2pt
\item[{\rm2)}] $s^a(T^a_{\tl\euv})\subset T^a_{\tlwv}\!$, \
$s^a(Z^a_{\tlv})\subset Z^a_{\tlwv}$ \ and \ $s^a(Z_{\tl\euv^a|\euv})
\subset Z_{\tlwv^a|v_\euv})$ \ and further,
\vskip2pt
\item[] $\tl p^a(T^a_{\tlwv})=T^a_{\tl\euv}$, \ $\tl p^a(Z^a_{\tlwv})=Z^a_{\tlv}\!$, 
\ and \ $\tl p^a(Z_{\tlwv^a|v_\euv})=Z_{\tl\euv^a|\euv}.$ \ Thus 
$\tl\imath(Z_{\tlwv|v_\euv})=Z_{\tl\euv|\euv}.$ 
}
}
\end{fact}
\begin{proof} Beweis klar!
\end{proof}
\section{Proof of (\ha1the\ha1$\,\tl k|k$-quasi-minimalistic 
refinement of) Theorem~\ref{MThm1}} 
%
%
%
\subsection{\bf General setting for the $\,\tl k|k$-quasi-minimalistic refinement
(Thm~\ref{MThm} below)}
$\ha0$
\vskip2pt
In the sequel, $\ell>2$ is an odd prime number, $k$ is a field with $\chr(k)\neq\ell$,
$\tl k|k$ is a Galois extension 
of fields e.g., $\tl k=\oli k$, 
and $\star=a$,$c\ha1.$ 
For a geometrically irreducible $k$-scheme $T\!$, e.g., $T=k$, and 
the base change $\tl T:=T\times_k\tl k$, 
recall the canonical exact sequence: 
\[
1\to\pi_1(\tl T)\to \pi_1(T)\to G(\tl k|k)\to1.
\]
Thus setting $\tl\Pi^\star(k)\!:=G(\tl k^\star|k)$ and 
$\tl\Pi^c(T)\!:=\pi_1(T)/\pi_1^{(2)}(\tl T)$ and
$\,\tl\Pi^a(T)=:\!\pi_1(T)/\pi_1^{(1)}(\tl T)$, one has 
$\tl\Pi^\star(T)\srjr\tl\Pi^\star(k)$ canonically, and 
$\tl\pi_1^\star(T)\!:=\ker\big(\tl\Pi^\star(T)\to\tl\Pi^\star(k)\big)
\triangleleft\, \tl\Pi^\star(T)$ is the ``geometric part'' of $\tl\Pi^\star(T).$
Notice: If $\tl k=\oli k$, then $\pi_1(\tl k)=1$ 
and $\pi_1(\tl T)=\oli\pi_1(T)$ is the geometric part of $\pi_1(T)$
in the usual sense, etc. Finally, we get a canonical 
commutative diagram:
\[
\begin{matrix}
1\to&\oli \pi_1(T)&\hor{}&\pi_1(T)&\hor{p_T}&\pi_1(k)&\to1\cr
&\hb3\ddwn{}&&\ddwn{}&&\ddwn&\cr
1\to&\tl\pi^c_1(T)&\hor{}&\tl\Pi^c(T)&\hor{\tl p^c_T}&\tl\Pi^c(k)&\to1\cr
&\Ddwn{\tl pr_{\tl T}}&&\Ddwn{\tl pr_T}&&\Ddwn{\tl pr_k}&\cr
1\to&\tl\pi^a_1(T)&\hor{}&\tl\Pi^a(T)&\hor{\tl p^a_T}&\tl\Pi^a(k)&\to1\cr
\end{matrix}
\]

Notice that the above commutative diagram is canonically a 
quotient of the corresponding diagram in which $\tl\pi_1^\star$, 
$\tl\Pi^\star$ are replaced by $\oli\pi_1^\star$, $\Pi^\star\!.$ 
Further, the process $T\leadsto\tl\Pi^\star(T)\to\tl\Pi^\star(k)$ 
is functorial in $T$ for geometrically integral 
separated $k$-schemes. Hence recalling the discussion and 
the commutative diagrams introduced right before 
Remarks~\ref{CND3} from the Introduction, one gets 
commutative diagrams, the middle diagram is a quotient of 
the first one, and the third one is a quotient of the middle one:
\vskip-12pt
\usetikzlibrary{decorations.pathmorphing}
\[
\mathsurround=0pt
\begin{tikzcd}[column sep=scriptsize] 
\pi_1(K_t) \arrow[r, two heads, "p_t" '] \arrow[d, two heads, "q_K"]
  &\pi_1(k_t) \arrow[d, two heads, "q_k"] \arrow[l, dashed, bend right=10, "s_t" '] 
    &\Pi^c(K_t) \arrow[r, two heads, "p^c_t" ']\arrow[d, two heads, "q_K^\star"] 
      &\Pi^c(k_t) \arrow[d, two heads, "q_k^\star"] 
                \arrow[l, dashed, bend right=10, "s^c_t" ']
    &\tl\Pi^c(K_t) \arrow[r, two heads, "\tl p^c_t" ']\arrow[d, two heads, "\tl q_K^\star"] 
      &\tl\Pi^c(k_t) \arrow[d, two heads, "\tl q_k^\star"] 
                \arrow[l, dashed, bend right=10, "\tl s^c_t"']
                          \\
\pi_1(K) \arrow[r, two heads, "p_K" '] \arrow[d, two heads, "\qx_X"]
    &\pi_1(k) \arrow[l, dashed, bend right=10, "s_K" '] \arrow[d, "{\rm id}"] 
                       \arrow[r, squiggly ] 
       &\Pi^\star(K) \arrow[r, two heads, "p_K^\star" '] \arrow[d, two heads, "\qx^\star_X"]
         &\pi_1(k) \arrow[l, dashed, bend right=10, "s^\star_K"'] \arrow[d, "{\rm id}"]             
                       \arrow[r, squiggly ] 
       &\tl\Pi^\star(K) \arrow[r, two heads, "\tl p_K^\star" '] 
             \arrow[d, two heads, "\tl \qx^\star_X"]
         &\tl\Pi^\star(k) \arrow[l, dashed, bend right=10, 
                "\tl s^\star_K"'] \arrow[d, "{\rm id}"]             
                        \\
\pi_1(X) \arrow[r, two heads, "p_X" ']&\pi_1(k) 
        \arrow[l, dashed, bend right=10, "s_X" ']
 &\Pi^\star(X) \arrow[r, two heads, "p_X^\star" ']&\pi_1(k) 
                    \arrow[l, dashed, bend right=10, "s_X^\star"']                 
 &\tl\Pi^\star(X) \arrow[r, two heads, "\tl p_X^\star" ']&\tl\Pi^\star(k) 
                    \arrow[l, dashed, bend right=10, "\tl s_X^\star"']                 
\end{tikzcd}
\]
$\bullet$ In the sequel, for $S=k\vrg X\vrg T$, $T=X\vrg\,K\vrg\,K_t$ and 
$\star=a$,$\,c$ consider the canonical projections:
\vskip7pt
\centerline{$\tl\Pi^c(S)\,\horsrjr{\tl{pr}_S}\,\tl\Pi^a(S)\!$, \ 
$\pi_1(T)\,\horsrjr{\tl{pr}^\star_T}\,\tl\Pi^\star(T)\!$, \ 
and $\pi_1(K_t)\,\horsrjr{pr_t}\ha2\tl\Pi^c(K_t)\horsrjr{\tl q^\star_T}\ha4\tl\Pi^\star(T).$} 
%
%
%
\begin{remarks}[{\bf Remaks~\ref{CND3} revisited}]
One has (compare with assertions from Remarks~\ref{CND3}):
\vskip2pt
\itm{25}{
\item[1)] Every section $s_T\in\Spi Tk$ gives rise canonically to 
sections $\tl s^\star_T:\tl\Pi^\star(k)\to\tl\Pi^\star(T)$ of $\tl p^\star_T$ 
such that $s_T$ is a lift of $\tl s^\star_T$, i.e., $\tl s^\star_T=\tl{pr}^\star_T\circ s_T$, 
and $\tl s^c_T$ is a lift of $\tl s^a_T$, i.e., $\tl s^a_T=\tl{pr}_T\circ \tl s^c_T.$
\vskip2pt
\item[2)] Every $s_t\in\Spi {K_t}{k_t}$ gives rise canonically 
to a section $\tl s^\star_t:\tl\Pi^\star(k_t)\to\tl\Pi^\star(K_t)$ of 
$\tl p^\star_t\!$  such that $s_t$ is a lift of $\tl s_t^\star$, i.e., 
$\tl s_t^\star=\tl{pr}_{K_t}^\star\circ s_t$, and $\tl s^c_t$ is a 
lifting of $\,\tl s^a_t$, i.e., $\tl s^a_t=\tl{pr}_{K_t}\circ \tl s^c_t.$ 
\vskip2pt
\item[3)] For $s_T\in\Spi Tk\!$, let $s_t\in\Spi{K_t}{k_t}$ be a 
lift of $s_T$, i.e., $q_k\circ s_T=q_T\circ s_t$, and $\,\tl s_T^\star$ 
be as defined at item~1). Then $\tl s_t^c\!:=pr_t\circ s_t$ 
is a section of $\,\tl p_t^c$ which lifts~$\tl s_T^\star.$ 
}
\end{remarks}
%
%
%
\begin{definition}[{\bf Definition~\ref{CND2} revisited}] 
We say that a section $\tl s^a_T:\tl\Pi^a(k)\to\tl\Pi^a(T)$ of 
$\tl p^a_T:\tl\Pi^a(T)\to\tl\Pi^a(k)$ is {\it $\tl k|k$-$t\!$-a.b.c.\ 
birationally liftable,\/} if there is a section $\,\tl s^c_t$ of $\,\tl p^c_t$ 
lifting $\tl s^a_T$, i.e., $\tl s^a_T\circ \tl q^a_k=\tl q^a_T\circ \tl s_t^c.$ 
If so, we also say that $\tl s^c_t$ is a {\it $\tl k|k$-$t\!$-a.b.c.\ 
birational lift\/} of $\tl s^a_T.$
\end{definition}
%
%
\begin{remarks}[{\bf Remaks~\ref{CND4} revisited}]
We remark the following:
\vskip2pt
\itm{25}{
\item[1)] Let $\tl s^a_K$ be $\tl k|k$-$t\!$-a.b.c.\ birationally liftable 
and $\tl s^c_t$ be a $\tl k|k$-$t\!$-a.b.c.\ birational lift of $\tl s^a_K.$ 
Then $\tl s^c_t$ is a $\tl k|k$-$t\!$-a.b.c.\ birational lift of 
$\tl s^a_X\!:=\tl q^a_K\circ \tl s^a_K$, thus $\tl s^a_X$ is $\tl k|k$-$t\!$-a.b.c.\ 
birationally\ha2liftable.
\vskip2pt
\item[2)] For the converse, let $\tl s^c_t$ be a $\tl k|k$-$t\!$-a.b.c.\ 
birational lift of a given $\tl s^a_X.$ Then it is relatively easy to show that
$\tl s^a_X$ has birational lifts $\tl s^a_K.$ But (in the full generality 
we work in) is not clear whether $\tl s^a_X$ has some birational lift 
$\tl s^a_K$ which is $\tl k|k$-$t\!$-a.b.c.\ birationally liftable.
\vskip2pt
\item[$\bullet$] To compensate, we say that a section $\tl s^a_X$ of
$\tl p^a:\tl\Pi^a(X)\to\tl\Pi^a(k)$ is {\it strongly\/} $\tl k|k$-$t\!$-a.b.c.\ 
birationally liftable, if $\tl s^a_X$ lifts to some $\tl s^a_K$ which is itself 
$\tl k|k$-$t\!$-a.b.c.\ birationally liftable.
}
\end{remarks}
\noindent
{\bf Hypotheses.} For $\ell\neq\chr(k)$ odd and a Galois extension
$\tl k|k$, consider the hypotheses: 
\vskip2pt \
{\sf \ (H)} \ \ $\mu_\ell\subset \tl k$ and the degree $[\tlk^a\!:\!k]$ 
is infinite and divisible by $\ell.$
\vskip2pt \ 
{\sf (H0)} \ Setting $\,\tl k\!:=k(\mu_\ell)$, the field extension 
$\,\tl k|k\,$ satisfies hypothesis {\sf(H)}.
%
%
%
\begin{example} 
For an odd prime number $\ell\neq\chr(k)$, one has the following:
\vskip2pt
\itm{25}{
\item[1)] If $\,k\,$ is not $\ell$-closed, then the separable closure
$\,\oli k|k\,$ satisfies hypothesis~{\sf(H)}. 
\vskip2pt
\item[2)] 
Setting $\tl k\!:=k(\mu_\ell)$, the following are equivalent:
\vskip2pt
\centerline{ (i) {\it $\,\tl k\,|\,k$ satisfies hypothesis\/} (H0). \ \ \ \ 
(ii) {\it $\,\tl k^\times\hb3/\ell$ is infinite.\/}}
\vskip2pt
\item[] If so, we will simply say that {\it $k$ satisfies hypothesis\/} (H0).
\vskip4pt
\item[3)] The hypothesis~{\sf(H0)} is quite general, e.g., the infinite 
finitely generated fields, and more general, any Hilbertian field, 
etc., satisfy hypothesis~{\sf(H0)}. And if $\,k$ satisfies~{\sf(H0)}, 
one has:
\vskip5pt
\noindent
$\hb{10}(*)$ \ {\it $\oli k=\cup_\ia\, k_\ia\,$ inductively, where 
$\,k_\ia|\,k$ are finite Galois extensions with $k_\ia|k$ satisfying\/} {\sf(H0)}.
}
\end{example}

Recalling the notions of $\,\tl k|k$-a.b.c.\ liftable sections and 
$\,\tl k|k$-$t$-a.b.c.\ liftable sections, and the commutative diagram 
$(*)_{\tl k|k}$ above, consider/define the following:
%
%
%
\begin{definition} For $T=X\!$,$\,K$, let a section 
$\,\tl s^a_T:\tl\Pi^a(k)\to\tl\Pi^a(T)$ of the canonical projection 
$\tl p^a_T:\tl\Pi^a(T)\to\tl\Pi^a(k)$ be given. For closed points 
$x\in X\!$, let $\tl Z_{x}\subset \tl\Pi^a(T)$ be the decomposition groups 
above the $k$-valuation $v_{x}$ of~$K$~with $\clO_{v_{x}}=\clO_{x}.$  
We say that $x\in X$
\vskip2pt
\itm{25}{
\item[a)] {\it defines\/} $\,\tl s^a_T$ if $x\in X(k)$ 
and $\im(\tl s^a_T)\subset \tl Z_x$ for some $\tl Z_x$ 
 above $v_x$.  
\vskip2pt
\item[b)] {\it quasi-defines\/} $\,\tl s^a_T$ if 
$\,k_x\!:=\kappa(x)\cap\tl k=k$, and $\,\im(\tl s^a_T)\!\subset\! \tl Z_{x}$ 
for some  $\,\tl Z_{x}$ 
above $\,v_{x}$.
\vskip2pt
\item[c)] {\it approximates\/} $\,\tl s^a_T$ if 
$\,\im(\tl s^a_T)\cap \tl Z_{x}\subset\im(\tl s^a_T)$ is an open
subgroup for some $\,\tl Z_{x}$ 
above~$v_{x}$.
}
$\hb9\bullet$ Note that for $\tl k=\oli k$, the notions ``defines'' and 
``quasi-defines'' are identical. 
\end{definition}

\noindent
Therefore, Theorem~\ref{MThm1} form Introduction is a 
special case of the following:
%
%
%
\begin{theorem}[{\bf $\tl k|k$-\ha1quasi-minimalistic\ha2$t$-BSC}]
\label{MThm}
Let $\,\tl k|k$ satisfies hypothesis {\sf(H)}, and $X$ be a complete 
integral normal $k$-curve with function field $K=k(X).$ In the above 
notation, one has:
\vskip2pt
\itm{25}{
\item[{\rm1)}] Any $\tl k|k$-$t$-a.b.c.\ birationally liftable section 
$\tl s^a_K\!:\!\tl\Pi^a(k)\to\tl\Pi^a(K)$ of $\tl p^a_K\!:\!\tl\Pi^a(K)\to\tl\Pi^a(k)$ 
is \textit{\textbf{quasi-defined}} by a unique closed point $\,x\in X$
as described above. 
\vskip2pt
\item[{\rm2)}] In particular, the strongly $\tl k|k$-$t$-\ha1a.b.c.\ birationally liftable 
sections $\,\tl s^a_X:\tl\Pi^a(k)\to\tl\Pi^a(X)$ of $\tl p^a_X:\tl\Pi^a(X)\to\tl\Pi^a(k)$ 
are \textit{\textbf{quasi-defined}} by closed points $x\in X$ as described above.
\vskip2pt
\item[{\rm3)}] Any $\tl k|k$-$t$-a.b.c.\ birationally liftable section 
$\tl s^a_X\!:\!\tl\Pi^a(k)\to\tl\Pi^a(X)$ of $\tl p^a_X\!:\!\tl\Pi^a(X)\to\tl\Pi^a(k)$ 
is \textit{\textbf{approximated}} by a closed point $\,x\in X$ as defined above.
}
\end{theorem}
\subsection{\bf Preparation for the proof of Theorem~\ref{MThm}}$\ha0$
\vskip2pt
Let $\tl k|k$ satisfying (H),  $X$ be a geometrically 
integral proper $k$-curve, and $K=k(X).$ This leads to a special case of 
the general situation from Section~3 as follows. Let $\LL\!:=k(t)=:\!k_t$ 
the rational function field in the variable $t$ over $k$ and 
$\EE\!:=K(t)=:\!K_t$ be the compositum of $K=k(X)$ and $L=k(t)$ 
over $k.$ Then $\EE|\LL$, that is, $K_t|k_t$ is a regular field 
extension (because $K|k$ was so). 
Thus setting $\tl k_t\!:=\tl k(t)$ and $\tl K\!:=K\tl k$, $\tl K_t\!:=\tl K(t)$,
one has $\tl\EE=\EE\tl\LL=\tl K_t$, etc. 
For the embeddings of Galois extensions 
$\tl K_t|K_t\hla\tl k_t|k_t\hla \tl k|k$, 
$\tl K_t^a|K_t\hla\tl k_t^a|k_t\hla \tl k|k$, let
\vskip5pt
\centerline{$G(\tl K_t|K_t)\hor{\tl\imath} G(\tl k_t|k_t)\hor{} G(\tl k|k)$, 
\quad
$\tl\Pi^\star(K_t)\horsrjr{p^\star_t}\ha3 \tl\Pi^\star(k_t)
\horsrjr{q^\star_k}\ha3 \tl\Pi^\star(k)$}
\vskip5pt
\noindent
be the resulting canonical isomorphisms of Galois groups, 
respectively the surjective morphisms of Galois groups, 
where~$\star=a$,$\,c.$
Further, since $\tl k|k$ satisfies hypothesis (H) above, there is a 
finite Galois subextension $k_1|k\hra\tl k|k$ with $\mu_\ell\subset k_1$ 
and $k_1^\times\hb3/\ell\neq1.$ Hence since $\ell>2$, all finite
sub-extensions $k'|k\hra \tl k|k$ with $k_1\subset k'$ satisfy 
$\mu_\ell\subset k'$ and $k'^\times\hb3/\ell\neq1$ 
as well.
\vskip2pt
To proceed, we consider the family of finite Galois subextensions 
$k_\ii|k$ of $\tl k|k$ with $k_1\subset k_\ii$, partially ordered by: 
$\ii\leqslant \jj$ iff $k_\ii\subset k_\jj$. Then we are in the context 
if Section~3, by setting $\NN_\ii=k_{\ii,t}\!:=k_\ii(t)\subset\tl k_t=\tl\LL$ and 
$\FF_\ii=\EE\NN_\ii=k_{\ii,t}\subset\tl k_t$, getting isomorphic 
projective systems of finite groups $G(k_\jj|k)\srjr G(k_\ii|k)$, 
$G(\NN_\jj|\LL)\srjr G(\NN_\ii|\LL)$, $G(\FF_\jj|\EE)\srjr G(\FF_\ii|\EE)$ 
for $\ii\leqslant\jj$, having the canonical isomorphisms 
$G(\tl\EE|\EE)\hor{\tl\imath} G(\tl\LL|\LL)\hor{} G(\tl k|k)$ as limit.
\vskip3pt
$\bullet$ {\it Concerning valuations\/}: Recall that all $k$-valuations $\vy\in\Val_k(k_t)$ 
are discrete, being either the $p(t)$-adic valuations $\vy=v_{k,p}$ with 
$p=p(t)\in k[t]$ the monic irreducible polynomials, or $\vy=v_\infty$ 
with uniformizing parameter $\pi_\infty={1\over t}.$ For $\vy\in\Val_k(k_t)$ 
consider the prolongations $\vz_\ii^a|\vz_\ii|\vy$ of~$\vy$ to 
$k_{\ii,t}^a |k_{\ii,t}|k_t$ with limit $\tl\vy^a|\tl\vy|\vy$ as 
prolongations of $\vy$ to $\tl k_t^a | \tl k_t| k_t.$ We notice that
$\tl\vy|\vz_\ii|\vy$ of $\tl k_t|\NN_\ii|k_t$ are unramified and 
$k_{\ii,t}\vz_\ii^\times\hb3/\ell\neq1$ (because $k_{\ii,t}\vz_\ii|k_1$ 
is finite). Therefore, we conclude that $\Val_k(k_t)\subset\Val^1(k_t).$
Similarly, for
$\FF_\ii\!:=K_{\ii,t}\!:=K k_{\ii,t}$ and $v\in\Val_k(K_t)$, consider its 
prolongations $w_\ii^a|w_\ii|v$ to $K_{\ii,t}^a | K_{\ii,t}|K_t$ and 
$\tl v^a|\tl v|v$ prolonging $v$ to $\tl K_t^a |\tl K_t|K_t$. 
\vskip3pt
Conclude that $\tl\LL|\LL=\tl k|k_t$, $\tl\EE|\EE=\tl K_t|K_t$ are as
the ones introduced/defined above in Section~3.
\vskip5pt
Next, let $\tl s^a_t:\tl\Pi^a(k_t)\to \tl\Pi^a(K_t)$ be an a.b.c.-liftable
section of the canonical (surjective) projection 
$\tl p^a_t:\tl\Pi^a(K_t)\srjr \tl\Pi^a(k_t)$, i.e., $\tl s^a_t$ lifts to a 
section $\tl s^c_t:\tl\Pi^c(k_t)\to \tl\Pi^c(K_t)$ of the canonical 
(surjective) projection $\tl p^c_t:\tl\Pi^c(K_t)\to \tl\Pi^c(k_t).$ 
Then recalling the canonical isomorphism 
\vskip5pt
\centerline{$\tl\imath:G(\tl K_t|K_t)\to G(\tl k_t|k_t)$ \
                                    defined by \ $\tl K_t|K_t\hla\tl k_t|k_t$,} 
\vskip5pt
\noindent
one has the following:
\vskip-20pt
%
%
\begin{fact}[{\bf Fact~\ref{fkt-finally} revisited}]
\label{fkt-finally-bis}
{\it In the above context, for $\euv\in\Val_k(k_t)$ and its prolongation 
$\tl\euv|\euv$ to $\tl k_t|k_t$, consider the corresponding 
inertia/decomposition groups $T^a_{\tl\euv}\!\subset 
     Z^a_{\tl\euv}\!\subset Z_{\tl\euv^a|\euv}\!\subset \tl\Pi^a(k_t).$ 
Then there is a unique minimal valuation $\tlwv\in\Val(\tl K_t)$ 
such that the following hold:
\vskip4pt
\itm{25}{
\item[{\rm1)}] $\tlwv|_{\tl k_t}=\tl\euv$, thus $\tlwv$ 
is trivial on $k$, i.e., $\tlwv\in\Val_k(\tl K_t)$ and 
$v_\euv\!:=\tlwv|_{K_t}\in\Val_k(K_t).$
\vskip4pt
\item[{\rm2)}] $\tl s^a_t(T^a_{\tl\euv})\subset T^a_{\tlwv}$, 
$\tl s^a_t(Z^a_{\tl\euv})\subset Z^a_{\tlwv}$, \ and 
$\tl s^a_t(Z_{\tl\euv^a|\euv})\subset Z_{\tlwv^a|v_\euv}$, \ and further,
\vskip3pt
\item[] $\tl p^a_t(T^a_{\tlwv})=T^a_{\tl\euv}$, \ $\tl p^a_t(Z^a_{\tlwv})=Z^a_{\tl\euv}$, 
\ and \ $\tl p^a_t(Z_{\tlwv^a|v_\euv})=Z_{\tl\euv^a|\euv}.$ Thus \
$\imath(Z_{\tlwv|v_\euv})=Z_{\tl\euv|\euv}.$
}
\noindent
In particular, every a.b.c\ liftable section $\tl s_t^a:\tl\Pi^a(k_t)\to \tl\Pi^a(K_t)$ 
of the canonical (surjective) projection $\tl p_t^a:\tl\Pi^a(K_t)\srjr \tl\Pi^a(k_t)$
gives rise to an injective map 
\[
\varphi:\Val_k(\tl k_t)\to\Val_k(\tl K_t)\vrg \quad \tl\euv\mapsto\tlw_{\tl\euv}\vrg
\]
such that the $k$-valuations $\tl\euv$ and $\tlw_{\tl\euv}$ satisfy the
conditions {\rm1), 2)} above.
}
\end{fact} 
\begin{proof} 
Beweis, klar!
\end{proof}
\subsection{\bf Places via {\it $\,\tl k |k$-$t\hb1$-\ha2}a.b.c.\ha3liftable sections}
$\ha0$
\vskip2pt
\noindent
If not otherwise explicitly stated, through out this subsection, 
the notation is that from Theorem~\ref{MThm}, that is: $T=X$,$\,K$
and $\tl s_T^a:\tl\Pi^a_1(k)\to \tl\Pi_1^a(T)$ is a $\,\tl k|k$-$t$-a.b.c.\ 
liftable section of the canonical projection 
$\tl p^a_T:\tl\Pi_1^a(T)\to \tl\Pi_1^a(k)$, and further, for $\star=a$,$\,c$,
we denote by $\tl s^\star_t:\tl\Pi_1^\star(k_t)\to \tl\Pi_1^\star(K_t)$ 
the $\tl k|k$-$t$-a.b.c.\ liftings of $\tl s_T^a$ to sections of 
$\tl p^\star_{K_t}:\tl\Pi^\star(K_t)\to \tl\Pi_1^\star(k_t)$ such that
$\tl s^c_t$ is a lifting of $\tl s^a_t.$  In particular, recalling the notations and 
the commutative diagrams introduced before Theorem~\ref{MThm}, 
one has the following:

\[
\begin{matrix}
&&\tl\Pi_1^c(K_t)&\horsrjr{\tl p^c_{t}}&\tl\Pi_1^c(k_t)&\horsrjr{\tl q^c_{k}}&\tl\Pi^c(k)&
         \ha{20}  \tl\Pi_1^c(K_t)&\horsrjl{\tl s_t^c}&\tl\Pi_1^c(k_t) \cr
&\ha{20}&\Ddwn{\tl{pr}_{K_t}}&&\Ddwn{\tl{pr}_{k_t}}&&\Ddwn{\tl{pr}_k}&
                  \ha{20}       \Ddwn{pr_{K_t}}&&\Ddwn{pr_{k_t}} \cr
&&\tl\Pi_1^a(K_t)&\horsrjr{\tl p^a_{t}}&\tl\Pi_1^a(k_t)&\horsrjr{\tl q^a_{k}}&\tl\Pi^a(k)&
         \ha{20}  \tl\Pi_1^a(K_t)&\horsrjl{\tl s_t^a}&\tl\Pi_1^a(k_t) \cr
\end{matrix}
\]
%
%
%
\begin{notation/remark} 
Let $v_{\infty}$, $v_{K,\infty}$ be
the ${1\over t}$-\ha1adic valuations of $k_t\vrg\,K_t$, thus 
$v_{\infty}=v_{K,\infty}|_{k_t}.$ For $\euv\in\Val_k(k_t)$
denote $d_{\vy}\!:=[k_t\vy :k].$ Obviously, one has:
\vskip5pt
\centerline{$d_{v_\infty}=1$ and $d_{v_p}=[k_t v_p:k]=\deg(p)$ 
is the degree of $p\in k[t].$} 
\vskip5pt
\noindent
The same holds, correspondingly for 
$w\in\Val_k(K_t)$, $\tl\euv\in\Val_{\tl k}(\tl k_t)$, $\tl w\in\Val_{\tlK}(\tlK_t)$, etc. 
\vskip2pt
\noindent
Further, given $p\in \tlk[t]$ monic irreducible, 
and $\alpha_p\in \tlk^a_t$ with $\alpha^\ell_p=p$, 
the following hold:
\vskip2pt 
\itm{25}{
\item[a)] $v_p$ is ramified in $\tlk_t(\alpha_p)\,\big|\,\tlk_t$, and $\tl\vy\in\Val_k(\tlk_t)$, 
$\tl\vy\neq v_p,v_\infty$, are unramified in $\tlk_t(\alpha_p)\,|\,\tlk_t$. 
\vskip2pt
\item[b)] $v_\infty$ is ramified in $\tlk_t(\alpha_p)\,\big|\,\tlk_t$ iff $\ell\!\nmid\!d_p$.
}
\vskip2pt
\noindent
We also notice that for $v\in\Val(\tlK_t)$ the following are equivalent:
\vskip2pt
\itm{25}{
\item[(i)] There is $\,p\in \tlk[t]$ monic irreducible such that 
$\,v=v_{\tlK,p}.$ 
\vskip2pt
\item[$({\rm ii})\hb4$] $v_\tlK\!:=v|_\tlK$ is trivial, $v|_{\tlk_t}$ is non-trivial, and 
$v\neq v_{\tlK,\infty}$.
}
\end{notation/remark}
\noindent
[{\it Proof of the last assertion.\/} For reader's sake we present the quite obvious proof. 
First, the direct implication is clear, because $\tlk=\oli k\cap \tlK\,$
implies: $p\in \tlk[t]$ is irreducible iff $p$ is irreducible over $\tlK$.
For the converse implication proceed as follows: Since $v|_K$ is trivial, 
$v$ is a $\tlK$-valuation of $\tlK_t=\tlK(t)$, and $v\neq v_{\tlK,\infty}$ 
implies that $v=v_{\tlK,q}$ is the $q$-adic $\tlK$-valuation for some 
monic irreducible polynomial $q\in \tlK[t]$. Since $v|_{\tlk_t}$ is nontrivial, 
there exists a unique monic irreducible $p\in \tlk[t]$ such that 
$v|_{\tlk_t}=v_{\tlk,p}$, and since $\tlk=\oli k\cap \tlK\,$ by hypothesis, 
one has that $p\in \tlK[t]$ is irreducible. On the other hand, since 
$v_{\tlk,p}=v_{\tlK,q}|_{\tlk_t}$, one must have $v_{\tlK,q}(p)=v_{\tlk,p}(p)>0$. 
Hence $q|p$ in $\tlK[t]$, thus $p=q$ (because both $p,q$ are irreducible 
monic).] 
\vskip7pt
Next, let $\Sigma\!:=\Val_k(\tlk_t)$, 
$\Sigma'\!:=\{v_\infty\}\cup\{v_p\in\Val_k(\tlk_t)\,\big|\,(\ell\vrg d_{v_p})=1\}$ 
and $\Sigma''\!:=\Sigma\backslash\Sigma'$, thus 
$\Sigma=\Sigma'\lower-2pt\hbox{$\scriptstyle\coprod$}\Sigma''.$ 
Define $\,\Sigma'_K,\Sigma''_K\subset\Val_K(\tlK_t)$ correspondingly, etc.
Notice that since $K|k$ is a regular field extension, every monic 
irreducible polynomial $p\in \tlk[t]$ is monic irreducible in $\tlK[t]$. 
Hence if $v_{\tlK,p}$ is the prolongation of $v_p\in\Val_k(\tlk_t)$ 
to $\tlK_t$, then one has:
\vskip5pt
\centerline{$\ha{20}(*)\ha{80}$ $d_{v_p}=[\tlk_t v_p:\tlk]=\deg(p)=[\tlK_t v_{\tlK,p}:\tlK]=d_{v_{\tlK,p}}$,$\ha{110}$}
\vskip5pt
\noindent
implying that $\Sigma'\subset\Sigma'_K$ and $\Sigma''\subset\Sigma''_K$.
Further, let $p\in \tlK[t]\backslash \tlk[t]$ be monic irreducible. Then $v_{\tlK,p}$ 
is trivial on $\tlk_t$, implying that $\Sigma'=\Val_k(\tlk_t)\cap\Sigma'_K$ and 
$\Sigma''=\Val_k(\tlk_t)\cap\Sigma''_K$.
\vskip5pt
\noindent

Recall the exact sequence 
$1\to \tlk_t^\times\hor{\bm\imath}\tlk^\times\oplus_{\tl\euv\in\Sigma}\tl\euv\ha1\tlk_t
\hor{\scalebox{.8}[.7]{deg}\,}\lvZ\to0$ with $\bm\imath(f)=a_f\oplus_{\tl\euv} \tl\euv(f)$,
where $a_f$ is the leading coefficient of $f$ and $\deg=\sum_{\tl\euv} d_{\tl\euv}.$ Hence
tensoring with $\lvZ/\ell$, on gets an exact the exact sequence 
$1\to \tlk_t^\times\hb3/\ell\to\tlk^\times\hb4/\ell\oplus_{\tl\euv\in\Sigma}\tl\euv\ha1\tlk_t/\ell
\to\lvZ/\ell\to0.$ Using the latter exact sequence, by Hilbert decomposition 
theory and Kummer theory one has the following well knowing fact (and therefore,
we do not repeat here the well known argument):
%
%
%
\begin{fact}
\label{fact3}
{\it In the above notation, the following hold:
\vskip3pt
\noindent \ \
{\rm (I)} In the above notation, setting $\,k^0_t\!:=\tlk^a\tlk_t$, 
$\,k'_t\!:=\tlk_t(\alpha_{v_p})_{v_p\in\Sigma'}$, 
$\,k''_t\!:=\tlk_t(\alpha_{v_p})_{v_p\in\Sigma''}$,
one has:
\vskip2pt
\itm{25}{
\item[{\rm1)}] The fields $k^0_t,\,  k'_t,\,  k''_t$ are linearly disjoint
over $\,\tlk_t$, and $\,\tlk^a_t|\tlk_t$ is the compositum 
$\,\tlk^a_t=k^0_tk'_tk''_t$.
\vskip2pt
\item[] Hence the Galois groups $\,G^0\!=G(k^0_t|\tl k_t)=\GGa{\tl k}$,
$\,G'\!=G(k'_t|\tl k_t)\,$ and $\,G''\!=G(k''_t|\tl k_t)\,$ satisfy:
\vskip4pt
\centerline{The canonical projection
$\,\GGa{k_t}\to G^0\times G'\times G''$ 
is an isomorphism.}
\vskip4pt
\item[{\rm2)}] Concerning generation of $\,G'$ and $\,G''$ one has:
\vskip2pt
\itm{5}{
\item[{\rm a)}] Given a fix generator $\,\tau_\infty\!\in\! T^a_{v_\infty}$ 
there are unique inertia generators $\,(\hb1\tau_\vy\!\in\! T^a_\vy)_{\vy\in\Sigma'}$
which topologically generate $G'$ and satisfy the unique prorelation
$\,\prod_{\vy\in \Sigma'}\tau_\vy=1$.
\vskip2pt
\item[{\rm b)}] $G''$ is profinite-freely generated by any system of inertia 
generators $\,(\tau_\vy\in T^a_\vy)_{\vy\in\Sigma''}$.
}}
\vskip2pt
\itm{10}{
\item[{\rm(II)}] The same holds, correspondingly, for $\,K_t$, and the 
sets of $\,K$-valuations $\,\Sigma'_K,\Sigma''_K\subset\Val_K(K_t)$. 
Further since $\Sigma'=\Val_k(k_t)\cap\Sigma'_K$ and 
$\Sigma''=\Val_k(k_t)\cap\Sigma''_K$, and one has:
\vskip2pt
\centerline{$k^0\subset K^0$, \ $k'_t\subset K'_t$, \ $k''_t\subset K''_t$, \
and \ $k^0=K^0\cap k^a_t$, \ $k'_t=K'_t\cap k^a_t$, \ $k''_t=K''_t\cap k^a_t$.}
}
}
\end{fact}
%
%
\begin{notation} In the above notation and that from Fact~\ref{fkt-finally-bis},  
for $\tl\euv\in\Val_k(\tlk_t)$ consider the resulting canonical
$\tl\euv$-valuation $\tl w_\tlK:=\tlwv|_{\tlK}\in\Val_k(\tlK)$ and its 
restriction $w_K=\tl w_\tlK|_K=\tlwv|_K.$ Further, denote
by $Z_{\tl w_\tlK|w_K}\subset\tl\Pi^a(T)$ the decomposition 
group of $\tl w_\tlK|w_K$ in $\tl\Pi^a(T).$
\end{notation}
%
%
%
\begin{lemma}
\label{non-trivial} 
In the above notation, there is $\tl\euv\in\Sigma'\!$ such that 
the corresponding $\tl w_\tlK|w_K$ satisfy:
\itm{25}{
\item[{\rm1)}] The $k$-valuations $\tl w_\tlK|w_K$ are non-trivial on $K$. 
In particular, $Kw_K\,|\,k$ is finite.
\vskip2pt
\item[{\rm2)}] One has that $\im(\tl s^a_T)\cap Z_{\tl w_\tlK|w_K}$ is 
an open subgroup in $\,\im(\tl s^a_T)$, hence infinite. 
}
\end{lemma}
\begin{proof} To 1): Obviously, $w_K$ is trivial 
iff $\tl w_K$ is trivial. {\it By contradiction,\/} suppose that the 
assertion of Claim~1 does not hold, that is, for every $\tl\euv\in\Sigma'$ 
the resulting $\tlwv$ is trivial on $K$. If so, then the map 
$\varphi:\Sigma'\to\Val_k(\tl K_t)$, $\tl\euv\mapsto\tlw_{\tl\euv}$ defined 
in Fact~\ref{fkt-finally-bis} has image $\varphi(\Sigma')\subset\Val_K(K_t)$
such that, by Fact~\ref{fkt-finally-bis},~1), $\tlwv|_{\tl k_t}=\tl\euv$ for all 
$\tl\euv\in\Val_k(\tlk_t)$. Therefore, if $\tl\euv=v_\infty$, then 
$\tlwv=v_{\tl K,\infty}$, and if $\tl\vy=v_p$ with $p\in \tl k[t]$ 
monic irreducible, then $\tlwv=v_{\tl K,p}$ on 
$\tl K_t$. Thus by~$(*)$ above, 
$[\tlk_t \tl\vy:\tlk]=d_p=[\tlK_t v_{\tlK,p}:\tlK]$,  
concluding that $\varphi(\Sigma')\subset\Sigma'_K.$ Hence the map 
below is a bijection
\[
\varphi:\Sigma'\to\varphi(\Sigma')\subset\Sigma'_K, \ \
   \tl\euv\mapsto\tlwv \ \ \hbox{with $\varphi(\Sigma')\subset\Sigma'_K$ stictly,}
\]
which via $\tl s_t^a\!:\tl\Pi^a(k_t)\to \tl\Pi^a(K_t)$ 
and $\tl p^a_{K_t}\!:\tl\Pi^a(K_t)\to \tl\Pi^a(k_t)$
is compatible with decomposition and inertia groups, i.e., 
if $\tl\euv\leftrightarrow \tlwv\vrg$ then
$\tl s^a_t(T^a_{\tl\euv})\!=\!T^a_{\tlwv}\vrg$ 
$\tl s^a_t(Z^a_{\tl\euv})\!\subset\! Z^a_{\tlwv}\vrg$
$\tl p^a_{K_t}(Z^a_{\tlwv})\!=\!Z^a_{\tl\euv}\vrg$ and the residue fields
satisfy $\tl K_t\tlwv=K\,\tlk_t \tl\vy$. \ Let $\,\tau_{\infty}\in T_{v_\infty}$ 
be a fixed inertia generator, hence $\,\tau_{{\tlK,\infty}}\!:=
\tl s^a_t(\tau_{\infty})\in T_{v_{\tlK,\infty}}$ generates $T_{v_{\tlK,\infty}}$
and $\tl p^a_{K_t}(\tau_{{\tlK,\infty}})=\tau_{\infty}$.
Further, let $(\tau_{\tl\euv}\in T^a_{\tl\euv})_{\tl\euv\in\Sigma'}$ 
and $(\tau_{\tl\euv}\in T^a_{\tl\euv})_{\tl\euv\in\Sigma''}$ and 
$(\tau_{\tlw}\in T^a_{\tlw})_{\tlw\in\Sigma'_K}$, 
$(\tau_{\tlw}\in T^a_{\tlw})_{\tlw\in\Sigma''_K}$ be systems 
of inertia generators as in Fact~\ref{fact3},~2) with
$\tau_{v_\infty}=\tau_\infty$, $\tau_{v_{K,\infty}}=\tau_{K,\infty}$.
\vskip2pt
{\bf Conclude} that $(\tau_{\tlwv})_{\tlwv\in\varphi(\Sigma')}
=\big(\tl s^a_t(\tau_{\tlv})\big)_{\tlv\in\Sigma'}$ is a proper subsystem of 
$(\tau_{\tlw})_{\tlw\in\Sigma'_K}$ such that 
\[
\sclx99{$\prod$}_{\tlwv\in\varphi(\Sigma')}\tau_{\tlwv}=
     \sclx99{$\prod$}_{\tl\euv\in\Sigma'}\tl s^a_t(\tau_{\tl\euv})=
\tl s^a_t(\sclx99{$\prod$}_{\tl\euv\in\Sigma'}\tau_{\tl\euv})=\tl s^a_t(1)=1.
\] 
Hence we reached a contradiction, and assertion~1) is proved.
\vskip4pt
To 2): Recall the inclusion $\varphi:\Val_k(\tl k_t)\hra\Val_k(\tl K_t)$, 
$\tl\euv\mapsto\tlwv$ from Fact~\ref{fkt-finally-bis}, and that by loc.cit.,~2), 
one has $\tl p^a(Z^a_{\tlwv})=Z^a_{\tl\euv}\subset\tl\Pi^a(\tlk_t)$ 
and $\tl p^a(Z_{\tlwv^a|v_\euv})=Z_{\tl\euv^a|\euv}\subset \tl\Pi(k_t).$ In terms
of decomposition fields, that implies the equality $k_{\tlw_\tlK}\!:=\tlK_t\tlwv\cap\tlk^a
=\tlk_t\tl\vy\cap \tlk^a=:\!k_{\tl\euv}$ as finite extension of $\tlk$. 
To proceed, recall the commutative diagrams:
\[
\begin{matrix}
\\[-9mm]
\tl\Pi^a(K_t)&\horsrjr{\tl p^a_{t}}&\tl\Pi^a(k_t)&\scalebox{.01}[3]{|} 
       &\tl\Pi^a(K_t)&\horsrjl{{\ \tl s^a_t}}&\tl\Pi^a(k_t)\cr
   \\[-5mm]
\dwn{\tl q^a_T}&&\dwn{\tl q^a_k} &\ha{30} & 
             \dwn{\tl q^a_T}&&\dwn{\tl q^a_k}\cr
   \\[-3mm]
\tl\Pi^a(T)&\horsrjr{\tl p^a_T}&\tl\Pi^a(k)& &\tl\Pi^a(T)&\horsrjl{\ \tl s^a_T}&\tl\Pi^a(k)\cr
\end{matrix}
\]

Since $\tl w_\tlK=\tlwv|_\tlK$ and $w_K=\tlwv|_K=v_\euv|_K$, by basics of Hilbert 
decomposition theory, one has $\tl q^a_T(Z_{\tlwv^a|v_\euv})\subset Z_{\tl w_\tlK|w_K}.$
Hence since $\tl q^a_T\circ\tl s^a_t=\tl s^a_T\circ q^a_{k}$, and taking into account 
that $\tl q^a_{k}(Z_{\tl\euv^a|\euv})=G(\tlk^a|k_{\tl\euv})$, one has finally 
commutative sub-diagrams of the above ones:
\[
\begin{matrix}
  \\[-4.25mm]
Z_{\tlwv^a|v_\euv}&\horsrjr{\tl p^a_{t}}&Z_{\tl\euv^a|\euv}&
    &Z^a_{\tlwv^a|v_\euv}&\horsrjl{\ \tl s^a_{t}}&Z_{\tl\euv^a|\euv}\cr
  \\[-3mm]  
\dwn{\tl q^a_T}&&\dwn{\tl q^a_k} &\ha{30} & 
                        \dwn{\tl q^a_T}&&\dwn{\tl q^a_k}\cr
  \\[-3mm]
Z_{\tl w_\tlK|w_K}&\horsrjr{\tl p^a_T}&G(\tlk^a|k_{\tl\euv})& 
   &Z_{\tl w_\tlK|w_K}&\horsrjl{\ \tl s^a_T}&G(\tlk^a|k_{\tl\euv})\cr
\end{matrix}
\]
that is, $\,\tl s^a_T\pml G(\tlk^a|k_{\tl\euv})\pmr\subset Z_{\tl w_\tlK|w_K}$.
In particular, since $\tlk^a| k$ satisfies Hypothesis~{\sf(H)}, thus
$\tlk^a| k$ has infinite degree, and $k_{\tl\euv}|k$ has finite degree
by the discussion above, one has the following:
\[
\im(\tl s^a_T)\cap Z_{\tl w_\tlK|w_K}\supset 
\tl s^a_T\pml G(\tlk^a|k_{\tl\euv})\cap Z_{\tl w_\tlK|w_K} \ 
\hbox{ are open subgroups of } \im(\tl s^a_T)\,.
\leqno{\indent(*)_{\tl s^a_K}}
\]
This completes the proof of assertion~2). 
\end{proof}
\noindent
Next we notice that in the case $T=K$, one can sharpen the above 
Lemma~\ref{non-trivial} as follows.
%
%
\begin{keylemma}
\label{key-lemma-1} 
In context from Lemma~\ref{non-trivial}, let
$T=K.$  Then $\tl w_\tlK|w_K$ from loc.cit.\ sartisfy:
\vskip2pt
\itm{25}{
\item[{\rm 1)}] $\tl w_\tlK\in\Val_k(K)$ depends on $\tl s^a_K$ only and 
not on the specific $\tl\euv\in\Val_k(\tlk_t)$ used to define it.
\vskip2pt
\item[{\rm2)}] $\tl K\tl w_\tlK\ha1|\ha1k$ is algebraic, $Kw_K\cap\tl k=k_t\euv\cap\tl k$, 
and $Kw_K|k$ and $\tl k|k$ are linearly disjoint over $k.$ 
}
\end{keylemma}
\begin{proof} We proceed along the following steps.
\vskip5pt
\noindent
{\bf Claim 1}. The non-trivial valuation $\,w_K\!:=\tlwv|_K$ from 
Lemma~\ref{non-trivial} does not dependent of $\tl\euv$.
\vskip5pt
\noindent

Indeed, if $\tl\euv'\in\Val_k(\tlk_t)$ and the resulting 
$\tl w'_\tlK\!:=\tlw'_{\tl\euv'}\in\Val_k(\tl K_t)$ and/or $w'_K\!:=\tlw'_{\tl\euv'}|_K$ 
is non-trivial, then the corresponding $k_{\tl\euv'}|k$ is finite. 
Hence $G(\tlk_{\tl\euv'}|k)\subset G(\tlk|k)$ is open and 
$\tl s^a_K\pml G(\tlk^a|k_{\tl\euv'})\pmr=\im(s^a_K)\cap Z_{\tl w'_\tlK|w'_K}.$
Therefore, $G_{\tl\euv,\tl\euv'}\!:=G(\tlk|k_{\tl\euv})\cap 
G(\tlk_{\tl\euv'}|k)\subset G(\tlk|k)$ is an open subgroup as 
well, and we conclude: $\tl s^a_K(G_{\tl\euv,\tl\euv'})\subset\im(\tl s^a_K)$
is open, thus infinite, implying:
\[
1\neq \tl s^a_K(G_{\tl\euv,\tl\euv'})\subset \tl s^a_K\pml G(\tlk^a|k_{\tl\euv})
\pmr\cap \tl s^a_K\pml G(\tlk^a|k_{\tl\euv'})\pmr 
\subset Z_{\tl w_\tlK|w_K}\cap Z_{\tl w'_\tlK|w'_K}.
\] 

Finally, the $k$-valuations 
of $\tlK$ are discrete\ha3---\ha1by the fact that $K=k(X)$ is the 
function field of a $k$-curve $X$, and therefore, non-equivalent 
$k$-valuation of $K$ are independent. Hence we must have 
$w_K=w'_K$ by Lemma~\ref{indepval}, thus concluding the 
proof of Claim 1.
\vskip5pt
{\bf Claim 2.} $Kw_K|k$ and $\tlk|k$ are linearly disjoint over $k$.
\vskip5pt\noindent
Indeed, identify $G(\tlK|K)=:\!G\!:=G(\tlk|k)$ under 
$G(\tlK|K)\hrr{\tl\imath}G(\tlk|k)$, and recall that \hbox{$k_{\tl\euv}|k\hra \tlk|k$} 
is a finite subextension, where $k_{\tl\euv}\!:=Kw_K\cap\tlk$. 
By Hilbert decomposition theory, $G$ acts transitively on the set 
$\clV_{w_K}$ of prolongation $\tl w'_K|w_K$ of $w_K$ to $\tlK| K$ by 
$\tl w_K^\sigma=\tl w_K\circ\sigma^{-1}\!$, and $Z_{\tl w_K|w_K}$ is the 
stabilizer of~$\tl w_K$. Setting $\tl w'_K\!:=\tl w_K^\sigma\!$ and 
$\tl\euv'\!:=\tl\euv^\sigma\!\!$, one has: First, if $\sigma^a\mapsto\sigma$ 
under $\tl\Pi^a(k)\to G$, then $\sigma(k_{\tl\euv})=k_{\tl\euv'}$, thus 
$G(\tlk|k_{\tl\euv'})=G(\tlk|k_{\tl\euv})^{\sigma^a}\subset G(\tlk|k)$
is an open subgroup of $G(\tlk|k).$ 
Second, if $\sigma^a_K\mapsto\sigma$ under \hbox{$G(\tlK^a|k)\to G$}, 
then $Z_{\tl w'^a_K|w_K}=Z_{\tl w^a_K|w_K}^{\sigma_K^a}$ inside 
$\tl\Pi^a(K)$. In particular, choosing $\sigma^a_K\!:=\tl s^a_K(\sigma^a)$, 
thus $\sigma^a_K\in\im(\tl s^a_K)$, the following hold:
\vskip2pt
 a)  $\im(\tl s^a_K)\subset \tl\Pi^a(K)$ is invariant under the 
 $\sigma^a$-conjugation.
\vskip2pt
b) $\tl s^a_K\pml G(\tlk^a|k_{\tl\euv'})\pmr\subset\im(\tl s^a_K)$ is open, 
and so is $\tl s^a_K\pml G(\tlk^a|k_{\tl\euv})\pmr\subset\im(\tl s^a_K).$ 
\vskip1pt
c) $\tl s^a_K\pml G(\tlk^a|k_{\tl\euv'})\pmr=\tl s^a_K\pml G(\tlk^a|k_{\tl\euv})\pmr^{\sigma^a}
\subset Z_{\tl w^a_K|w_K}^{\sigma^a}=Z_{\tl w'^a_K|\euw}$, because
$\tl s^a_K\pml G(\tlk^a|k_{\tl\euv})\pmr\subset Z_{\tl w^a_K|w_K}.$
\vskip2pt
\noindent
{\bf Conclude:}
$G_{\tl\euv,\tl\euv'}\!:=\tl s^a_K\pml G(\tlk^a|k_{\tl\euv'})\pmr
\cap \tl s^a_K\pml G(\tlk^a|k_{\tl\euv})\pmr\subset \im(\tl s^a_K)$
is open in $\im(\tl s^a_K)$, hence infinite, and 
$
1\neq G_{\tl\euv,\tl\euv'}\subset \tl s^a_K\pml G(\tlk^a|k_{\tl\euv'})\pmr
\cap \tl s^a_K\pml G(\tlk^a|k_{\tl\euv})\pmr\subset
Z_{\tl w'^a_K|w_K}\cap Z_{\tl w^a_K|w_K}\,.
$
Hence arguing as in the proof of Claim 1, one gets $\tl w_K=\tl w'_K.$ 
Equivalently, $\sigma\in Z_{\tl w_K|w_K}$, thus finally, implying 
that $Kw_K\cap\tlk=k$, as claimed. This concludes the proof of 
Key Lemma~\ref{key-lemma-1}.
\end{proof}
\subsection{\bf Concluding the proof of Theorem~\ref{MThm}}$\ha0$
\vskip2pt
In the context/notation of Theorem~\ref{MThm}, we notice that
assertion~2) follows from assertion~1) by mere definitions,
and assertion~3) follows from Lemma~\ref{non-trivial},~2).
Thus it is left to prove assertion~1). 
\vskip2pt
This being said, let $\tl s^a:\tl\Pi^a(k)\to \tl\Pi^a(K)$ be a 
given $\bm t$-a.b.c.\ birationally liftable section of the 
canonical projection $\tl p^a:\tl\Pi^a(K)\to \tl\Pi^a(k)$. Then by~Key 
Lemma~\ref{key-lemma-1}, there is a unique non-trivial
$\tl w_\tlK\in\Val_k(K)$ which together with its 
restriction $w_K\!:=(\tl w_\tlK)|_K\in\Val_k(K)$ satisfy:
\vskip5pt
$(*)$ \  $Kw_K|k$ and $\tl k|k$ are linearly disjoint over $k.$
\vskip5pt
Since $K=k(X)$ with $X$ a complete normal $k$-curve and 
$w_K\in\Val_k(K)$, it follows that $w_K$ has a center 
$x_{w_K}\in X$ such that $\clO_{w_K}=\clO_{x_{w_K}}$ 
and $\eum_{w_K}=\eum_{x_{w_K}}\!$, and similarly, 
$\tl w_\tlK$ has a center $x_{\tl w_\tlK}\in\tl X=X_{\tlk}$ such that 
$x_{\tl w_\tlK}\mapsto x_{w_K}$ under the canonical projection
$\tl X\to X$. In particular, $\kappa_{x_{w_K}}=K w_K$, 
and therefore, by $(*)$ above, it follows that $\kappa_{x_{\tl w_\tlK}}|k$ 
and $\tlk|k$ are linearly disjoint over $k$. And $\tl w_\tlK|w_K$ 
are defined by the points $x_{\tl w_\tlK}| x_{w_K}$ as 
required in Theorem~\ref{MThm}. Finally, by the uniqueness 
part of the Key Lemma~\ref{key-lemma-1}, one has that
the points 
$x_{\tl w_\tlK}| x_{w_K}$ are unique with the property 
that $\im(\tl s^a)\cap Z_{x_{\tl w_\tlK}|x_{w_K}}\neq1.$
This concludes the proof of Theorem~\ref{MThm}.
\section{Final comments/Open questions}
Naturally, the ``elephant in the room'' is the question 
whether the BSC holds in the geometric case, i.e., for 
geometrically integral normal $k$-curves $X$, where 
$k=k_0(Z)$ is a non-constant function field with $k_0$ 
not $\ell$-closed for some $\ell\neq\chr(k)$. This being 
said, here is a short list of questions which might be 
addressed with methods similar to the ones developed 
in this manuscript (and some {\it new\/} ideas\ha1), 
notations being as introduced in the Introduction.
\vskip7pt 
\noindent \ \
0) Prove all the above results for $\ell=2$, provided $\chr\neq2$
(after replacing $\mu_\ell$ by $\mu_4$).
\vskip5pt
\noindent \ \
1) Suppose that $k|\lvQ$ is f.g.\ and $\oli{\lvQ}\subset\tl k.$ Prove 
the $\tl k|k$-$\bm t$-\ha1minimalistic BSC.
\vskip5pt
\noindent \ \
2) Suppose that $\mu_{2\ell}\subset \tl k$ and $\tl k^\times\hb3/\ell$ infinite. 
Does the $\tl k|k$-$\bm t$-\ha1minimalistic BSC hold\ha1? 
\vskip5pt
\noindent \ \
3) \scalebox{.99}[1]{Replacing $\lvP^1_t$ (in the $\bm t$-BSC) by a 
$k$-curve or a $k$-variety $Z$, formulate \& prove the $Z$-BSC.}
\vskip5pt
\noindent \ \
4) For $X$ a proper smooth $k$-variety, formulate \& prove
the $t$-BSC over $X$ (maybe the $Z$-BSC).
\vskip5pt
\noindent \ \
5)  Let $k$ be as above. Does the (minimalistic) BSC hold for 
$K_t|k_t$, e.g., for $k=k_0(u)$, $k_0=\oli k_0\,$?
\vskip3pt \ \ 
\bltt \ This would sharpen {\sc Bogomolov--Rovinsky--Tschinkel}~\cite{BRT} over 
$k\!:=k_0(\hb1t,\hb1u)$, $k_0=\oli k_0$.


\begin{thebibliography}{XXXX}

     


    



\bibitem[BRT]{BRT} F.A.\ha2Bogomolov, M.\ha2Rovinsky,\ha2 Y.\ha2Tschinkel, 
{\it Homomorphisms of multiplicative groups of fields preserving algebraic dependence,\/}
European\ha2J.\ha2Math, {\bf9} (2019), 656--685.


\bibitem[Be]{Be} Bresciani, G., {\it On the birational section 
    conjecture with strong birationality assumptions,\/} Invent.\ha3Math.
    {\bf235} (2024), 129–150.

\bibitem[B-V]{B-V} Bresciani, G.\ and Vistoli, A.,
      \textit{An elementary approach to Stix’s proof of the real 
      section conjecture}, (2020). See  {\sf arXiv:2012.06278 [math.AG], 18 Dec 2020}.



%
%
%

\bibitem[BOU]{BOU} Bourbaki, Alg\`ebre commutative, Hermann Paris 1964.











\bibitem[Fa]{Fa} Faltings, G., {\it Curves and their fundamental 
       groups {\rm (following Grothendieck, Tamagawa and Mochizuki),}}
         Ast\'erisque, Vol {\bf252} (1998) Expos\'e 840.

\bibitem[F-J]{F-J} Fried, M.\ and Jarden, M., Field Arithmetic
    (third revised edition). Ergebnisse der Mathematik und 
    ihrer Grenzgebiete, 3. Folge; Springer Verlag, ISSN: 0071-1136.

\bibitem[GGA]{GGA} Geometric Galois Actions~I, LMS LNS Vol {\bf 242},
     eds L.~Schneps -- P.~Lochak, Cambridge Univ.\ Press 1998.

\bibitem[G1]{G1} Grothendieck, A., {\it Letter to Faltings, June 1983},
    See [GGA].

\bibitem[G2]{G2} Grothendieck, A., {\it Esquisse d'un programme, 1984}.
    See [GGA].
    



\bibitem[H-Sz]{H-Sz} Harari, D. and Szamuely, T., {\it Galois sections 
for abelianized fundamental groups,\/} {\it Appendix\/} by E. V. Flynn, 
Math. Annalen {\bf344} (2009), 779--800.
       




   
       


\bibitem[Ko1]{Ko1} Koenigsmann, J.,
   {\it On the `section conjecture' in anabelian geometry,}
   J.\ reine angew.\ Math. {\bf588} (2005), 221--235.

\bibitem[Ko2]{Ko2} Koenigsmann, J., {\it Solvable absolute Galois
   groups are metabelian,\/} Inventiones Math.\ {\bf 144} (2001), 1--22.

\bibitem[KPR]{KPR}  Kuhlmann, F.-V., Pank, M., Roquette, P., 
   {\it Immediate and purely wild extensions of valued fields,}
   Manuscripta Math. {\bf55} (1986), 39--67.   
   



     
       

\bibitem[Lu]{Lu} L\"udtke, M., {\it The $p$-adic section conjecture 
 for localisations of curves,\/} Dissertation, 2020. See 
 {\sf urn:nbn:de:hebis:30:3-574318}.
   

 \bibitem[Mz1]{Mz1}
      Mochizuki, Sh., \textit{Topics surrounding the anabelian geometry of hyperbolic curves}, in: Galois groups and fundamental groups, Math. Sci. Res. Inst. Publ. \textbf{41} (1990), 120--140.

\bibitem[Mz2]{Mz2} Mochizuki, Sh., {\it The local pro-$p$ Grothendieck 
conjecture for hyperbolic curves,\/} Invent. Math.\ha3{\bf138} (1999), 319--423.

\bibitem[Mz3]{Mz3} Mochizuki, Sh., \textit{Topics surrounding the anabelian 
    geometry of hyperbolic curves}, Math.Sci.Res.Inst.Publ. {\bf41} (2003), 119--165.

\bibitem[Mz4]{Mz4} Mocizuki, Sh., {\it Absolute anabelian cuspidalizations 
of proper hyperbolic curves.,\/} J. Math. Kyoto Univ. {\bf47} (2007), 451--539.

\bibitem[Mu]{Mu} Mumford, D., The red book of varieties and schemes,
     LNM 1358, 2nd edition, Springer Verlag 1999.
     
\bibitem[Na]{Na}  Nakamura, H., {\it Galois rigidity of the \'etale 
fundamental groups of punctured projective lines,\/} J. reine angew. 
Math. {\bf411} (1990) 205--216.    



%











\bibitem[P0]{P0} Pop, F., {\it Galoissche Kennzeichnung 
$p$-adisch abgeschlossener K\"orper,\/} J.\ha2reine\ha2angew.\ha2Math. 
{\bf392} (1988), 145--175. 

{\it $\lvZ/\ell$ abelian-by-central 
   Galois theory of prime divisors,\/} in: The Arithmetic of Fundamental 
   Groups: PIA 2010, ed. Jakob Stix; Springer-Verlag 2012; 225--244. 

\bibitem[P1]{P1} Pop, F., {\it $\lvZ/\ell$ abelian-by-central 
   Galois theory of prime divisors,\/} in: The Arithmetic of Fundamental 
   Groups: PIA 2010, ed. Jakob Stix; Springer-Verlag 2012; 225--244. 

\bibitem[P2]{P2} Pop, F.,
   {\it On the birational $p$-adic section conjecture},
   Compositio Math. {\bf146} (2010), 621--637.
   
\bibitem[P3]{P3} Pop, F.,  {\it Z/p metabelian birational $p$-adic 
section conjecture for varieties,\/} Compositio Math. {\bf153} (2017), 1433--1445.   

  






\bibitem[St0]{St0} Stix, J., {\it Rational points and arithmetic of 
fundamental groups. Evidence for the section conjecture,\/}
Vol.\ha2{\bf2054}, LNM Berlin, Springer, 2013.

\bibitem[St1]{St1} Stix, J., {\it Birational $p$-adic Galois sections in 
     higher dimensions,\/} Israel J.\ha3Math {\bf198} (2013), 49--61.

     
\bibitem[St2]{St2} Stix, J., \textit{On the birational section conjecture 
     with local conditions}, Invent. Math. \textbf{199} (2015), 239--265.     

\bibitem[Sz]{Sz} Szamuely, T., {\it Groupes de Galois de corps de type fini
   $($d'apr\`es Pop$\ha1)$,\/} Ast\'erisque {\bf 294} (2004), 403--431.


\bibitem[Ta]{Ta}  Tamagawa, A., {\it The Grothendieck conjecture for 
affine curves,\/} Compositio Math.\ha3{\bf109} (1997), 135--194. 

\bibitem[To1]{To1} Topaz, A., {\it Commuting-Liftable Subgroups 
   of Galois Groups II,\/} J. reine angew. Math.\ {\bf 730} (2017), 65--133. 





\bibitem[Wk]{Wk} Wickelgren, K., \textit{$2$-nilpotent real Section Conjecture}, 
     Math. Ann. {\bf 358} (2014), 361--387.

\end{thebibliography}
\end{document}